\documentclass[ejs]{imsart}

\usepackage{placeins}
\usepackage{xr-hyper}
\usepackage{mathrsfs}
\usepackage{bm}
\usepackage{nameref}
\usepackage{appendix}
\usepackage[utf8]{inputenc}
\usepackage{cite}
\usepackage{tikz}
\usepackage{amsfonts}
\usepackage{amsmath}
\usepackage{bbm}
\usepackage{dsfont}
\usepackage{amsthm}
\usepackage{amssymb}
\usepackage{relsize}
\usepackage{csquotes}
\usepackage{graphicx}
\usepackage{epstopdf}
\usepackage{xcolor}
\usepackage[ruled]{algorithm2e}
\usepackage[utf8]{inputenc}
\RequirePackage[colorlinks,citecolor=blue,urlcolor=blue,filecolor=blue]{hyperref}
\usepackage{verbatim}
\newcommand{\dd}{\hspace{0ex}\text{d}}
\newcommand{\R}{\mathbb{R}}
\newcommand{\uh}{(a,\mathbf{h})}

\newcommand{\dcut}{\dd\Pi_2(\textbf{f})\phi_n(Q|Y^{-\ell}_{1:n})\dd Q}
\newcommand{\diag}{\text{diag}}
\DeclareMathOperator*{\argmax}{arg\,max}
\DeclareMathOperator*{\argmin}{arg\,min}
\theoremstyle{plain}
\newtheorem{theorem}{Theorem}
\newtheorem{lemma}{Lemma}
\newtheorem{cor}{Corollary}
\newtheorem{prop}{Proposition}
\theoremstyle{definition}

\newtheorem{assume}{Assumption}

\newtheorem{defn}{Definition}
\newtheorem{rk}{Remark}

\makeatletter
\newcommand*{\addFileDependency}[1]{
  \typeout{(#1)}
  \@addtofilelist{#1}
  \IfFileExists{#1}{}{\typeout{No file #1.}}
}
\makeatother
\newcommand*{\myexternaldocument}[1]{%
    \externaldocument{#1}%
    \addFileDependency{#1.tex}%
    \addFileDependency{#1.aux}%
}

\myexternaldocument{supplement_ejs}

\begin{document}
\begin{frontmatter}
\title{Efficient Bayesian estimation and use of cut posterior in semiparametric hidden Markov models}
\runauthor{Moss \& Rousseau}
\runtitle{Efficiency and cut posterior in semiparametric HMMs}

\begin{aug}
\author[A]{\fnms{Daniel} \snm{Moss}\ead[label=e1]{daniel.moss@stats.ox.ac.uk}},
\author[A]{\fnms{Judith} \snm{Rousseau}\ead[label=e2]{judith.rousseau@stats.ox.ac.uk}}
\address[A]{Department of Statistics,
University of Oxford, United Kingdom.
\printead{e1,e2}}
\end{aug}

\begin{abstract}
We consider the problem of estimation in Hidden Markov models with finite state space and nonparametric emission distributions. Efficient estimators for the transition matrix are exhibited, and a semiparametric Bernstein-von Mises result is deduced. Following from this, we propose a modular approach using the cut posterior to jointly estimate the transition matrix and the emission densities. We first derive a general theorem on contraction rates for this approach. We then show how this result may be applied to obtain a contraction rate result for the emission densities in our setting; a key intermediate step is an inversion inequality relating $L^1$ distance between the marginal densities to $L^1$ distance between the emissions. Finally, a contraction result for the smoothing probabilities is shown, which avoids the common approach of sample splitting. Simulations are provided which demonstrate both the theory and the ease of its implementation.
\end{abstract}


\end{frontmatter}

\section{Introduction}\label{sec:intro}
Hidden Markov models (HMMs) are a broad and widely used class of statistical models, enjoying applications in a variety of settings where observed data is linked to some ordered process, for which an assumption of independently distributed data would be both inappropriate and uninformative. Specific applications include modelling of weather \cite{ailliot2009space,hughes1999non}, circadian rhythms \cite{huang2018hidden}, animal behaviour \cite{langrock2015nonparametric,10.1214/16-AOAS1008}, finance \cite{mamon2007hidden}, information retrieval \cite{teh2006hierarchical,miller1999hidden}, biomolecular dynamics \cite{horenko2008likelihood}, genomics \cite{yau2011bayesian} and speech recognition\cite{rabiner1989tutorial}.

\vspace{1ex}
In this paper, we consider inference in finite state space HMMs. Such models are characterised by an unobserved (latent) Markov chain $(X_t)_{t\geq 1}$ taking values in $[R]=\{1,\dots, R\}$ with $R<\infty$, evolving according to a transition matrix $Q$. Conditionally on $X_t=j$, $Y_t\sim F_j$ where $F_j$ is the \textit{emission distribution} with associated density $f_j$. These models generalise independent mixture models, which are obtained as a special case when the $X_t$ are independent and identically distributed. Here we assume that $R$ is known so that the parameters in such models are then $Q$ and $\mathbf F  = (F_1, \dots, F_R)$. 

\vspace{1ex}
Most of the work on HMMs considers parametric models, where the emissions are assumed to admit densities in a parametric class $\{f_\theta:\theta\in\Theta\}$ where $\Theta\subset\mathbb{R}^d$ for some $d<\infty$, see for instance \cite{douc2014nonlinear,fruhwirth2006finite,cappe2009inference}.  However such an assumption leads to inference which is strongly influenced by the choice of  the parametric family $\{f_\theta:\theta\in\Theta\}$. This problem has been often discussed in the literature, see for instance \cite{rabiner1989tutorial}, \cite{fox2007sticky} or \cite{yau2011bayesian}, especially but not solely in relation to clustering or segmentation. In the seminal paper \cite{gassiat2016inference}, the authors show identifiability under weak assumptions on the emissions, provided the transition matrix is of full rank, paving the way for estimation of semiparametric models. In \cite{alexandrovich2016nonparametric}, it is shown further that the number $R$ of hidden states may be identified. In \cite{hsu2012spectral,Anandkumar2014} and \cite{de2017consistent}, frequentist estimators of $Q$ and $\mathbf{f} = (f_1, \dots, f_R)$ respectively have been proposed using spectral methods, showing in particular that $Q$ can be estimated at the rate $1/\sqrt{n}$. Frequentist estimation of the emission densities has also been addressed using penalised least squares approaches \cite{de2016minimax} and spectral methods \cite{de2017consistent}. However, no results exist on the asymptotic distribution of frequentist estimators for $Q$, nor on efficient estimation for $Q$.

\vspace{1ex}
Although Bayesian nonparametric estimation methods have been considered in practice in Hidden Markov models, see for instance \cite{yau2011bayesian} or \cite{fox2007sticky}, little is known about their theoretical properties. While \cite{vernet2015conc,vernet2015posterior} established posterior consistency  under general conditions on the prior and refined the analysis to derive posterior concentration rates on the marginal density of $\ell$ successive observations, $g^\ell_{Q, \mathbf F}$, no results exist regarding the properties of Bayesian procedures when seeking to recover the parameters $Q$ and $\mathbf F$, or
other functionals of $Q,\mathbf F$ which are often of interests in HMMs. For instance, when interests lie in clustering or segmentation, the quantities of interest are the smoothing probabilities, being the conditional distribution of the latent states given the observations. In \cite{de2017consistent}, the authors obtain rates of convergence for frequentist estimators of the smoothing probabilities, but their result requires either a sup-norm convergence for the estimator $\hat{\mathbf f}$ or  splitting the data into two parts, with estimation of $\mathbf f$ based on one part of the data and estimation of the smoothing probabilities based on the other part of the observations. In this paper we intend to bridge this gap, concentrating on Bayesian semiparametric methods while also exhibiting non-Bayesian, semiparametric efficient estimators of $Q$. 

\vspace{1ex}

\vspace{1ex}
We first construct a family of priors $\Pi_1$ on $(Q,\mathbf{F})$ which we show in Theorem \ref{BvMM} leads to an asymptotically normal posterior distribution for $\sqrt{n}(Q-\hat{Q})$, of variance $V$ detailed therein. Here $\hat{Q}$ is a freqentist estimator, exhibited in Theorem \ref{MLEM}, for which $\sqrt{n}(\hat Q - Q^*)$ converges in distribution to $N(0,V)$, with $Q^*$ being the the frequentist true parameter under which the data is assumed to be generated. Consequently, Bayesian estimates associated to such posteriors such as the posterior mean, enjoy parametric convergence rates to $Q^*$, and importantly credible regions for $Q$ are asymptotically confidence intervals. We then refine this construction to obtain a scheme for which $V$ is the optimal variance for semiparametric estimation of $Q$, being the inverse efficient Fisher information, in Theorem \ref{BvM}. Semiparametric Bernstein-von Mises properties are highly non trivial results and correspondingly sparse in the literature, moreover in \cite{freedman1999wald,castillo2015bernstein,castillo2012semiparametric,rivoirard2012bernstein} a number of counterexamples are exhibited, where semiparametric Bernstein-von Mises results do not hold. Since this property is crucial to ensure that credible regions are also confidence regions and thus robust to the choice of the prior distribution, it is important to study and verify it.

\vspace{1ex}
Our approach to obtaining these results follows the ideas of \cite{Gassiat2018}, extending their work on mixtures to the more complex HMM setting. In particular, the construction of  $\hat Q $ (together with the prior on $Q, \mathbf F$) is based on a  parametric approximation  to the nonparametric emission model, with the property that estimation in this model with appropriately `coarsened' data leads to a well-specified parametric model for estimation of $Q$. Once we reduce to such a parametric model, we have asymptotic normality of the corresponding MLE as in \cite{bickel1998asymptotic} as well as Bernstein-von Mises for the posterior as in \cite{DeGunst2008}, although the intuition that our `coarsened' data is less informative translates to an inefficient asymptotic variance $V$. We then show that we can construct efficient estimators of $Q$ by choosing a  type of coarsening, similar to the approach of \cite{Gassiat2018}. The proof techniques in our context are however significantly more complex since the notions of Fisher information and score functions are much less tractable in hidden Markov models (see for instance \cite{Douc2004}).



\vspace{1ex}
The prior distributions $\Pi_1(\dd Q,\dd\mathbf{F})$ considered above, which lead to the Bernstein-von Mises property of the posterior on $Q$, rely on a crude modelling of the emission densities and are therefore not well behaved as far as the estimation of $\mathbf F$ is concerned. To overcome this problem, we adapt in Section \ref{contraction} the cut posterior approach (see \cite{lunn2000winbugs,bennett2001errors,plummer2015cuts,jacob2017better,carmona2020semi}) to our semiparametric setting. Cut posteriors were originally proposed in the context of modular approaches to modelling, where a model is assembled from a number of constituent models, each with its own parameters $\theta_i$ and data $Y_i$. In the usual construction, the cut posterior has the effect of `cutting feedback' of one of the (less reliable) data sources $Y_i$ on the other  parameters (associated to more reliable data).

\vspace{1ex}
Our approach, though also modular, departs from this setting in that we use a single data source but wish to choose different priors for different parts of the parameter. This means that we consider a prior $\Pi_2(\dd\mathbf{F}|Q)$, then combine $\Pi_1(\dd Q|Y_{1:n})$ with the conditional posterior $\Pi_2(\dd \mathbf{F} | Q,Y_{1:n})$ to produce a joint distribution. In this way, we construct a distribution over the parameters which is not a proper Bayesian posterior, but which simultaneously satisfies a Bernstein-von Mises theorem for the posterior marginal on $Q$, and is well concentrated on the emission distributions and other functionals such as the smoothing probabilities. Through this construction, we manage to combine the ``best of both worlds" in the estimation of the parametric and nonparametric parts. We believe that this idea could be used more generally in other semiparametric models. 

\vspace{1ex}
As previously mentioned, the existing posterior concentration result in the semiparametric HMMs covered only the marginal density $g^{\ell}_{Q, \mathbf F}$ of a fixed number $\ell$ of consecutive observations, see \cite{vernet2015conc}. A key step in obtaining posterior contraction rates on the emission distributions $\mathbf F$ is an inversion inequality allowing us to deduce $L^1$ concentration of the posterior distribution of the emission distributions from concentration (at the same rate) of the marginal distribution of the observations. This is  established in Theorem \ref{inversion}, from which we derive contraction rates for the cut-posterior on $\mathbf f$ in  Theorem \ref{L1contractemission}.   This inversion inequality is of independent interest and can be used outside the framework of Bayesian inference.  We finally show in Theorem \ref{smoothingcontract} that these results lead to posterior concentration of the smoothing distributions, which are the conditional distributions of the hidden states given the observations, building on \cite{de2017consistent} but refining the analysis so that we require neither sup-norm contraction rate on the emissions $\mathbf f$, nor a splitting of the data set in two parts.

\paragraph*{Organisation of the paper}
The paper is organised as follows. In Section \ref{efficiency}, we introduce the model and the notions involved, together with a general strategy for inference on $Q$ and $\mathbf{f}$ based on the cut posterior approach. In Section \ref{sec:Pi1}, we study the estimation of $Q$, proving asymptotic normality of the posterior and asymptotic efficiency. In Section \ref{contraction}, the cut posterior approach is studied, posterior contraction rates for $\mathbf{f}$ are derived together with posterior contraction rates for the smoothing probabilities. Theoretical results from both sections are then illustrated in Section \ref{simulation}. The most novel proofs for both of these sections are presented to Section \ref{proofs}, with more standard arguments, along with further details of simulations, deferred to the supplementary material \cite{supplement}. All sections, theorems, propositions, lemmas, assumptions, algorithms and equations presented in the supplement are designed with a prefix S.

\section{Inference in semiparametric finite state space HMMs}\label{efficiency}
In this section we present  our strategy to estimate $Q$ and $\mathbf F = (F_1, \dots , F_R)$, based on a modular approach which consists of first constructing the marginal posterior distribution of $Q$ based on a first prior $\Pi_1 $ on $(Q, \mathbf F)$, and then combining it with the conditional posterior distribution of $\mathbf F$ given $Q$ and $ Y_{1:n}$ based on a different prior $\Pi_2(\dd\mathbf{F}|Q)$. 
\subsection{Model and notation}\label{sec:notation}
Hidden Markov models (HMMs) are latent variable models where one observes $Y_{1:n} = (Y_t)_{1\leq t\leq n}$ whose distribution is modelled via latent (non observed) variables $X_{1:n} = (X_t)_{1\leq t\leq n} \in [R]^n$ which form a Markov chain. In this work we consider finite state space HMMs:
\begin{equation}\label{model}
Y_t | X^n \sim F_{X_t} , \quad (X_t)_{t\leq n} \sim \text{MC}(Q), \quad \mathbb{P}_{Q,\mathbf{F}}(X_t=r|X_{t-1}=s) = Q_{rs}, \quad r,s \leq R
\end{equation} 
and the number of states $R$ is assumed to be known throughout the paper. 
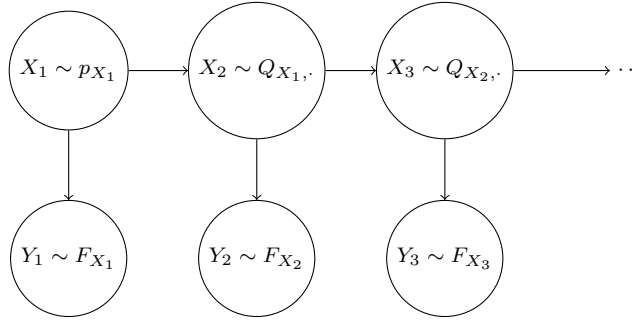
\begin{figure}[h]
\centering
\begin{tikzpicture}[scale=2.5][node distance=2cm]
\draw (0,0) node[circle,draw] (A) {$X_1\sim p_{X_1}$}  
(1,0) node[circle,draw] (B) {$X_2\sim Q_{X_1,\cdot}$}
(2,0) node[circle,draw] (C) {$X_3\sim Q_{X_2,\cdot}$}
(0,-1) node[circle,draw] (D) {$Y_1\sim F_{X_1}$}
(1,-1) node[circle,draw] (E) {$Y_2\sim F_{X_2}$}
(2,-1) node[circle,draw] (F) {$Y_3\sim F_{X_3}$}
(3,0) node[] (G) {$\cdots$};

\draw[->]
  (A) edge (B) (B) edge (C) (A) edge (D) (B) edge (E) (C) edge (F) (C) edge (G);
\end{tikzpicture}
\caption{Visual representation of data generating process of HMM}
\end{figure}
\vspace{1ex}

The parameters are then the transition matrix $Q = (Q_{rs})_{r,s \leq R}$ of the Markov chain  and the emission distributions $\mathbf F=\{F_r\}_{r\leq R}$, which represent the conditional distribution of the observations given the latent states. For a transition matrix $Q$ we denote by $p_Q$ its invariant distribution (when it exists).

\vspace{1ex}
Throughout the paper we denote by $\mathbf F^* = \{F_r^*\}_{r\leq R}$ and $Q^*$ respectively the true emission distributions and the true transition matrix. The aim is to make inference on $\mathbf F^*, Q^*$, and some functionals of these parameters, using likelihood based methods and in particular Bayesian methods. We assume  that the distributions $F_j^*$, $j=1, \dots, R$ are absolutely continuous with respect to some measure $\lambda$ on $\mathcal Y\subset \mathbb R^d$, $d\geq 1$  and we  denote by $f_1^*, \dots, f_R^* $ their corresponding densities. 
 
\vspace{1ex} 
When the latent states $(X_t)_{t\geq 1}$ are independent and identically distributed on $[R]$, the parameters are not identifiable unless some strong assumptions are made on the $F_j$'s, see \cite{Allman2009}. However, in  \cite{gassiat2016inference}, it is proved that under weak assumptions on the data generating process, both $Q$ and $\mathbf F$ are identifiable up to label swapping (or label switching). More precisely, let $\mathbb P^{(\ell)}_{Q, \mathbf F} $ be the marginal distribution of $\ell$ consecutive observations from model \ref{model} with parameters $Q$ and $\mathbf F$, so that
\begin{equation}\label{ell-marginal}
\mathbb P^{(\ell)}_{Q, \mathbf F}(dy_{1:\ell}) = \sum_{i_1, \dots, i_\ell=1}^R p_{Q}(i_1)\prod_{t=1}^{\ell-1}\left( Q_{i_t, i_{t+1}} F_{i_t}(dy_t)\right) F_{i_\ell}(dy_\ell).
\end{equation} 
Under assumptions of ergodicity $p_Q$ exists and is unique. For example, this holds under the assumption $Q_{ij}>0$ for all $i,j\in [R]$. see for instance equation (1) in \cite{gassiat2014posterior}. Denote by $\mathcal{Q}=\Delta_{R}^R\subset [0,1]^{R\times R}$ the $R-$fold product of the $R-1$ dimensional simplex and let $\mathcal P$ be the set of probability measures on $ \mathbb R^d$.
Consider the following assumption:
\begin{assume}\label{identifiability}
\begin{itemize}
\item (i) The latent chain $(X_t)_{t\geq 1}$ has true transition matrix $Q^*=(Q^*_{ij})$ satisfying $\det Q^* \neq 0$.

\item (ii) The true emission distributions $\{F_r^*\}_{r=1}^R$ are linearly independent.
\end{itemize} 
\end{assume}

Then from \cite{gassiat2016inference}, if \textbf{Assumption} \ref{identifiability} holds, then  for any $Q \in \mathcal Q$ and any $F_1, \dots, F_R \in \mathcal P $, if $\mathbb P^{(3)}_{Q, \mathbf F}  =\mathbb P^{(3)}_{Q^*, \mathbf F^*} $ then  $Q=Q^*$ and $\mathbf F = \mathbf F^*$ up to label swapping. By ``up to label swapping" we precisely mean that there exists a permutation $\tau$ of $[R]$ such that ${}^\tau Q=Q^*$  and  ${}^\tau\mathbf F = \mathbf F^*$, where ${}^\tau Q = (Q_{\tau(r), \tau(s)}, r,s\in [R])$ and ${}^\tau \mathbf F = (F_{\tau(1)},\dots,F_{\tau(R)})$. The requirement for such a permutation $\tau$ is unavoidable, since the labelling of the hidden states is fundamentally arbitrary. Correspondingly, the results which follow will always be given up to label swapping. In a slight abuse of notation, we will sometimes interchange $\mathbf F$ and $\mathbf{f}= (f_1, \dots, f_R)$, the latter being the densities of $\mathbf{F}$ with respect to some measure $\lambda$ in a dominated model.

\vspace{1ex}
The likelihood associated to model \eqref{model}, when $(F_r)_{r\in [R]}$ are dominated by a measure $\lambda$, is then given by
\begin{equation}\label{likelihood}
L_n(Q,\mathbf F)=g^{(n)}(Q, \mathbf F)(y_1, \dots, y_n) = \sum_{i_1, \dots, i_n=1}^R p_{Q}(i_1)\prod_{t=1}^{n-1} Q_{i_t, i_{t+1}} f_{i_t}(y_t) f_{i_n}(y_n).
\end{equation}
Extension to initial distributions different from the stationary one is straightforward, under the exponential forgetfulness of the Markov chain, which holds under our assumptions below, see Section \ref{sup:sec:forgetfulness}.

\vspace{1ex}
If $\Pi$ is a prior on $\mathcal Q \times \mathcal{F}^R$, with $\mathcal{F} = \{f : \mathbb R^d \rightarrow \mathbb R_+: \int f(y)d\lambda(y)  = 1\} $ is the set of densities on $\mathbb R^d$ with respect to $\lambda$,  then the Bayesian posterior $\Pi(\cdot|Y_{1:n})$ is defined as follows: for any Borel subset $A$ of $\mathcal{Q}\times\mathcal{F}^R$, we have
\begin{equation}\label{Posterior}
    \Pi(A|Y_{1:n})=\dfrac{\int_A L_n(Q,\mathbf F)\dd\Pi(Q,\mathbf F)}{\int_{\mathcal{Q}\times\mathcal{F}^R} L_n(Q,\mathbf F)\dd\Pi(Q,\mathbf F)}.
\end{equation}
This posterior is well defined as soon as almost surely with respect to the prior $p_Q$ exists, which holds for instance when $\Pi(\min_{i,j}Q_{ij}>0)=1$, which implies that the transiton matrix is ergodic $\Pi-$almost surely by the earlier remarks. Throughout the paper we consider the parametrization of $Q$ given by $\tilde Q = (Q_{rs}, r\in [R], s \in [R-1]) $  so that for each $r$ we have $Q_{rR}=1-\sum_{s<R}Q_{rs}$. Hence, specification of the matrix $Q$ amounts to specification of the $(R)\times (R-1)$ matrix $\tilde{Q}\in\tilde{\mathcal{Q}}$ for which $\sum_{s=1}^{R-1}Q_{rs}\leq 1$ for all $r$, and we will identify $Q$ with $\tilde{Q}$ (and $\mathcal{Q}$ with $\tilde{\mathcal{Q}}$) when making statements about asymptotic distributions.
\vspace{1ex}

It will be helpful to consider that the stationary HMM $(X_t,Y_t)_{t\geq 1}$ is defined on $(X_t,Y_t)_{t\in\mathbb{Z}}$ where the $(X_t)_{t\in\mathbb{Z}}$ and the $(Y_t)_{t\leq 0}$ are not observed. This is possible to define by considering the reversal of the latent chain. 

\vspace{1ex}
We  use $\mathbb{P}_*$ to denote the joint law of the variables $(X_t,Y_t)_{t\in\mathbb{Z}}$ under the  stationary distribution associated with the parameters $Q^*,\mathbf F^*$. Estimators are then understood to be random variables which are measurable with respect to $\sigma((Y_t)_{1\leq t\leq n})$.

\subsection{Cut posterior inference: A general strategy for joint inference on $Q$ and $\mathbf{f}$} \label{sec:cut:gene}

In Section \ref{sec:normaQ} below we construct a family of prior distributions $\Pi_1$ on $\mathcal Q \times \mathcal{F}^R$ such that the associated marginal posterior distribution of $Q$ follows a Bernstein-von Mises theorem  centred at an asymptotically normal regular estimator $\hat{Q}$,  i.e. with probability going to 1 under $\mathbb{P}_*$,
\begin{equation}\label{BvMdef}
\lVert \Pi_1( \min_{\tau \in \mathcal S_R}  \sqrt{n}({}^\tau Q-\hat{Q}) \in (\cdot) | Y_{1:n}) -N(0,V)\rVert_{TV} = o_{\mathbb{P}_*}(1)
\end{equation} 
and
 \begin{equation} \label{asymptoticnormaldef}
 \min_{\tau \in \mathcal S_R}\sqrt{n}({}^\tau\hat{Q}-Q^*)\Rightarrow^* N(0,V), 
 \end{equation}
where ${\Rightarrow}^*$ denotes convergence in distribution under $\mathbb{P}_*$ and where $ \mathcal S_R $ is the set of  permutations of $[R]$.
 

\vspace{1ex}

However, the choice of  $\Pi_1$ which we require for this control over $Q$ leads to a posterior distribution on $\mathbf F$ which is badly behaved, see Section \ref{sec:Pi1}. In order to jointly estimate well  $Q$ and $\mathbf F$, we propose the following cut posterior approach.  Consider a second conditional prior distribution $\Pi_2( d\mathbf F| Q) $ which leads to a conditional posterior distribution of $\mathbf F$ given $ Q, Y_{1:n}$ of the form 
$$ 
    \Pi_2(d\mathbf F|Q, Y_{1:n})=\dfrac{ L_n(Q,\mathbf F)\dd\Pi_2(\mathbf F|Q)}{\int_{\mathcal{F}^R} L_n(Q,\mathbf F)\dd\Pi_2(\mathbf F|Q) }.$$ 
    
The cut posterior is then defined as the probability distribution on $\mathcal Q \times \mathcal F^R$ by :
\begin{equation}\label{cutposterior}
\dd\Pi_{cut} ((Q ,\mathbf F )|  Y_{1:n}) = \dd\Pi_1(Q | Y_{1:n} ) \dd \Pi_2 (\mathbf F|Q, Y_{1:n}).
\end{equation}
Note that if $\Pi_2 (d\mathbf F| Q) = \Pi_1 (d\mathbf F| Q) $ then $\dd\Pi_{cut} ((Q ,\mathbf F )|  Y_{1:n}) $ is a proper posterior distribution and is equal to $ \dd \Pi_1((Q ,\mathbf F )|  Y_{1:n}) $. The motivation behind the use of the cut posterior $\dd\Pi_{cut} ((Q ,\mathbf F )|  Y_{1:n}) $ is to keep the good behaviour $\Pi_1(Q| Y_{1:n})$ while being flexible in the modelling of  $\mathbf F$ to ensure that the posterior distribution over both $\mathbf F$ and $Q$ (and functionals of the parameters) are well behaved. 
\vspace{1ex}

Adapting the proof techniques from \cite{ghosal2007convergence} to posterior concentration rates for cut posteriors, we derive in Section \ref{contraction} contraction rates for cut posteriors in terms of the $L^1$ norm of $g^{(3)}_{Q,\mathbf F} -g^{(3)}_{Q^*,\mathbf F^*} $, where $g^{(3)}_{Q,\mathbf F}$ is the density of $\mathbb P^{(3)}_{Q,\mathbf F} $ with respect to the dominating measure $\lambda$. That is, we show that under suitable conditions and choice of $\epsilon_n=o(1)$, $\Pi_{cut} (\|g^{(3)}_{Q,\mathbf F} -g^{(3)}_{Q^*,\mathbf F^*}\|_1\leq \epsilon_n |  Y_{1:n})  = 1 + o_{\mathbb{P}_*}(1)$.

\vspace{1ex}
To derive cut posterior contraction rates in terms of the $L^1$ norm of $\mathbf f - \mathbf f^*$, we prove  in Section \ref{sec:emissions} an inversion inequality in the form 
$$ \sum_{r=1}^R \|f_r - f_r^*\|_1 \lesssim \|Q - Q^*\| + \|g^{(3)}_{Q,\mathbf F} -g^{(3)}_{Q^*,\mathbf F^*}\|_{L^1},$$
 which is of independent interest.
 
\vspace{1ex}
We also derive cut-posterior contraction rates for the smoothing probabilities 
$(p_{Q, \mathbf F}(X_i = \cdot | Y_{1:n}))_{i=1,\dots,n}$ in Section \ref{sec:smoothing}. In contrast with \cite{de2017consistent}, concentration rates for $(p_{Q, \mathbf F}(X_i = \cdot | Y_{1:n}))_{i=1,\dots,n}$ do not require to split the data into 2 groups nor do they require to to have a control of $f_r - f_r^*$ in sup-norm. We can avoid these difficulties thanks to the Bayesian approach as is explained in Section \ref{sec:smoothing}.

\vspace{1ex}
In our implementation of the cut posterior, we adopt a nested MCMC approach of the kind detailed in \cite{carmona2020semi} and \cite{jacob2017better}, see Section \ref{simulation} for details. 

\vspace{1ex}
In the following section we present $\Pi_1$ and show that the associated marginal posterior distribution $\Pi_1 (Q| Y_{1:n})$ is asymptotically normal in the sense of Equation (\ref{BvMdef}).

\section{ Semi - parametric estimation of $Q$: Bernstein-von Mises property and efficient estimation }\label{sec:Pi1}

The prior  $\Pi_1 $ is based on a simple histogram prior on the $f_1, \dots , f_R$. For the sake of simplicity we present the case of univariate data; the multivariate case can be treated similarly. Without loss of generality we assume that  the observations belong to $[0,1]$, and note that if $\mathcal Y= \mathbb R$, we can transform the data to $[0,1]$ via a $C^1$ diffeomorphism (such as $\phi(x)=\frac{\exp(x)}{1+\exp(x)}$) prior to the analysis. The prior relies on a partition of the space $[0,1]$ into finitely many intervals $\{I_1,I_2,\dots\}$, and transforming the data is equivalent to constructing a prior based on the corresponding partition $\{\phi^{-1}(I_1),\phi^{-1}(I_2),\dots\}$ of $\mathbb{R}$. The construction of $\Pi_1$ is very similar to the construction considered in \cite{Gassiat2018}.

\vspace{1ex}
Let $M \in \mathbb N$ with $M \geq R$ and consider $\mathcal I_M = (I_m^{(M)}, m \leq {\kappa_M})$ a partition of $[0,1]$ into $ \kappa_M$ bins, with  $\kappa_{(\cdot)} :\mathbb{N}\rightarrow\mathbb{N}$ a strictly increasing sequence. Given $\mathcal I_M $, we consider the 
model of piecewise constant densities as the set, $\mathcal F_M$, of densities  with respect to Lebesgue measure, in the form:
\begin{equation}\label{hist}
    f_{\omega}=\sum_{m=1}^{\kappa_M} \dfrac{\omega_m}{\lvert I_m^{(M)} \rvert}\mathds{1}_{I_m^{(M)}}, \quad \min_m \omega_m \geq 0, \quad \sum_{m=1}^{\kappa_M} \omega_m = 1 ,
\end{equation}
where $\lvert I_m^{(M)} \rvert$ denotes the length of interval $I_m^{(M)}$. We could for instance consider a sequence of dyadic partitions with $\kappa_M=2^M$, such partitions are admissible for sufficiently large $M$ in the sense we detail below, and are used for the empirical investigation of Section \ref{simulation}.

\vspace{1ex}

The parameters for this model are then $Q$ and $\underline{\omega}_{(M)}=(\omega_{mr})_{m=1,\dots,\kappa_M}^{r=1,\dots, R}$, the latter of which varies in the set $$\Omega_M= \{\underline{\omega}_{(M)} \in [0,1]^{\kappa_M\times R} \hspace{1ex}; \sum_{m=1}^{\kappa_M}\omega_{mr} = 1, \forall r\in [R] \}.$$ Through (\ref{hist}), we identify each $\underline{\omega}_{(M)}\in\Omega_M$ with a  vector of emission densities  $\mathbf{f}_{\underline{\omega}_{(M)}} \in\mathcal{F}_M^R$, and thus a prior distribution $\Pi_M$ over the parameter space $\mathcal{Q}\times\Omega_M$ is identified with a prior distribution $\Pi$ over $\mathcal{Q}\times\mathcal{F}_M^R$. The corresponding posterior distribution is denoted $\Pi_M(\cdot |Y_{1:n})$ and is defined through (\ref{Posterior}).

\vspace{1ex}
Throughout this section we write $\underline{\omega}_{r} := ( \omega_{mr}, m \leq \kappa_M )$ and we denote by $\mathcal M(\mathcal I_M)$ the hidden Markov model associated with densities of the form \eqref{hist}. Note that for all $M>0$, $\mathcal M(\mathcal I_M)$ is of dimension $(\kappa_M-1)^R R(R-1)<\infty$. 

A key argument used in \cite{gassiat2016inference} to identify $Q^*$ from  $g^{(3)}_{Q^*, \mathbf f^*}$, is to find a partition $\mathcal I_M$ for some $M>0$ such  that the matrix 
$$ \mathbf{F}^{*}(\mathcal I_M) := \left( \begin{array}{ccc}
 F^{*}_{1} (I_1^{(M)}),&  \cdots ,&  F^{*}_{1} (I_{\kappa_M}^{(M)}) \\
  F^{*}_{2} (I_1^{(M)}),&  \cdots ,&  F^{*}_{2} (I_{\kappa_M}^{(M)})\\
  \vdots & \cdots\cdot&\quad \vdots \\
   F^{*}_{R} (I_1^{(M)}),&  \cdots ,&  F^{*}_{R} (I_{\kappa_M}^{(M)}) \end{array}\right)$$
   has full rank. We call such a partition an \textit{admissible} partition. 
\begin{defn}\label{admispartition}
A partition $\mathcal I_M$ is said to be \textit{admissible} for $f_1^*, \dots, f_R^{*}$ if the rank of $\mathbf{F}^{*}(\mathcal I_M)$ is equal to $R$. 
\end{defn}

\begin{rk}\label{datareduction}
Note that although we are using piecewise constant functions to model the emission densities, we do not assume that the $f_r^*$, $r \in [R]$ are piecewise constant and the simplified models $M$ are not meant to lead to good approximation of the emissions densities. However, interestingly, as far as the parameter $Q$ is concerned, the likelihood induced by such a model is not mis-specified but corresponds rather to a coarsening of the data.  Indeed it corresponds to the likelihood associated to observations of $Y_i^{(M)} = (\mathds{1}_{I_1^{(M)}}(Y_i), \cdots,\mathds{1}_{I_{\kappa_M}^{(M)}}(Y_i))$ and for such observations the simplified model leads to a well -specified likelihood. Note also that although we are modelling densities with respect to Lebesgue measure we do not require $F_r^*$ to have density with respect to Lebesgue measure, since the quantity of importance are the probabilities $  F^{*}_{r} (I_m^{M)})$ , $r\in [R] , \, m\in [\kappa_M]$. 

\vspace{1ex}
This particular coarsening was introduced in the context of mixtures by \cite{Gassiat2018}. It is not at all obvious how other types of coarsening can be found in order to use \textit{ valid} parametric models in the sense that they are well specified models for the coarsened data and the parameter of interest $Q$. 
\end{rk}

\subsection{Asymptotic normality and Bernstein-von Mises} \label{sec:normaQ}

In this section we study the asymptotic behaviour of the marginal posterior distribution of $Q$ under model $\mathcal M(\mathcal I_M)$, together with the asymptotic normality of the maximum likelihood estimator $\hat{Q}_{n,M}$ in this model. As mentioned in Remark 
\ref{datareduction}, the likelihood associated to model $\mathcal M(\mathcal I_M)$ is given by 
\begin{equation} \label{reducedLikelihood}
L_n(Q, \omega_{(M)} )= \sum_{i_1, \dots, i_n=1}^R  p_Q(i_1) \left(\prod_{t=1}^{n-1}Q_{i_{t},i_{t+1}}f_{\underline{\omega}_{i_t}}(Y_t^{(M)})\right)f_{\underline{\omega}_{i_n}}(Y_n^{(M)})
\end{equation}
with the abuse of notation 
$f_{\underline{\omega}_{i_t}}(Y_t^{(M)}) = \sum_{m=1}^{\kappa_M}\frac{\omega_{mi_t}Y_t^{(M)}(m) }{ |I_m^{(M)}|} $ for $t=1,\dots,n$.

In other words, our likelihood becomes one of a hidden Markov model with finite state space and multinomial emission distributions, and under $\mathbb{P}_*$, $(Y_1^{(M)}, \cdots, Y_n^{(M)})$ arise from a hidden Markov model with multinomial emission distributions with parameters $\underline{\omega}_{(M)}^*$ and transition matrix $Q^{*}$, with $\omega_{mr}^* = F^*_r(I_m^{(M)})$. We write $\theta=\theta^{(M)}:=(Q,\underline{\omega}_{(M)})$ for the parameter of $\mathcal{M}(\mathcal{I}_M)$ and suppress the superscript on $\theta$ when the dependence on $M$ is clear.

\vspace{1ex}
Asymptotic normality of the maximum likelihood estimator (MLE) of parametric finite state space hidden Markov models was considered for instance in \cite{bickel1998asymptotic}, who showed that the MLE was asymptotically normal with covariance matrix given by the inverse of the Fisher information matrix, which is given by the limiting covariance matrix of the score statistics, see  Lemma 1 of \cite{bickel1998asymptotic}.
\vspace{1ex}

Let $J_M(\theta)$ be the Fisher information matrix  associated to the likelihood \eqref{reducedLikelihood}:
$$ J_M( \theta^*) = \lim_{n \rightarrow \infty} \frac{ \mathbb E_* \left[ - D^2 \log L_n( Q^*, \underline{\omega}_{(M)}^*) \right]}{n}, $$
see \cite{bickel1998asymptotic}. 
Also let $J_M[Q,Q], J_M[\omega,\omega]$,  and $J_M[Q,\omega]$ denote the submatrices corresponding to the second derivatives with respect to $(Q,Q)$, $(\omega,\omega)$ and $(Q,\omega)$ respectively. 

\vspace{1ex}

The following theorem demonstrates that asymptotically normal, parametric-rate estimators of the transition matrix exist, although such estimators may not have optimal asymptotic variance.

\begin{theorem}\label{MLEM} 
Let $M>0$ and $\mathcal I_M$ be an admissible partition for $\mathbf F^*$, and let $\hat{Q}_{n,M}=\hat{Q}_n^{(M)}(Y_{1:n})$, be the MLE in model $\mathcal M(\mathcal I_M)$, given observations $Y_{1:n}$. Assume that  (i) of \textbf{ Assumption} \ref{identifiability} is satisfied that the transition matrix $Q^*$ is irreducible and aperiodic and that
 $ \omega_{mr}^* > 0 $ for all $m \in [\kappa_M]$ and all $r \in [R]$.  We then have (up to label-swapping)
$$\sqrt{n}(\hat{Q}_{n,M}-Q^*)\Rightarrow^* N(0,\tilde{J}_{M}^{-1}),$$
where $\tilde{J}_{M}^{-1}$ is  positive definite and
$$\tilde{J}_{M}^{-1} = (J_{M}^{-1})[Q,Q] = J_M[Q,Q] - J_M[Q,\omega]J_M[\omega,\omega]^{-1}J_M[\omega,Q].$$
\end{theorem}


The main difficulty in the proof of Theorem \ref{MLEM} is showing that the Fisher information matrix $J_M $ is invertible. 
\begin{proof}
The model $\mathcal M(\mathcal I_M)$ is a regular parametric HMM, 
 hence using Theorem 1 of \cite{bickel1998asymptotic}, we establish asymptotic normality of the MLE for the parameter $\theta=(Q,\underline{\omega}_{(M)})$ as soon as $J_M $ is invertible. We prove invertibility of $J_M $  in Section \ref{invertibleFI}, then
  projecting onto the $Q$ co-ordinates gives the result.
\end{proof}

To derive the Bernstein-von Mises Theorem associated to model $\mathcal M(\mathcal I_M)$, we need  the following assumption on the prior distributions on $Q$ and $\underline{\omega}_{(M)}$: 
\begin{assume} \label{ass:priorpi1}
\begin{itemize} 
\item The prior $\Pi_Q$  on $Q$ has positive and continuous density on $\mathcal Q$.
\item The prior on $\underline{\omega}_{(M)}$ has positive and continuous density with respect to Lebesgue measure. 
\end{itemize}
\end{assume}

\begin{theorem}\label{BvMM} 
Let $M>0$, $\mathcal I_M$ be an admissible partition for $\mathbf f^*$, and let $\Pi_M$ be a prior satisfying \textbf{Assumption} \ref{ass:priorpi1}.  Assume that  (i) of \textbf{Assumption} \ref{identifiability} is satisfied that the transition matrix $Q^*$ is irreducible and aperiodic and that
 $\omega_{mr}^* >0$ for all $m \in [\kappa_M]$ and all $r \in [R]$.  We then have (up to label-swapping)
$$\|\Pi_M( \sqrt{n}(Q - \hat{Q}_{n,M}) \in (\cdot ) | Y_{1:n}) -  N(0,\tilde{J}_{M}^{-1})\|_{TV} = o_{\mathbb{P}_*}(1).$$
\end{theorem}
As with the proof of Theorem \ref{MLEM}, parametric results apply as soon as the Fisher information matrix is seen to be invertible.
\begin{proof}
Inspecting the proof of Theorem 1 of \cite{bickel1998asymptotic}, we see (for $\theta=(Q,\underline{\omega}_{(M)})$, $J_M=J_M(\theta^*)$ the Fisher information and $l_n=\log g_n$ the log-likelihood) that
$$\sqrt{n}(\hat{\theta}-\theta^*) = n^{-\frac{1}{2}}J_M^{-1}\nabla_\theta l_n(\theta^*) +o_{\mathbb{P}_*}(1), $$
which up to $o_{\mathbb{P}_*}(1)$ is equal to the $T_n$ of \cite{DeGunst2008} at which their Bernstein-von Mises result (Theorem 2.1) is centred. Then this result implies the total variation convergence (in probability) of the posterior to the given normal distribution, by considering the marginal posterior on (the free entries of) $Q$. We remark further that the MLE $\hat{\theta}$ is regular, which follows from the characterisations of Fisher information in Lemmas 1 and 2 of \cite{bickel1998asymptotic}, the expansion of the MLE above, and an application of Le Cam's third lemma (Example 6.7 of \cite{VanderVaart1998}) along the lines of the proof of Lemma 8.14 in \cite{VanderVaart1998}.
\end{proof}

\vspace{1ex}

An interesting feature of Theorems \ref{MLEM} and \ref{BvMM} is that they essentially only require that $\mathcal I_M$ is an admissible partition of $\mathbf F^*$. For a given partition, this is an assumption on $\mathbf F^*$, and indeed the choice of the partition is important. Note however, that under Assumption \ref{identifiability} and for instance if the (Lebesgue) densities $f_r^*$ are positive and continuous, then for all sequences of embedded partitions $(\mathcal I_M)_M$ with radius going to 0 there exists an $M$ such that $\mathcal I_M$ is admissible for $\mathbf F^*$. More discussion is provided in Section \ref{simulation}.

\vspace{1ex}
Combining Theorems \ref{MLEM} and \ref{BvMM}, we see that credible regions for $Q$ based on the posterior associated to $\Pi_M$ are also asymptotic confidence regions. Their size may not be optimal however, even asymptotically. To ensure that such credible regions have optimal size while being asymptotic confidence regions, we would require that  $\tilde{J}_{M}^{-1}$ is the best possible (asymptotic) variance, but this is not true in general; while $\tilde{J}_{M}^{-1}$ is the efficient covariance matrix for the estimation of $Q$ in model $ \mathcal M(\mathcal I_M)$, it is not necessarily the semiparametric efficient covariance matrix for the estimation of $Q$ in model \eqref{likelihood}.  The existence of an efficient estimator of $Q$ in the semiparametric hidden Markov model, with likelihood (\ref{likelihood}) has not been established, although the fact that $\sqrt{n}$-convergent estimates of $Q$ exist in the literature (see Section \ref{invertibleFI}) indicate that semiparametric  efficient estimation of $Q$ should be possible.

\vspace{1ex}
In the following section we construct an efficient estimator for $Q$ and a prior leading to an efficient Bernstein-von Mises theorem. 

\vspace{1ex}



\subsection{Efficient estimation} \label{sec:efficiency}

In \cite{Gassiat2018}, in the context of semiparametric mixture models with finite state space,  the authors also consider the prior  model \eqref{hist} for the emission distributions. They derive for fixed $M$ a Bernstein-von Mises theorem similar to Theorem \ref{BvMM}, and they show that if we let $M=M_n$ go to infinity sufficiently slowly as $n\rightarrow\infty$, and if the corresponding partitions $\mathcal I_{M_n}$ are embedded with radius going to 0, then $\tilde J_{M_n} $ converges to the efficient semiparametric Fisher matrix for $Q$ and the estimator $\hat{Q}_{n,M_n}$ is efficient.
\vspace{1ex}

In this section we prove a similar result, however the proof is significantly more involved than in the case of mixture models. To study the theory on efficient semiparametric estimation  of $Q$ in the semiparametric HMM models, we follow the approach of \cite{mcneney2000application}. We first prove the LAN expansion in local submodels, which allows us to describe the tangent space, and then we prove that $\tilde J_M$ converges to $\tilde J  $ as $M$ goes to infinity where $\tilde J $ is the efficient Fisher information matrix for $Q$. Throughout Section \ref{sec:efficiency} we assume that the $F_r^* $ have density $f_r^*$ with respect to Lebesgue measure on $[0,1]$ and that the following holds:

\begin{assume}\label{ass:Qbound}
    For all  \, $i,j \in [R], \quad Q^*_{ij}>0 .$
\end{assume}

\subsubsection{ Scores and tangent space in the semiparametric model}\label{subsec:tangentspace}

We begin by exhibiting the LAN expansion for our model, following the framework of \cite{mcneney2000application}. As is usual in semiparametric efficiency arguments (see also Chapter 25 of \cite{VanderVaart1998}), this involves identifying the score functions and LAN norm along one-dimensional submodels passing through the true parameter. Since these submodels are themselves parametric, this identification can be made by following the framework of \cite{Douc2004}, who considered asymptotic normality in the context of parametric HMMs. A more thorough treatment, alongside a recollection of the relevant definitions and results in \cite{mcneney2000application}, can be found in Section \ref{sec:suppconv} of the supplementary material - in what follows, we aim to give an overview.

\vspace{1ex}
The first step is to exhibit a LAN expansion for the semiparametric model we consider. The parameter space is $\Theta =\mathcal{Q}\times\mathcal{F}^R$ and we consider the tangent space, in the sense of Definition 1 of \cite{mcneney2000application}, at the parameter $(Q^*,\mathbf{f}^*)$ as $$\mathcal{H}=\{(a,\mathbf{h}), a\in\mathbb{R}^{R\times (R-1)}, \quad \mathbf{h}=(h_r)_{r=1}^R, \quad h_r\in\mathcal{H}_r\}$$ where $\mathcal{H}_r=\mathcal{H}_r(f_r^*) = \{h_r\in L^2(f_r^*) : \int h_rf_r^* = 1, \lVert h_r\rVert_{\infty} <\infty \}$ and where we use the parametrization $\tilde Q$ defined in Section \ref{sec:notation}. Write $$\theta_n(a,\mathbf{h})=\left(Q^*+\frac{a}{\sqrt{n}}, f_1^*(1+\frac{h_1}{\sqrt{n}}),\dots, f_R^*(1+\frac{h_R}{\sqrt{n}})\right ).$$ Then $\theta_n(a,\mathbf{h})$ is a perturbation of the parameter $\theta^*=(Q^*,\mathbf{f}^*)$ along the path characterised by a given element $(a,\mathbf{h})\in \mathcal{H}$. Consider the submodel
$$\Theta_{a, \mathbf{h}} = \{ \theta_t = \left(Q^* + t a , ( f_r^*( 1 + t h_r) : r \in [R])\right), \, \lvert t \rvert \leq  \epsilon \}, $$
for some $\epsilon$ (depending on $a,\mathbf{h}$) sufficiently small that $\Theta_{a,\mathbf{h}}\subset\Theta$. Then for given $a,\mathbf{h}$ and for sufficiently large $n$, the perturbed parameters $\theta_n(a,\mathbf{h})$ are elements of $\Theta_{a,\mathbf{h}}$. This means that, for each $a,\mathbf{h}$, we can make an asymptotic expansion of the log-likelihood ratio between $\theta_n(a,\mathbf{h})$ and $\theta^*$ by analysing this ratio in the sub model $\Theta_{a,\mathbf{h}}$. To this end, we expand the gradient of the log-likelihood $\ell_n^{(a,\mathbf{h})}(t)$, in $\Theta_{a,\mathbf{h}}$ at $t_0=0$ as
$$\frac{1}{\sqrt{n}}\nabla_t\ell_n^{(a,\mathbf{h})}(t_0) = \frac{1}{\sqrt{n}}\sum_{k=1}^n\Delta_{k,\infty}^{(a,\mathbf{h})} + o_{L^2(\mathbb{P}_*)}(1),$$
where

\begin{align}\label{eq:submodscore}
    \Delta_{k,\infty}^{(a,\mathbf{h})} &= \Delta_{k,\infty}^{(a,\mathbf{h})}(Y_{-\infty:k} ) = \mathbb{E}_*[\phi(X_{k-1},X_k,Y_k)|Y_{-\infty:k}] \nonumber \\ &+ \sum_{i=-\infty}^{k-1}\left( \mathbb{E}_*[\phi(X_{i-1},X_i,Y_i)|Y_{-\infty:k}] - \mathbb{E}_*[\phi(X_{i-1},X_i,Y_i)|Y_{-\infty:k-1}] \right),
\end{align}
and where, writing $a_{rR}=-\sum_sa_{rs}$ and $Q_{r,R}=1-\sum_sQ_{rs}$, \begin{equation}\label{eq:phiscore}
    \phi(X_{k-1},X_k,Y_k) = \sum_{r=1}^R\sum_{s=1}^R 1\{X_{k-1}=r,X_k=s\}\frac{a_{r,s}}{Q^*_{r,s}} + \sum_{s=1}^R 1\{X_k=s\} h_s(Y_k).
\end{equation} These formulae arise as an application of the results of Section 6.1 of \cite{Douc2004} to the (parametric) model $\Theta_{a,\mathbf{h}}$.

\vspace{1ex}
We note that the contribution of the first term of the right hand side of \eqref{eq:phiscore}, which may be rewritten as $$ \sum_{r=1}^R\sum_{s=1}^{R-1} \left(\frac{1\{X_{k-1}=r,X_k=s\}}{Q^*_{rs}}-\frac{1\{X_{k-1}=r,X_k=R\}}{Q^*_{rR}}\right)a_{rs},$$to the expression defined in \eqref{eq:submodscore}, is precisely the score function for estimation in the model $\Theta_{a,\mathbf{0}}$ in which the emission densities are fixed and known. We see that this is equal to $a^TS_{Q^*}(Y_{-\infty:k})$, for $S_{Q^*}$ the score at $Q^*$ in the $R\times(R-1)$-dimensional parametric model with known emissions and unknown $Q$, which takes the form

\begin{equation}\label{ScoreQ}
\begin{split} 
S_{Q^*}(r, s) &(Y^{}_{-\infty : k}) = 
 \dfrac{\mathbb{P}_*({X_{k-1}=r},{X_k=s}|Y^{}_{-\infty : k}) }{Q^*_{rs}}-\dfrac{\mathbb{P}_*({X_{k-1}=r},{X_k=R}|Y^{}_{-\infty : k})}{Q^*_{rR}} \\
    &+\sum_{i=-\infty}^{k-1}\Bigg\{ \dfrac{\mathbb{P}_*({X_{i-1}=r},{X_i=s}|Y^{}_{-\infty : k})
    -\mathbb{P}_*({X_{i-1}=r},{X_i=s}|Y^{}_{-\infty : k-1})}{Q^*_{rs}}
    \\
    &\hspace{6ex}-\dfrac{\mathbb{P}_*({X_{i-1}=r},{X_i=R}|Y^{}_{-\infty : k})-\mathbb{P}_*({X_{i-1}=r},{X_i=R}|Y^{}_{-\infty : k-1})}{Q^*_{rR}}\Bigg\},
    \end{split}
\end{equation}
where we index with $(r,s)\in[R]\times[R-1]$ for convenience sake, but consider $S_{Q^*}$ as a vector of length $R(R-1)$.

\vspace{1ex}
We then rewrite \eqref{eq:submodscore}, substituting also the expression in \eqref{eq:phiscore} as
\begin{align}\label{eq:npscores}
    \Delta_{k,\infty}^{(a,\mathbf{h})} = a^TS_{Q^*} +  \sum_{r=1}^R \Bigg\{& \mathbb{P}_*(X_k=r|Y_{-\infty:k})h_r(Y_k) \nonumber \\ &+ \sum_{i=-\infty}^{k-1} (\mathbb{P}_*(X_i=r|Y_{-\infty:k}) -  \mathbb{P}_*(X_i=r|Y_{-\infty:k-1}))h_r(Y_i) \Bigg\}.
\end{align}

\vspace{1ex}
To set notation for what follows, we will denote the $r^{th}$ summand of the above display, for a given direction $h_r$, as
\begin{align}\label{Scoreh}
    H_r(h_r)(Y_{-\infty:k}): = & \hspace{1ex} \mathbb{P}(X_k=r|Y_{-\infty:k})h_r(Y_k) \nonumber \\ & + \sum_{i=-\infty}^{k-1} (\mathbb{P}(X_i=r|Y_{-\infty:k}) -  \mathbb{P}(X_i=r|Y_{-\infty:k-1}))h_r(Y_i).
\end{align}

The above discussion means that, through a Taylor expansion, we can write
\begin{align*}
    \ell_n(\theta_n((a,\mathbf{h}))) - \ell_n(\theta_*) &= \frac{1}{\sqrt{n}}\sum_{k=1}^\infty \Delta_{k,\infty}^{(a,\mathbf{h})} - \frac{1}{2n}\left(-\dfrac{\dd^2}{\dd t^2}\rvert_{t=t_1}\ell_n^{(a,\mathbf{h})}(t) \right) \quad  \\
    & = \Delta_{n,(a,\mathbf{h})} - \frac{1}{2}\lVert (a,\mathbf{h}) \rVert_{\mathcal{H}}^2 + o_{\mathbb{P}_*}(1).
\end{align*}
with $t_1 \in (0,n^{-1/2})$ and by choosing $\lVert (a,\mathbf{h}) \rVert_{\mathcal{H}}^2$ the Fisher information at $t=0$ in the model $\Theta_{a,\mathbf{h}}$, which is defined in \cite{Douc2004} as the $L^2$ norm\footnote{The $L^2$ norms of all $\Delta^{(a,\mathbf{h})}_{k,\infty}$ coincide by stationarity.} of $\Delta^{(a,\mathbf{h})}_{0,\infty}$.
We have $\Delta_{n,(a,\mathbf{h})} = \frac{1}{\sqrt{n}}\sum_{k=1}^\infty \Delta_{k,\infty}^{(a,\mathbf{h})} $,  is linear in $(a,\mathbf{h})$ and satisfies $\Delta_{n,(a,\mathbf{h})}\rightarrow N(0, \lVert (a,\mathbf{h}) \rVert_{\mathcal{H}}^2)$, by the discussion directly preceding Theorem 2 in \cite{Douc2004}. We also used above the local uniform convergence of the second derivative of the score to the Fisher information matrix at $t=0$ in $\Theta_{a,\mathbf{h}}$, which is guaranteed by Theorem 3 of \cite{Douc2004}.

\vspace{2ex}
The preceding discussion shows that our model is LAN, which is the first step in understanding efficient estimators of a parameter of interest. In our case, the parameter of interest will be $v_n(P_{n,\theta_n\uh})=Q$, which has `derivative' $\dot{v}\uh= a \in \mathbb{R}^p$ in the sense of \cite{Douc2004}, with $p=R(R-1)$.

\vspace{1ex}
Following what precedes, we  apply the convolution theorem also described in \cite{mcneney2000application}, and originally proven in \cite{van1991differentiable}. This theorem essentially states that the limiting law of a regular estimator is lower bounded by that of a Gaussian random variable, whose covariance matrix is the covariance of the \textit{efficient influence function}, which is itself characterised by the tangent space and the parameter `derivative' $\dot{v}$.

\vspace{1ex}
In Section \ref{sec:suppconv}, we detail the application of this theorem, the arguments are similar to those used in the iid setting, see for instance Section 25 of \cite{VanderVaart1998}. This essentially involves identifying influence functions $\dot{v}^T_b$ for estimation of one-dimensional functions $b^TQ$ at $b^TQ^*$, as we vary $b\in\mathbb{R}^p$. By considering the elements for which $a=0$, we first find that $\dot{v}^T_{b}$ is in the orthogonal complement (with respect to the LAN norm on $\mathcal{H}$) of the linear span of the scores in the models $\Theta_{0,\mathbf{h}}$, which is the span of the $H_r(h_r)$ as we vary $r$ and $h_r\in\mathcal{H}_r$. Write $\mathcal{A}$ for the projection onto this space, and write
\begin{equation}\label{eq:eff_score}
    \tilde{S}_Q = S_Q - \mathcal{A}S_Q
\end{equation}
for the projection of the score function $S_Q$ onto the orthogonal complement of this space, which we call the efficient score. By making further standard arguments, again detailed in Section \ref{sec:suppconv}, we find that the influence function $\dot{v}_b^T$ has variance $b^T\Tilde{J}^{-1}b$, where $\Tilde{J}$ is the covariance matrix of the efficient score $\Tilde{S}_{Q^*}$, which is the efficient information matrix. By varying $b$, this characterises the optimal limiting covariance matrix as $\tilde{J}^{-1}$, the inverse efficient information matrix.

\vspace{1ex}

We show in the following section that $\tilde J_M $ converges to $\tilde J$ when $M$ goes to infinity and that, if $M_n$ is a sequence going to infinity slowly enough, $\hat Q_{M_n}$ is asymptotically efficient and the posterior distributions  $\Pi_{M_n} ( dQ | Y_{1:n})$ satisfy the efficient Bernstein-von Mises theorem. 

\subsubsection{ Approximation by the models $\mathcal{M}(\mathcal{I}_M)$}

To prove  that $\tilde J_M $ converges to $\tilde J$ we need some additional assumptions on the true generating process and the partitions:
\begin{assume}\label{densityratio}
For all $r=1,\dots,R$, the densities $f^*_r$ are continuous, (Lebesgue) almost-everywhere positive, linearly independent  and
$$\forall r,s\leq R \hspace{2ex}0<\inf\limits_x\dfrac{f^*_r(x)}{f^*_s(x)}<\sup\limits_x\dfrac{f^*_r(x)}{f^*_s(x)}<\infty.$$
\end{assume}
We also consider a sequence of partitions with vanishing radius.
\begin{assume}\label{binsrefine}
The sequence $(\mathcal I_M)_{M \geq R} $ of partitions is embedded, and 
$$\max\limits_{m\leq \kappa_M}\text{Vol}(I_m^{(M)})\rightarrow 0, \quad \text{ as } M\rightarrow\infty,$$
where $\text{Vol}(A)$ denotes the Lebesgue measure of $A$.
\end{assume}
\textbf{Assumption} \ref{binsrefine} is directly under the control of the practitioner, and is verified for instance by considering nested dyadic partitions, as is done in Section \ref{simulation}. \textbf{Assumption} \ref{densityratio} is an assumption on the data generating process, although is very common in the context of semiparametric efficiency. Note also that, under \textbf{Assumptions} \ref{densityratio} and \ref{binsrefine}, there exists $M_0>0$ for which  $\mathcal I_{M_0}$ is admissible for $\mathbf f^*$.

\vspace{1ex}
In Proposition \ref{scoreconvergence}, we show that the efficient Fisher information matrix is precisely the limit of $\tilde{J}_{M}$. To do this, we first define the efficient scores $\tilde{S}_{Q^*}^{(M)}$ for the histogram models, whose covariance matrix will be $\tilde{J}_M$, analogously to what was done in Section \ref{subsec:tangentspace} in the context of the full semiparametric model. Recall from Remark \ref{datareduction} that the likelihood in model  corresponds to a likelihood of the hidden Markov model with multimomial emission distributions and observations $Y_i^{(M)} = (\mathds{1}_{I_1^{(M)}}(Y_i), \cdots,\mathds{1}_{I_{\kappa_M}^{(M)}}(Y_i))$, and so we emphasize in our notation that the scores in $\mathcal{M}(\mathcal{I}_M)$ depend only on these summaries. A straightforward adaptation of the earlier presentation then leads us to define the score functions for $Q$ in $\mathcal{M}(\mathcal{I}_M)$ by

\begin{equation}\label{ScoreQM}
\begin{split} 
S_{Q^*}^{(M)}(r, s) &(Y^{(M)}_{-\infty : k}) = 
 \dfrac{\mathbb{P}_*({X_{k-1}=r},{X_k=s}|Y^{(M)}_{-\infty : k}) }{Q^*_{rs}}-\dfrac{\mathbb{P}_*({X_{k-1}=r},{X_k=R}|Y^{(M)}_{-\infty : k})}{Q^*_{rR}} \\
    &+\sum_{i=-\infty}^{k-1}\Bigg\{ \dfrac{\mathbb{P}_*({X_{i-1}=r},{X_i=s}|Y^{(M)}_{-\infty : k})
    -\mathbb{P}_*({X_{i-1}=r},{X_i=s}|Y^{(M)}_{-\infty : k-1})}{Q^*_{rs}}
    \\
    &\hspace{6ex}-\dfrac{\mathbb{P}_*({X_{i-1}=r},{X_i=R}|Y^{(M)}_{-\infty : k})-\mathbb{P}_*({X_{i-1}=r},{X_i=R}|Y^{(M)}_{-\infty : k-1})}{Q^*_{rR}}\Bigg\}.
    \end{split}
\end{equation}
for $(r,s)\in[R]\times[R-1]$. For the model $\mathcal{I}_M$, the perturbations on the emissions vary in $$\mathcal{H}_{r,M}:=\{ h  = \sum_{m=1}^{\kappa_M}  \alpha_m \mathds{1}_{I_m}(y), \, \alpha_m \in \mathbb R, \, \int h f_{r,M}^*(y)dy =0 \} .$$  and we then define, for $h_r \in \mathcal H_{r,M}$, 
\begin{equation}\label{ScorehM}
\begin{split} 
    H_{r,M}(h_r)(Y^{(M)}_{-\infty : k}) &=\mathbb{P}_*(X_j=r|Y^{(M)}_{-\infty:k})h_r(Y_k) \\&+\sum_{i=-\infty}^{k-1}\left(\mathbb{P}_*(X_i=r|Y^{(M)}_{-\infty:k})-\mathbb{P}_*(X_i=r|Y^{(M)}_{-\infty:k-1})\right)h_r(Y_i).
   \end{split}
\end{equation}
Following again the arguments of Section \ref{subsec:tangentspace}, we write $\mathcal{A}_M$ for the projection onto the space spanned by the $H_{r,M}(h_r)$, and finally define

$$ \tilde{S}^{(M)}_{Q^*} = S^{(M)}_{Q^*} - \mathcal{A}_MS^{(M)}_{Q^*},$$
whose covariance matrix is $\tilde{J}_M$, the efficient Fisher information matrix for the model $\mathcal{M}(\mathcal{I}_M)$.

\vspace{1ex}
 The convergence of $\tilde{S}^{(M)}_{Q^*}$ to $\tilde{S}_{Q^*}$ is then established by a martingale argument. To show convergence of the projection operators $\mathcal{A}_M$ to $\mathcal{A}$ is more involved; we use a deconvolution argument which shows that boundedness in the space of the nuisance scores $H_r(h_r)$ implies boundedness of the $h_r$ in the index space $\mathcal{H}_r$ (and likewise for $H_{r,M}(h_r),\mathcal{H}_{r,M}$). These intermediate arguments, and the proof of the following results, are in Section \ref{proofs}.

\begin{prop}\label{scoreconvergence}
Grant \textbf{Assumptions} \ref{identifiability},  \ref{ass:Qbound}-\ref{binsrefine}. Then
$$\Tilde{S}^{(M)}_{Q^*}\rightarrow\Tilde{S}_{Q^*}, \quad  M \rightarrow \infty $$
where the convergence is in $L^2(\mathbb{P}_*)$. Moreover, we have
$\tilde{J}_M\rightarrow \tilde{J}$ as $M \rightarrow \infty $
where $\tilde{J}\geq \tilde{J}_M$ is the efficient information for estimating $Q$ in the full data model, and  is invertible.
\end{prop}
From  Proposition \ref{scoreconvergence} we deduce the following result. 

\begin{theorem}\label{BvM}
Let $\hat{Q}_n=\hat{Q}_n^{(M_n)}$ where $M_n\rightarrow\infty$ sufficiently slowly. Then under \textbf{Assumptions} \ref{identifiability},  \ref{ass:Qbound}-\ref{binsrefine}, $\hat{Q}_n$ is a $\mathbb{P}_*$-regular estimator and satisfies (up to label-swapping)
$$\sqrt{n}(\hat{Q}_n-Q^*)\rightarrow N(0,\tilde{J}^{-1})$$
where $\tilde{J}$ is the variance of the efficient score function, as defined in Proposition \ref{scoreconvergence}.
In particular, $\hat{Q}_n$ is an efficient estimator of $Q$ in the full semiparametric model. Moreover, for $\Pi_n$ a sequence of priors  placing mass on models $\mathcal M(\mathcal{I}_{M_n})$ respectively, and satisfying \textbf{Assumption} \ref{ass:priorpi1},
we have (up to label-swapping) that
$$\left\lVert \Pi_n(\sqrt{n}(Q-\hat{Q}_n) | Y_{1:n}) - \mathcal{N}(0,\Tilde{J}^{-1}) \right\rVert_{TV} = o_{\mathbb{P}_*}(1).$$
\end{theorem}

\section{Cut posterior contraction}\label{contraction}
In what follows, we study contraction rates for the cut posterior $\Pi_{cut}(\cdot|Y_{1:n})$ defined in Section \ref{efficiency}.

\subsection{Concentration of marginal densities $g^{(\ell)}_{Q,\mathbf{f}}$}
In this section, we present Proposition \ref{L1contractmarginal} which controls $\Pi_{cut}(\lVert g^{(\ell)}_{Q,\mathbf{f}} - g^{(\ell)}_{Q^*,\mathbf{f}^*}\rVert \leq \epsilon_n|Y_{1:n})$. This result follows from Theorem \ref{cutcontract}, which is an adaptation of the general approach of \cite{ghosal2007convergence} and is of independent interest, see Section \ref{sup:sec:cutcontract} for full details.

\vspace{1ex}
For the sake of simplicity we  consider a prior in the form  $\Pi_2( d\mathbf f|Q)= \Pi_2( d\mathbf f)$; extension to the case where the prior $\Pi_2$ depends on $Q$ is straightforward from Section \ref{sup:sec:cutcontract}. Note however that the conditional posterior on $\mathbf f$ given $Q$ depends on $Q$ through the likelihood. 

\vspace{1ex}
Hence, similarly to  \cite{vernet2015conc} we consider the following assumptions used to verify the Kulback-Leibler condition.
\begin{assume}\label{emissioncontractcondition}
Let $\epsilon_{n},\tilde{\epsilon}_{n}>0$ denote two sequences such that $\tilde{\epsilon}_{n}\leq \epsilon_{n} = o(1)$ and $n\tilde{\epsilon}_{n}^2\rightarrow\infty$ . Assume that the prior $\Pi_2$ on $\mathcal{F}^R$ satisfies the following conditions:

\textbf{A.} There exists $C_{\Pi_2}>0$ depending on the choice of prior $\Pi_2$ and a sequence $S_{n}\subset\mathcal{F}^{R}$ such that
$
\Pi_2\left(S_{n}\right) \gtrsim \exp \left(-C_{\Pi_2} n \tilde{\epsilon}_{n}^{2}\right)
$
and such that for all $
\mathbf f \in S_{n}$, there exists a set $S \subset \mathcal{Y}$ and functions $\tilde{f}_{1}, \ldots , \tilde{f}_{R}$ satisfying, for all $1\leq i\leq R$ and for $C_{R,Q}=4+\log(\frac{2R}{\min_{ij}Q^*_{ij}})+10^4\frac{R^2}{(\min_{ij}Q^*_{ij})^5}$,
\begin{align*}
    \int_{S}\left[ \frac{\left|f_{i}^{*}(y)-f_{i}(y)\right|^{2}}{f_{i}^{*}(y)} + \frac{\left|f_{i}^{*}(y)-\tilde{f}_{i}(y)\right|^{2}}{\tilde{f}_{i}(y)} \right] \lambda(\dd y)  &\leq \tilde{\epsilon}_{n}^{2}; \hspace{2ex}\int_{S} f_{i}^{*}(y) \max\limits_{1 \leq j \leq R} \log \left(\frac{\tilde{f}_{j}(y)}{f_{j}(y)}\right) \lambda(\dd y)  \leq \tilde{\epsilon}_{n}^{2}.\\
    \int_{S^c} f_{i}^{*}(y) \max _{1 \leq j \leq R} \log \left(\frac{f^*_{j}(y)}{f_{j}(y)}\right) \lambda(\dd y) & \leq C_{R,Q}\tilde{\epsilon}_{n}^{2};\hspace{2ex}\int_{S^c} [\tilde{f}_{i}(y)+f_{i}^{*}(y)] \lambda(\dd y)  \leq \tilde{\epsilon}_{n}^{2}. 
\end{align*}

\textbf{B.} For all constants $C>0$, there exists a sequence $\left(\mathcal{F}_{n}\right)_{n \geq 1}$ of subsets of $\mathcal{F}^{R}$ and a constant $C^\prime>0$ such that
$$
\Pi_2\left(\mathcal{F}_{n}^{c}\right)=o\left(\exp \left(-C n \tilde{\epsilon}_{n}^{2}\right)\right)
, \quad \text{and }\quad 
N\left(\frac{\epsilon_{n}}{12}, \mathcal{F}_{n}, d\right) \lesssim \exp \left(C^{\prime} n \epsilon_{n}^{2}\right),
$$
where $N\left(\frac{\epsilon_{n}}{12}, \mathcal{F}_{n}, d\right)$ is the $\frac{\epsilon_{n}}{12}$-covering number of $\mathcal{F}_n$ with respect to $d(f,\tilde{f})=\max\limits_{i=1,\dots,R}\lVert f_i-\tilde{f}_i\rVert_{L^1}$.
\end{assume}


\begin{prop}\label{L1contractmarginal}
Let $(Y_t)_{t\geq 1}$ be observations from a finite state space HMM with transition matrix $Q$ and emission densities $\mathbf f=(f_r)_{r=1,\dots,R}$. Grant \textbf{Assumptions} \ref{identifiability}-\ref{binsrefine} and consider the cut posterior obtained by choosing $\Pi_1$ as the prior $\Pi_M$ of Theorem \ref{BvMM} associated to the admissible partition $\mathcal I_M$, and $\Pi_2$ such that \textbf{Assumption} \ref{emissioncontractcondition} is verified for suitable $\epsilon_n,\tilde{\epsilon}_n$ such that $n\epsilon_n^2 \gtrsim \log n$. Then for any  $K_n\rightarrow\infty$, up to label-swapping,
$$\Pi_{cut}(\{ (Q,f) : \lVert g^{(3)}_{Q,\mathbf f} - g^{(3)}_{Q^*,\mathbf f^*} \rVert_{L_1} >  K_n\epsilon_n\}|Y_n) = o_{\mathbb{P}_*}(1).$$
\end{prop}

\begin{rk}
Given Proposition \ref{L1contractmarginal}, we may use the results of \cite{vernet2015conc} to derive posterior contraction rates for a number of priors $\Pi_2$. We explore the case that $\Pi_2$ is a (product of) Dirichlet process mixture of normals in Section \ref{sec:DirMix}, as this is a popular choice for density estimation.
\end{rk}

\begin{rk}
We could also choose $\Pi_1=\Pi_{1,n}$ equal to $\Pi_{M_n}$ for some $M_n\rightarrow\infty$ sufficiently slowly to give the refined control over the marginal cut posterior on $Q$ as described in Theorem \ref{BvM}.
\end{rk}

\subsection{Concentration of emission distributions}\label{sec:emissions}
Proposition \ref{L1contractmarginal} applies the general contraction result of Theorem \ref{cutcontract} to obtain an estimation result for the marginal distribution of the observations. Given that we already have control over the transition matrix in this setting when using $\Pi_1$ as a histogram prior, it remains to establish concentration rates for the emission distributions. Theorem \ref{inversion} allows us to translate a rate on a marginal distribution into a corresponding rate on the emission distribution.
\begin{theorem}\label{inversion}
Let the HMM satisfy the assumptions of Proposition \ref{L1contractmarginal}, and let $g^{(3)}_{Q^*,f}$ be the marginal density for $3$ consecutive observations under the parameters $(Q^*,f)$. Then there exists a constant $C=C(f^*,Q^*)>0$ such that, for sufficiently small $\lVert g^{(3)}_{Q^*,f} - g^{(3)}_*\rVert_{L_1}$,
$$\sum_{r=1}^R\lVert f_r - f_r^* \rVert_{L^1} \leq C \lVert g^{(3)}_{Q^*,f} - g^{(3)}_*\rVert_{L_1}.$$
\end{theorem}

\begin{rk}
Theorem \ref{inversion} provides an inversion inequality from $L^1$ to $L^1$, which has interest in both Bayesian and frequentist estimation of emission densities. It is the first such result for the $L^1$ distance, with Theorem 6 of \cite{de2016minimax} establishing a similar inequality in the $L^2$ case. Given that the testing assumptions of Theorem \ref{cutcontract} (and more generally, of results based on \cite{ghosal2007convergence}) are much more straightforward to verify for the $L^1$ distance, our result has particular interest in the context of Bayesian (or pseudo-Bayesian) settings.
\end{rk}

Using Proposition \ref{L1contractmarginal} and Theorem \ref{inversion} together with Theorem \ref{BvMM}, we easily deduce the following.

\begin{theorem}\label{L1contractemission}
Let $Y_{1:n}\sim \mathbb{P}^{n}_{*}$ be distributed according to the HMM with parameters $Q^*$ and $\mathbf{f}^*$, and grant \textbf{Assumptions} \ref{identifiability}-\ref{binsrefine}. Let $\Pi_1=\Pi_M$ and let $\Pi_2$ satisfy \textbf{Assumption} \ref{emissioncontractcondition}. Then, for $\epsilon_n$ as in Proposition \ref{L1contractmarginal} and any $K_n\rightarrow\infty$, we have (up to label-swapping)
$$\Pi_{cut}(\{\lVert Q-Q^*\rVert>\frac{K_n}{\sqrt{n}}, \max_{r=1,\dots,R}\lVert f_r-f_r^*\rVert_{L^1}>K_n\epsilon_n \} | Y_{1:n} )= o_{\mathbb{P}_*}(1),$$ and 
$$\|\Pi_{cut}( \sqrt{n}(Q - \hat{Q}_{n,M}) \in (\cdot ) | Y_{1:n}) -  N(0,\tilde{J}_{M}^{-1})\|_{TV} = o_{\mathbb{P}_*}(1).$$
\end{theorem}

\subsection{Concentration of smoothing distributions}\label{sec:smoothing}
When clustering data, the \textit{smoothing distribution} $\mathbb{P}_*(X_k=x|Y_{1:n})$ is often of interest. Our final main result concerns recovery of these probabilities using a (cut) Bayesian approach, establishing contraction of the posterior distribution over these smoothing distributions in total variation, by combining novel arguments with the inequality given in Proposition 2.2 of \cite{de2017consistent}. We recall the notation $\theta=(Q,\mathbf{f})$.

\begin{theorem}\label{smoothingcontract}
Grant the assumptions of Theorem \ref{L1contractemission}, together with 
\begin{equation}\label{tail:frtrue}
\max_{r \in [R]} \|\sqrt{f_r^*}\|_1 <+\infty,    
\end{equation} 
and assume that \begin{equation}\label{smoothingcond}
    \Pi_2(\max_i\lVert f_i\rVert_{L_2} > e^{\gamma n\epsilon_n^2}|Q)\leq C^\prime e^{-2 \gamma n\epsilon_n^2} \textrm{ and } \max_i\lVert f_i^*\rVert_{L_2}\leq e^{\gamma n\epsilon_n^2} 
\end{equation}for some constants $\gamma, C^\prime>0$. Then for $\epsilon_n$ as in Proposition \ref{L1contractmarginal} for which $n\epsilon_n^3\rightarrow 0$ and for any $K_n\rightarrow\infty$, we have (up to label swapping)
$$\Pi_{cut}(\lVert \mathbb{P}_\theta(X_k=x|Y_{1:n}) - \mathbb{P}_*(X_k=x|Y_{1:n})\rVert_{TV}>K_n\epsilon_n|Y_{1:n}) = o_{\mathbb{P}_*}(1) .$$
\end{theorem}

\begin{rk}
The requirement that $n\epsilon_n^3\rightarrow 0$ is used in the proof, although we expect that this condition is not fundamental and it may be possible to weaken with appropriate proof techniques. It is not clear if assumption \eqref{smoothingcond} is crucial, it is however very weak and satisfied for instance as soon as $E_{\Pi_2}(\|f_r\|^2_{L^2})< \infty$ for all $r$. 
\end{rk}

\subsection{Example: Dirichlet process mixtures of Gaussians}\label{sec:DirMix}
In this section, we show that \textbf{Assumption} \ref{emissioncontractcondition} is verified when $\Pi_2$ is a Dirichlet process mixture of Gaussians. We use the results of Section 4 of \cite{vernet2015conc}, in which \textbf{Assumption} \ref{emissioncontractcondition} is verified for $\mathbf{f}^*=(f_r^*)_{r=1}^R$ with each $f_r^*$ in the class
\begin{align*}
    \mathcal{P}(\beta,L,\gamma)=\{&f\in\mathcal{F}: \log f\textrm{ is locally }\beta-\textrm{H\"older with derivatives }l_j=(\log f)^{(j)} \\
    & \textrm{and }\lvert l_{k_\beta}(y) - l_{k_\beta}(x) \rvert \leq r!L(y)\lvert y-x\rvert^{\beta-k_\beta}\textrm{ when }\lvert x-y\rvert\leq\gamma \}.
\end{align*}
and satisfying weak tail assumptions which we give in \textbf{Assumption} \ref{emissionconditionDPM}. Here $\beta>0$, $L$ is a polynomial function, $\gamma>0$ and $k_\beta = \lceil \beta \rceil - 1$ for $\lceil\cdot\rceil$ the usual ceiling operator.

\vspace{1ex}
The following result, which is a corollary of Theorem 4.3 of \cite{vernet2015conc}, shows that for prior choices satisfying \textbf{Assumption} \ref{priorcondition}, the conditions of Theorem \ref{L1contractemission} are verified with \begin{equation}\label{eq:DPMMrate}
    \epsilon_n=n^{-\frac{\beta}{2\beta+1}}\log(n)^{t}
\end{equation} where $t>t_0\geq \left(2+\frac{2}{\gamma}+\frac{1}{\beta}\right)\left(\frac{1}{\beta}+2\right)^{-1}$.

\begin{cor}\label{DPMNormal}
Let $Y_{1:n}\sim \mathbb{P}^{n}_{*}$ be distributed according to the HMM with parameters $Q^*$ and $\mathbf{f}^*$, and grant \textbf{Assumptions} \ref{identifiability}-\ref{binsrefine}. Let $\Pi_1=\Pi_M$ and let the prior $\Pi_2$ on $\mathcal{F}^R$ take the form of an $R-$fold product of Dirichlet process mixtures of Gaussians, in which the base measure $\alpha$ and variance prior $\Pi_\sigma$ satisfy \textbf{Assumption} \ref{priorcondition}. Suppose further that the true emission distributions $(f_i^*)_{i=1}^R$ satisfy \textbf{Assumption} \ref{emissionconditionDPM} and that, for each $i$, $f^*_i\in\mathcal{P}(\beta,L,\gamma)$. Then Theorems \ref{L1contractemission} and \ref{smoothingcontract} hold with $\epsilon_n$ as in \eqref{eq:DPMMrate}, where $t>t_0\geq \left(2+\frac{2}{\gamma}+\frac{1}{\beta}\right)\left(\frac{1}{\beta}+2\right)^{-1}$.
\end{cor}

\begin{proof}
The proof follows from \cite{vernet2015conc} who showed that \textbf{Assumption} \ref{emissioncontractcondition} is verified under the stated assumptions with $\lambda(\dd y)$ the Lebesgue measure, appropriately chosen $B_n$,
$\tilde{\epsilon}_n=n^{-\frac{\beta}{2\beta+1}}\log(n)^{t_0}$ with $t_0$ as in the statement, and $\epsilon_n$ as in the statement.
\end{proof}

The rate $n^{-\frac{\beta}{2\beta+1}}$ is minimax optimal for the classes $\mathcal{P}(\beta,L,\gamma)$ under iid sampling assumptions, see \cite{maugis2013adaptive}. Although not proved here, we strongly believe the rate $\epsilon_n$ to be minimax up to $\log n$ factors, since an iid sampling assumption corresponds to estimating $f_1,\dots,f_R$ after observing $X_{1:n},Y_{1:n}$ which is easier than estimating $f_1,\dots,f_R$ from $Y_{1:n}$ only. We thank the anonymous reviewer for pointing this out to us.

\begin{rk}
It is also possible to verify the conditions of Proposition \ref{L1contractmarginal} in the case of a countable observation space, again by following the example set out in \cite{vernet2015conc}, in which case \textbf{Assumption} \ref{emissioncontractcondition} is verified with $\lambda(\dd y)$ the counting measure on $\mathbb{N}$ by using a Dirichlet process prior on the emissions whose base measure satisfies a tail condition, and under a tail assumption on the true emissions. In this case, a parametric rate up to log factors is obtained.
\end{rk}

\section{Practical considerations and simulation study}\label{simulation}
In this section, we discuss the practical implication of the method and results described in the previous sections. As described in Section \ref{sec:cut:gene} we consider a cut posterior approach where $\Pi_1$ is based on a histogram prior on the emission distributions and where $\Pi_2 $ is a Dirichlet process mixture of normals for the emission densities.  We first describe the implementation of $\Pi_1( Q| Y_{1:n})$.

\subsection{MCMC algorithm for $Q$}\label{MCMCpi1}
We first define the construction of the partition $\mathcal I_M$. In Section \ref{sec:Pi1} it is defined as a partition on $[0,1]$ (or more generally on $[0,1]^d$), but using a monotonic transform, this is easily generalized to a partition of $\mathbb R$ (or $\mathbb R^d$). In this section we restrict ourselves to dyadic paritions, i.e. $I_m^{(M)} = (G_0^{-1}((m-1) 2^{-M}), G_0^{-1}(m 2^{-M})) $, $m=0, \cdots , 2^M-1$, where $G_0$ is a continuous strictly increasing  function  from $\mathbb R$ to $[0,1]$.  

\vspace{1ex}
Recall that for Theorems \ref{MLEM} and \ref{BvMM} to be valid, we merely require that $\mathcal I^{(M)} = (I_m^{(M)}, m \leq 2^M-1)$ is admissible and that for all $r\leq R, m\leq 2^M-1$, $F_r^*(I_m^{(M)}) >0$.  Interestingly Theorem \ref{BvM} does not restrict the choice of $G_0$, but restricts the choice of $M$. 
In practice however the choice of $M$ and $G_0$ matters and more details are provided below. 

\vspace{1ex}
Once $\mathcal I_{M}$ is chosen, we consider a prior on $Q, \underline{\omega}_M$ (suppressing the $M$ henceforth) and for the sake of computational simplicity we consider the following family of Dirichlet priors: 
\begin{align*}
     \forall i\leq R, \quad & Q_{i\cdot} = (Q_{ij},j\in [R]) \sim \mathcal D( \gamma_{i1}, \cdots, \gamma_{iR}) \\  &\omega_{i \cdot} = (\omega_{im}, m \in [2^M] ) \sim \mathcal D( \beta_{i1}, \cdots, \beta_{i2^M}). 
\end{align*}
We then use a Gibbs sampler on $(Q,\underline{\omega},\mathbf{X})$ where given $\mathbf{X},Y_{1:n}$, 
\begin{align*}
Q_{i\cdot} &\stackrel{ind}{\sim} \mathcal D(\gamma_{i1}+ n_{i1}, \cdots, \gamma_{iR}+ n_{iR}), \quad n_{ij} = \sum_{t=2}^n \mathds{1}_{X_{t-1} = i, X_{t} = j} \\
\omega_{i \cdot} &\stackrel{ind}{\sim} \mathcal D(\beta_{i1}+ N^{(i)}_{0}, \cdots, \beta_{i2^M} + N^{(i)}_{2^M-1}), \quad N^{(i)}_{m} = \sum_{t=1}^n \mathds{1}_{Y_t \in I_m^{(M)}, X_t = i},
\end{align*}
and the conditional distribution of $X$ given $Q, \underline{\omega}, Y$ is derived using the forward - backward algorithm (see \cite{marin2007bayesian} or \cite{fruhwirth2006finite}). To overcome the usual label-switching issue in mixtures and HMMs, we take the approach of Chapter 6 of \cite{marin2007bayesian} which deals with MCMC in the mixtures setting, in which the authors propose relabelling relative to the posterior mode as a post-processing step (with likelihood computed also with forward-backward).
\vspace{1ex}

In our simulation study we have considered  $R=2$ hidden states  with transition matrix  and emission distributions \[Q^*=\begin{pmatrix}0.7 & 0.3\\ 0.2 & 0.8 \end{pmatrix}, \quad 
F_1^*\sim \mathcal{N}(-1,1), F_2^*\sim\mathcal{N}(1,1).\]

We have considered data of size $n=1000,2500,5000,10000$, obtained by restricting a single simulated data set of size $10000$. To study the effect of  $M$ we run our MCMC algorithm targeting $\Pi_1(\cdot|Y_{1:n})$ with $\kappa_M=2,4,8,16,64,128$ bins.

\vspace{1ex}
We took $\gamma_{ij}=\beta_{ij}=1$ so that independent uniform priors over the simplex were used for rows of the transition matrix and the histogram weights, and chose 
$$G_0(y) = \begin{cases} (1+e^{-y})^{-1} & \lvert y \rvert >3 \\ \zeta+\eta y & \lvert y\rvert \leq 3 \end{cases} $$
with $\zeta,\eta\in\mathbb{R}$ chosen so that $G_0$ is continuous. The sigmoid $y\mapsto (1+e^{-y})^{-1}$ is a popular transformation from $\mathbb{R}$ to $[0,1]$, however its steep gradient near $y=0$ means that several data points end up in the middle bins when using the partition on $\mathbb{R}$ induced by a dyadic partition in $[0,1]$, which leads to an over-coarsening of the data. The linear interpolation in the range $\lvert y \rvert \leq 3$ is intended to provide more discrimination between data points lying in this interval, which includes the vast majority (approx 97.5\%) of the data.

\vspace{1ex}
We ran the MCMC for $150000$ iterations for each binning, discarding $10000$ iterations as burn-in and retaining one in every twenty of the remaining draws, for a total of $7000$ posterior draws.


\vspace{1ex}
The fitted distributions for $Q$ under the priors $\Pi_1$, with varying $M$, are shown in Figures \ref{pi1plots:N1000}, \ref{pi1plots:N2500}, \ref{pi1plots:N5000} and \ref{pi1plots:N10000}  and demonstrate Theorems \ref{BvMM} and \ref{BvM} as we detail below. To make an additional comparison with a typical Bayesian nonparametric approach, we also fitted a model with a prior $\Pi^\prime$, which used Dirichlet priors on the transition matrix as in $\Pi_1$ but Dirichlet process mixtures of Gaussians to model the emission densities. We defer further details of $\Pi^\prime$ to Section \ref{sec:pi2sims} as it is more relevant as a comparison with $\Pi_2$, since both have a much higher computational cost in comparison to $\Pi_1$ when implementing as described in Section \ref{sec:pi2sims}, as elaborated on in Section \ref{extrasims}. In Figure \ref{fig:full_bayes_transmat}, we compare distributions for the transition matrix under $\Pi^\prime(\cdot|Y_{1:n})$ with the distribution under $\Pi_1(\cdot|Y_{1:n})$ for a selection of values of $M$, but we emphasize the large difference in computational tractability. 

\vspace{1ex}
Even when $\kappa_M=2$, the posterior distribution under the prior $\Pi_1$, as $n$ increases, looks increasingly like a Gaussian, though its variance is quite large. For slightly larger values of $\kappa_M$, this Gaussian shape is preserved but with lower variance, demonstrating Proposition \ref{infogrows} which states that the Fisher information grows as we refine the partition. However, we can also see that taking $\kappa_M$ too large leads to erroneous posterior inference, which may simply be biased (for instance when $n=5000$ and we take $\kappa_M=64$) or lose its Gaussian shape entirely (for instance when $n=1000$ and $\kappa_M\in\{64,128\}$). This demonstrates that the requirement of Theorem \ref{BvM} that $M_n\rightarrow\infty$ sufficiently slowly has practical consequences and is not merely a theoretical artefact.

\vspace{1ex}
While the latter issue is easily diagnosed by eye, the former issue is somewhat more worrisome as the bias cannot be so easily identified when one does not have access to the true data generating process. In Figure \ref{fig:full_bayes_transmat}, we see preliminary empirical evidence that this problem of bias may also be present when taking a fully Bayesian approach based on the prior $\Pi^\prime$, indicating that our approach may work better than a fully Bayes approach if $M_n$ can be tuned well. We emphasize however that this simulation study is very limited in scope, and is only intended to demonstrate our results rather than to make conclusive comparisons.

\vspace{1ex}
For the histogram prior $\Pi_1$, we propose the following heuristic to tune $\kappa_M$ after computing the posterior for a range of different numbers of bins. For a small number of bins (say $\kappa_M=4$), take the posterior mean as a reference estimate (say $\hat{Q}_0$), which should have low bias and moderate variance. For higher $\kappa_M$ values, compare $\hat{Q}_0$ $\kappa_M$-specific posterior mean and the $1-\alpha$ credible sets $C_\alpha$, for $\alpha = 0.05,0.1$ say. We consider that $\kappa_M$ is not too large if $\hat{Q}_0$ is well within the bounds of $C_\alpha$. We have added some further markings to Figure \ref{pi1plots:N5000} which illustrate this approach.

\vspace{1ex}
We emphasize that it is most important not to refine the partition too quickly, so even when adopting this heuristic one may wish to favour lower values of $\kappa_M$. When fitting $\Pi_2$ as discussed in Section \ref{sec:pi2sims}, we used posterior draws based on one possible refinement, in which $\kappa_M=4$ for $n=1000$, $\kappa_M=8$ for $n\in\{2500,5000\}$ and $\kappa_M=16$ for $n=10000$ - see Figure \ref{pi2plots:main}.

\begin{figure}
    \centering
    \includegraphics[width=\linewidth]{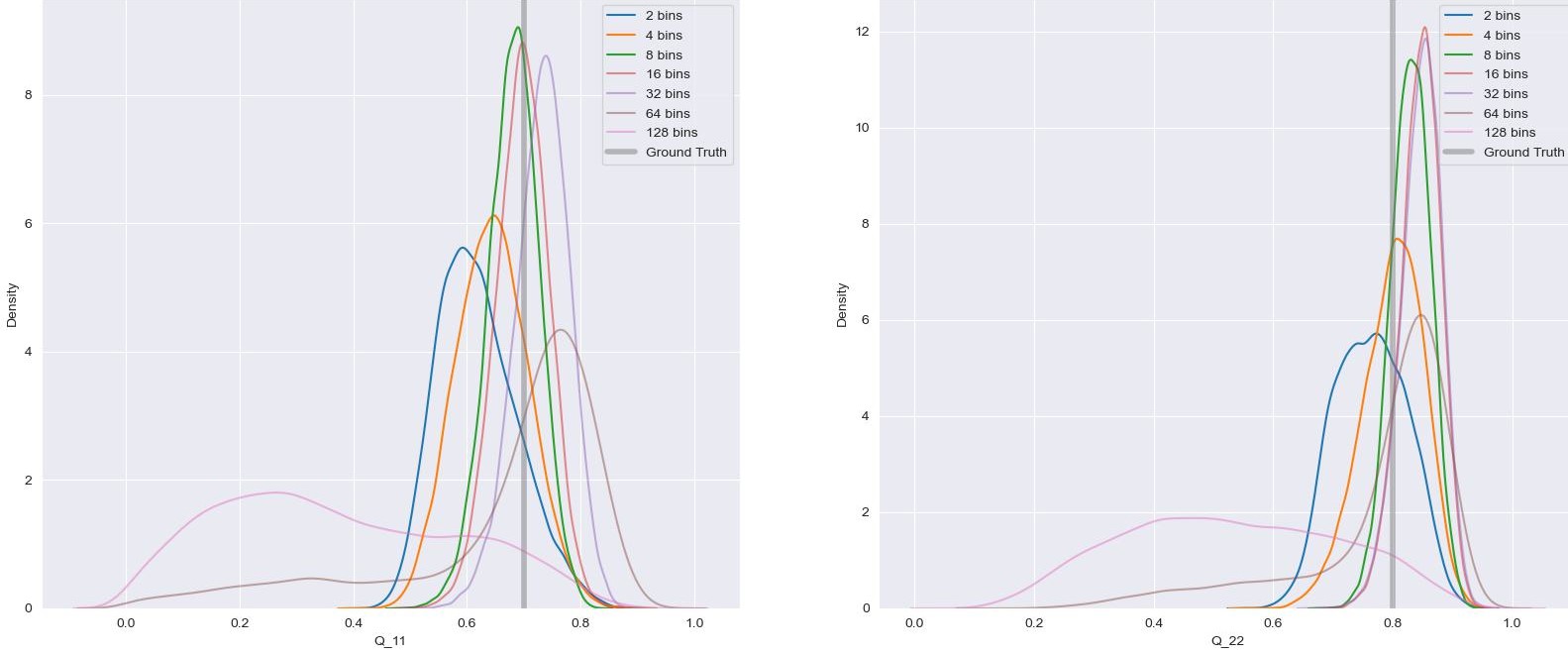}
    \caption{Posterior draws for $Q$ under priors $\Pi_1$ with $\kappa_M\in\{2,4,\dots,128\}$ bins when $n=1000$.}
    \label{pi1plots:N1000}
\end{figure}
\begin{figure}
    \centering
    \includegraphics[width=\linewidth]{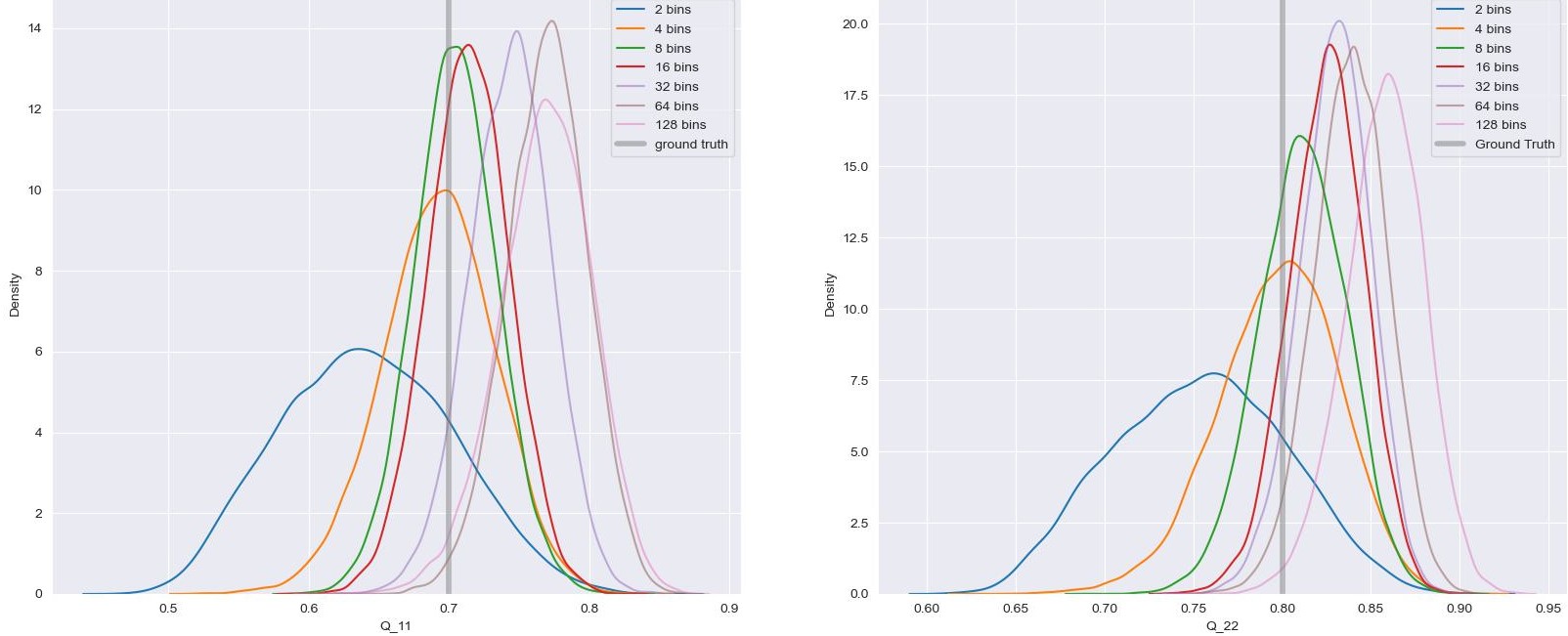}
    \caption{Posterior draws for $Q$ under priors $\Pi_1$ with $\kappa_M\in\{2,4,\dots,128\}$ bins when $n=2500$.}
    \label{pi1plots:N2500}
\end{figure}
\begin{figure}
    \centering
    \includegraphics[width=\linewidth]{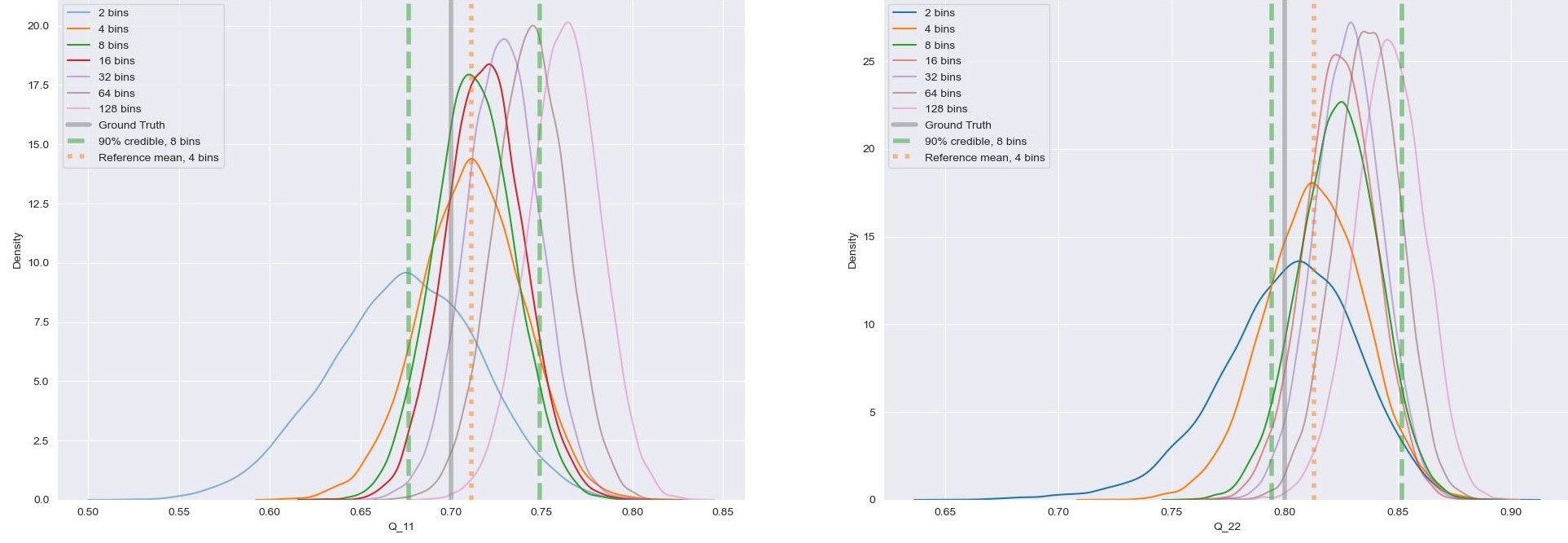}
    \caption{Posterior draws for $Q$ under priors $\Pi_1$ with $\kappa_M\in\{2,4,\dots,128\}$ bins when $n=5000$. The dotted orange line is the reference mean, taken when $\kappa_M=4$. The dashed green lines are the bounds of the 90\% credible set when $\kappa_M=8$.}
    \label{pi1plots:N5000}
\end{figure}
\begin{figure}
    \centering
    \includegraphics[width=\linewidth]{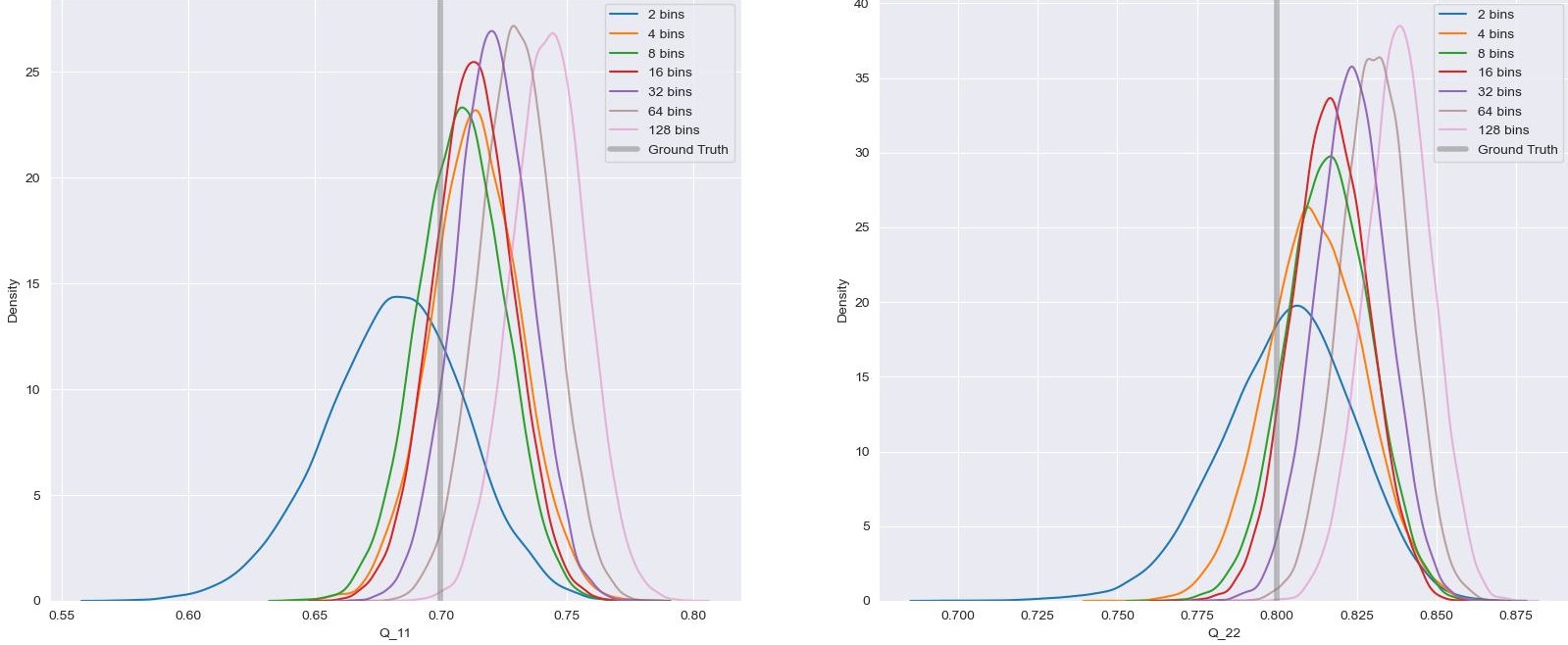}
    \caption{Posterior draws for $Q$ under priors $\Pi_1$ when $n=10000$, when using $\kappa_M\in\{2,4,\dots,128\}$ bins.}
    \label{pi1plots:N10000}
\end{figure}
\begin{figure}
    \centering
    \includegraphics[width=\linewidth]{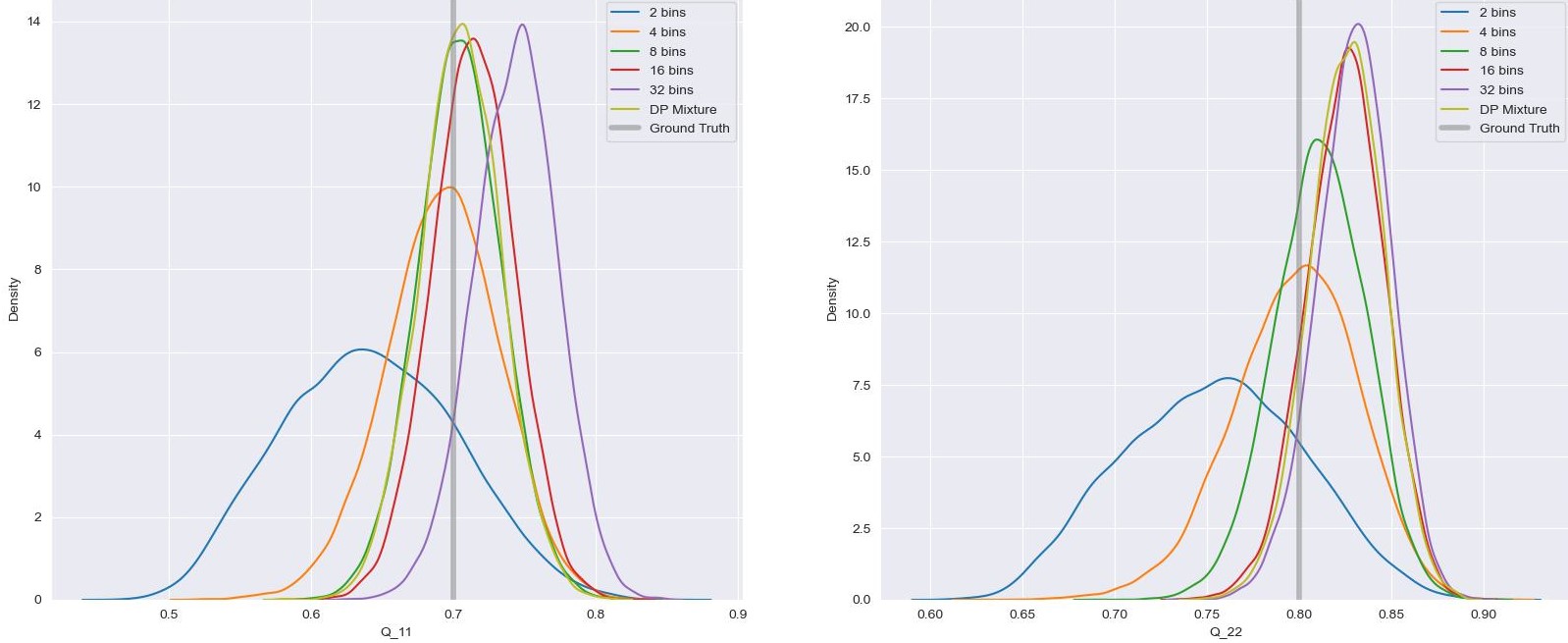}
    \centering
    \includegraphics[width=\linewidth]{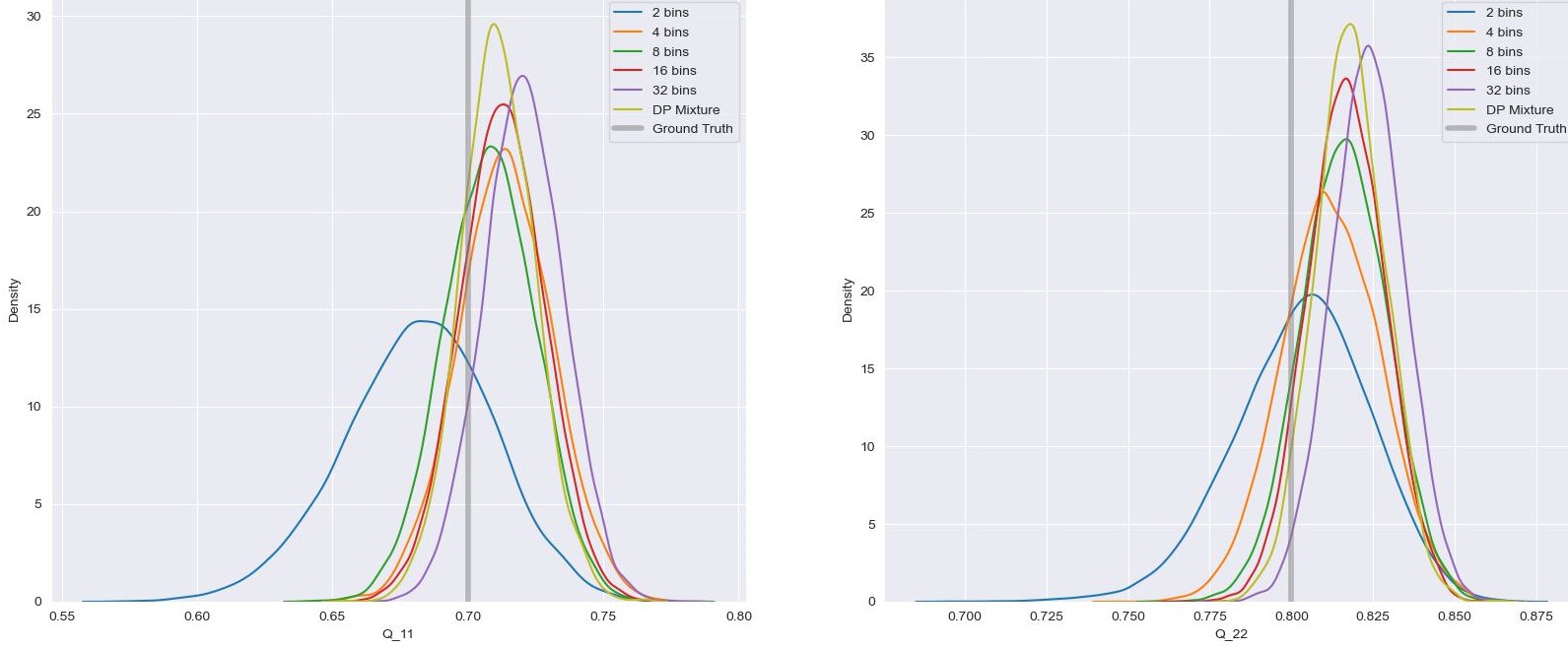}
    \caption{Posterior draws for $Q$ under priors $\Pi_1$ (for $\kappa_M\in\{2,4,8,16,32\}$) and $\Pi^\prime$ (DP Mixture line) when $n=2500$ (top) and $n=10000$ (bottom).}
    \label{fig:full_bayes_transmat}
\end{figure}

\subsection{MCMC algorithm for $\{f_r\}_{r=1}^R$}\label{sec:pi2sims}
After simulating draws from the marginal posterior $\Pi_1$ of $Q$, we then use a Gibbs sampler to target the cut-posterior distribution $\Pi_2(f|Q,Y)$, for which the prior $\Pi_2(f|Q)$ is a Dirichlet mixture as follows. 
For each hidden state $r$ (whose labels are fixed after relabelling in Section \ref{MCMCpi1}) we use independent dirichlet Process mixtures of normals to model $f_r$:
\begin{equation}\label{DPM}
    f_r(y)=\int \phi_{\sigma^{(r)}}(y-\mu)\dd P^{(r)}(\mu), \quad P^{(r)} \stackrel{iid}{\sim} DP(M_0, \mathcal N(\mu_c, \sigma_c^2))
\end{equation}
with $M_0=1$, $\mu_c=0$ and $\sigma^2_c=1$. The kernels $\phi_\sigma$ are centred normal distributions with variance $v=\sigma^2$, and $v^{(r)}$ is equipped with an $\text{InvGamma}(\alpha_\sigma,\beta_\sigma)$ prior\footnote{We remark that Proposition \ref{DPMNormal} verifies the conditions of Proposition \ref{L1contractmarginal} under an inverse gamma prior on the standard deviation. We have instead adopted in our simulations the common practice of placing an inverse gamma prior on the variance for computational convenience.}. Note that the scale parameter $\sigma$ is fixed across $\mu$'s for each $r$. The MCMC procedure will rely on the following stick breaking representation of (\ref{DPM}) involving mixture allocation variables $s^{(r)}=(s^{(r)}_i)_{i\leq n}$ and stick breaks $(V_j^{(r)})_{j\in\mathbb{N}}$ associated to a vector of weights $(W_j^{(r)})_{j\in\mathbb{N}}$:

\begin{align}\label{DPMstick}
    &Y_i|\mathbf{X},\mu,s,\sigma,W\stackrel{iid}{\sim} \phi_{\sigma^{(r)}}(\cdot - \mu_{s_i}^{(X_i)})\quad s_i|V,\mathbf{X},\mu,\sigma^2\stackrel{indep}{\sim}W^{(X_i)}, \quad W_j^{(r)}=V_j^{(r)}\prod_{i<j}(1-V_i^{(r)}) \nonumber \\& V_j^{(r)}\stackrel{iid}{\sim}\beta(1,M_0)\quad \mu_j^{(r)}\sim\mathcal{N}(\mu_c,\sigma_c^2)\quad v^{(r)}\sim\text{InvGamma}(\alpha_\sigma,\beta_\sigma)\quad \mathbf{X}\sim MC(Q) .
\end{align}

\vspace{1ex}
In order to approximately sample from the corresponding posterior, we replace the above prior with the Dirichlet-Multinomial process (see Theorem 4.19 of \cite{ghosal2017fundamentals}) with truncation level $S_{\max}=\lfloor  \sqrt{n} \rfloor$, in which $W$ is instead sampled from a Dirichlet distribution with parameter \begin{equation}\label{Dirdraws}
    \alpha = \left(\frac{M_0}{S_{\max}},\frac{M_0}{S_{\max}},\cdots,\frac{M_0}{S_{\max}}\right).
\end{equation}
This truncation level is suggested as a rule of thumb in the remark at the end of Section 5.2 of \cite{ghosal2017fundamentals}. Further details on implementation can be found in Section \ref{extrasims}.

\vspace{1ex}

\paragraph{Nested MCMC}
For the $i^{th}$ draw from the cut posterior, we would ideally sample first $Q_i$ from $\Pi_1(\cdot|Y)$, and then $f_i=(f_{ir})_r$ from $\Pi_2(\cdot|Q_i,Y)$. For the first step, we used Algorithm \hyperref[AlgoS1]{SA1} with burn in and thinning. However, we encounter difficulty when simulating from $\Pi_2(\cdot|Q_i,Y)$, as the $Q_i$ changes when $i$ changes, and so an MCMC approach effectively needs to mix $i$ by $i$. In order to achieve this, we run a nested MCMC approach (see \cite{plummer2015cuts}) in which, for each $i$, an interior chain of length $C$ is run, taking the final draw from the interior chain as our $i^{th}$ global draw.
\vspace{2ex}

Concerns about the computational cost of nested MCMC have been expressed, especially when there is strong dependence between the two modules (in our case, representing the transition and emissions) as in Section 4.2 of \cite{liu2022stochastic}. However, we found little improvement beyond $C=10$ interior iterations, see Figure \ref{pi2plots:minichain}. We expect that this is down to the well localised posterior $\Pi_1(\cdot|Y)$ on $Q$ which means that the $Q_i$, and hence the targets $\Pi_2(\cdot|Q_i,Y)$, don't vary too much $i$ by $i$. Plots of the posterior mean when $C=10$ are detailed in Figure \ref{pi2plots:main}. Since we used the thinned draws of $Q$ from $\Pi_1$, we ran 7000 such interior chains.

\vspace{1ex}
When using such a small number of iterations for the interior chain, nested MCMC is not costly compared to other cut sampling schemes (see e.g. Table 1 of \cite{liu2022stochastic}). We further suggest that the computational cost compared to fully Bayesian approaches with fixed targets is not as high as it may seem, given that the use of such interior chains should, at least partially, subsume the need for thinning. Indeed running such interior chains with a fixed target for each $i$ would be precisely the same as thinning. A more detailed development of this idea can be found in Section 4.4 of the supplement to \cite{carmona2020semi}.

\paragraph{Comparison to fully Bayesian approach}
As mentioned in Section \ref{MCMCpi1}, we also consider the fully Bayesian model with prior $\Pi^\prime$ as a means of comparison. The prior $\Pi^\prime$ independently places a Dirichlet prior on the rows of the transition matrix as discussed beforehand $\Pi_1$, as well as a Dirichlet process mixture prior over the emission densities as discussed earlier in this section for $\Pi_2$.

\vspace{1ex}
In the implementation of $\Pi^\prime$, we ran the MCMC for 70000 iterations, discarding the first 10000 as burn-in and thinning at a rate of one in ten observations. This was chosen so that we had an approximate matching with the computational cost of 7000 iterations of the cut posterior, each with 10 interior iterations. We plot the resulting emission densities in Figure \ref{emissionsfullbayes}. In comparison to Figure \ref{pi2plots:main}, the pointwise bands seem to capture the ground truth more accurately. We remark however that Theorem \ref{L1contractemission} only provides guarantees on $L^1$ concentration and so the plots will not entirely reflect the theory; $L^1$ credible sets are rather less easily visualised. We also emphasize our earlier comment from Section \ref{MCMCpi1} that the simulation study is limited in scope, and is not intended to provide conclusive comparisons with other approaches.


\vspace{1ex}





\begin{figure}
  \begin{minipage}[b]{\linewidth}
    \centering
    \includegraphics[width=\linewidth]{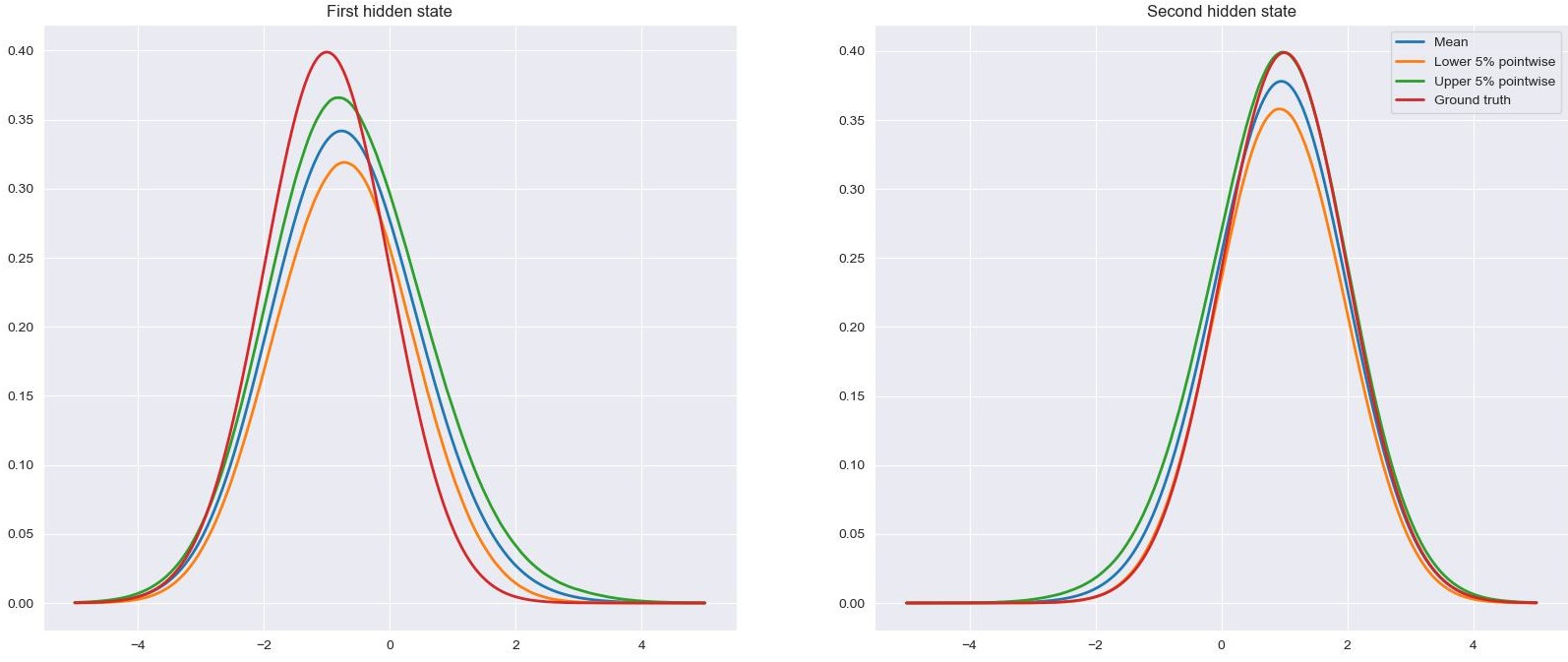} 
    \vspace{2ex}
  \end{minipage} 
  \begin{minipage}[b]{\linewidth}
    \centering
    \includegraphics[width=\linewidth]{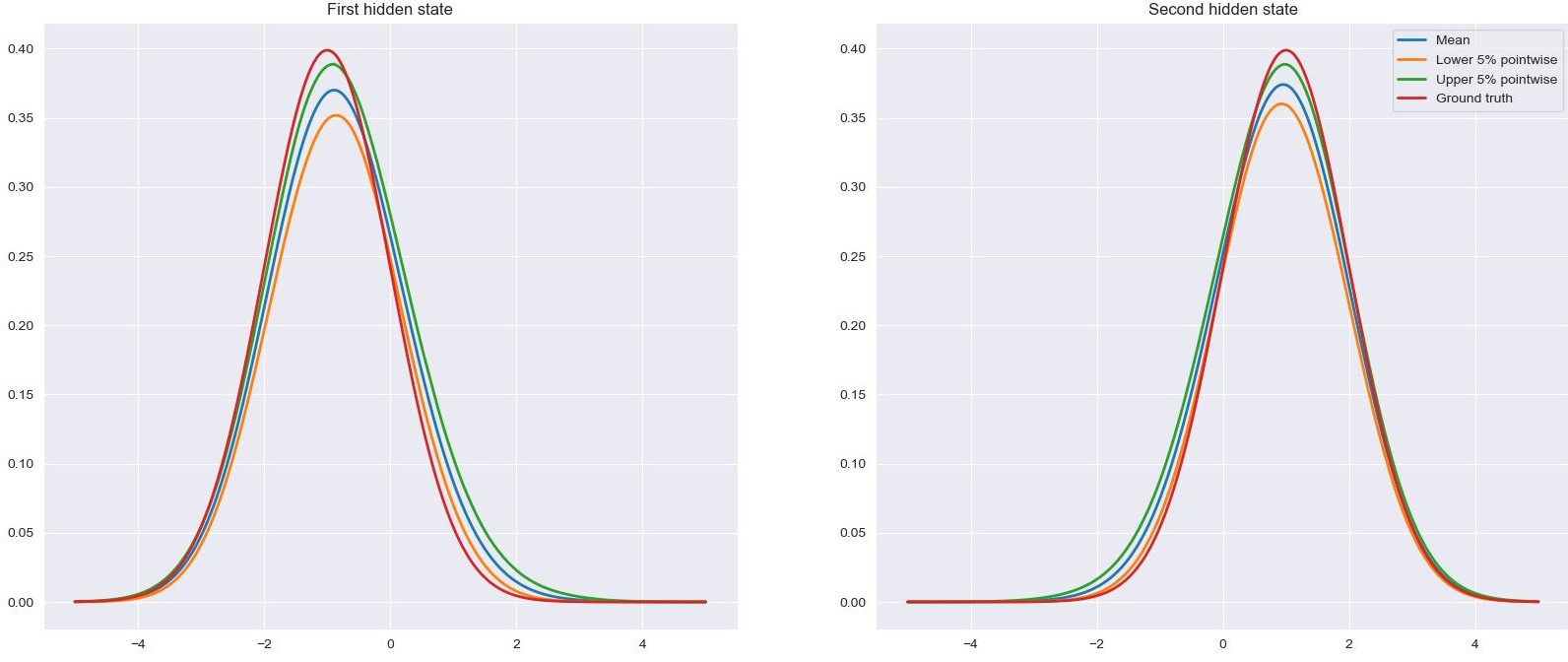} 
  \end{minipage}
  \begin{minipage}[b]{\linewidth}
    \centering
    \includegraphics[width=.95\linewidth]{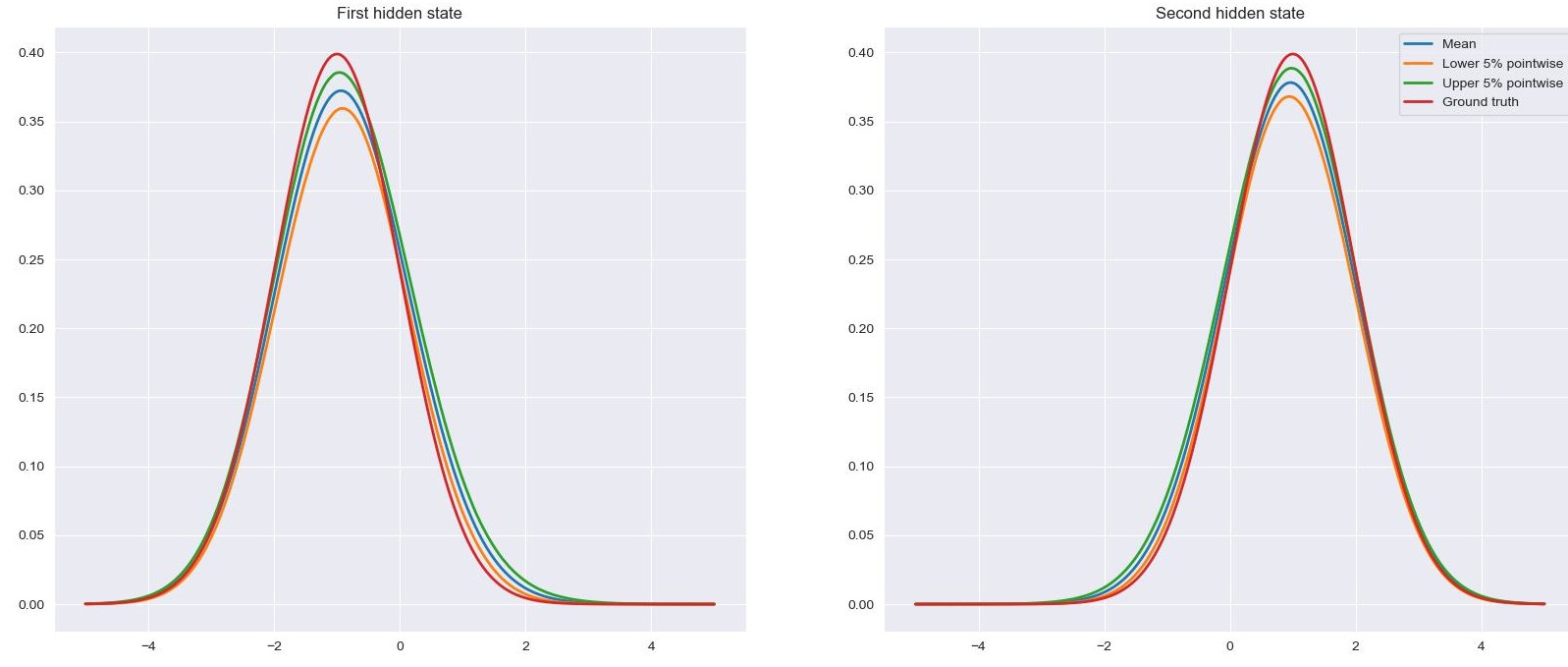} 
  \end{minipage}
  \caption{Cut posterior mean and pointwise 90\% credible bands for emissions when refining the partition with $n=2500$ (top), $n=5000$ (middle) and $n=10000$ (bottom), each with $C=10$ interior iterations. See top plot of Figure \ref{pi2plots:minichain} for $n=1000$.}
  \label{pi2plots:main} 
\end{figure}

\begin{figure}
\centering
\includegraphics[width=\linewidth]{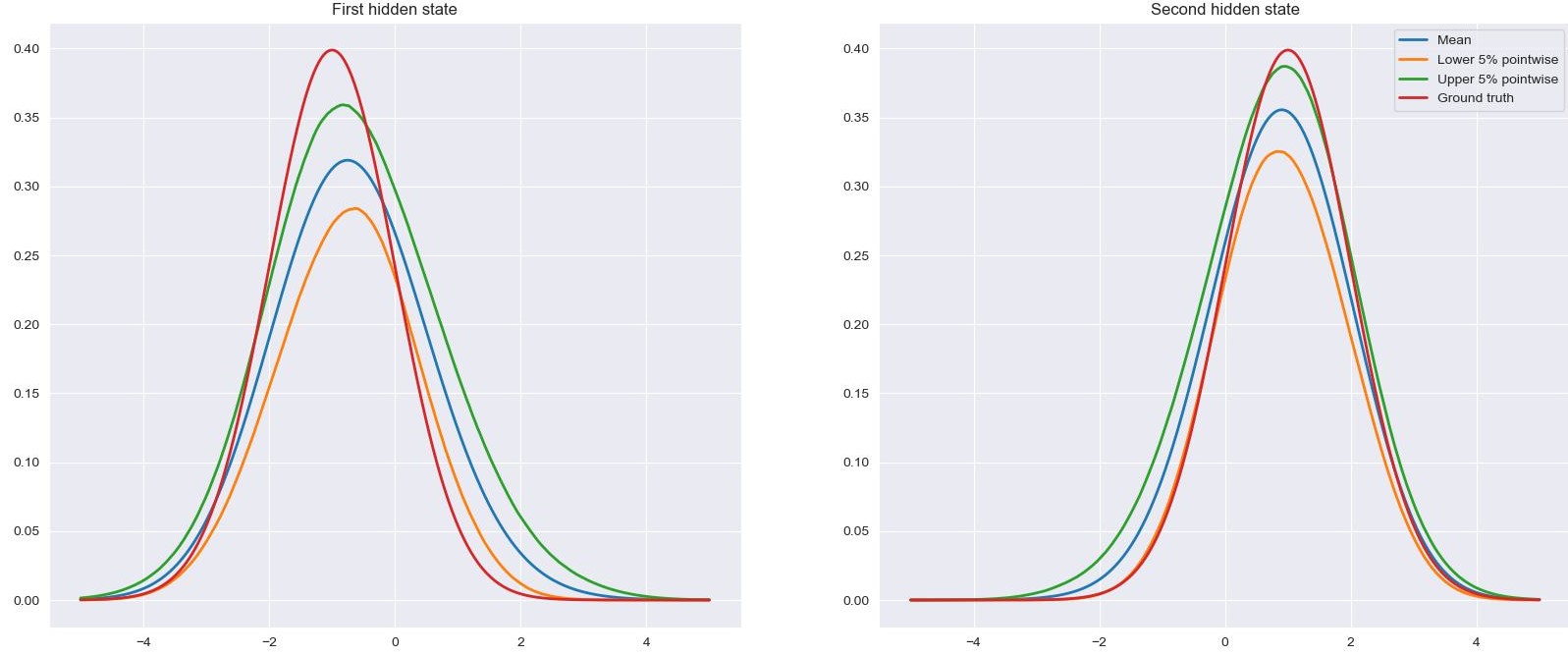} 
\includegraphics[width=\linewidth]{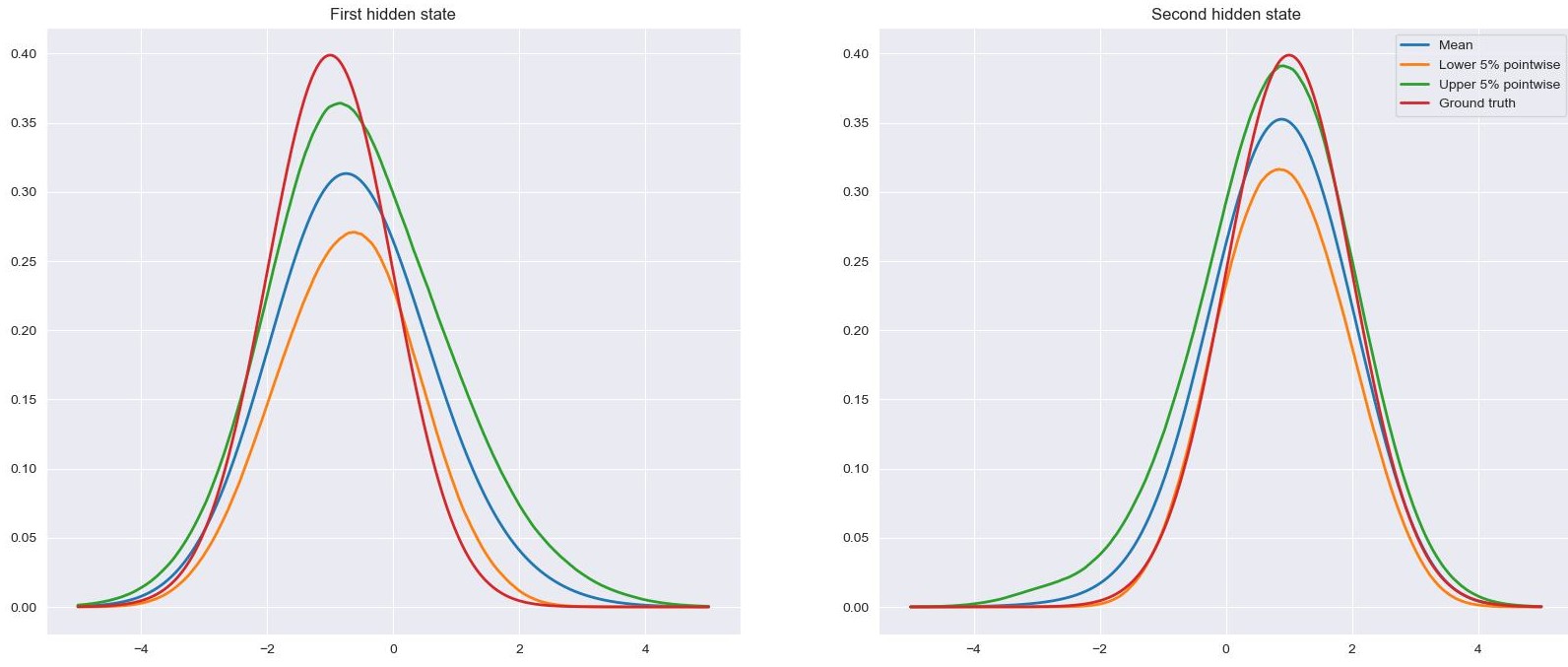} 
  \caption{Cut posterior mean and pointwise 90\% credible bands for emissions with $n=1000$ and four bins. Top: $C=10$ interior iterations; Bottom: $C=100$ interior iterations.}
  \label{pi2plots:minichain} 
\end{figure}

\begin{figure}
\centering
\includegraphics[width=\linewidth]{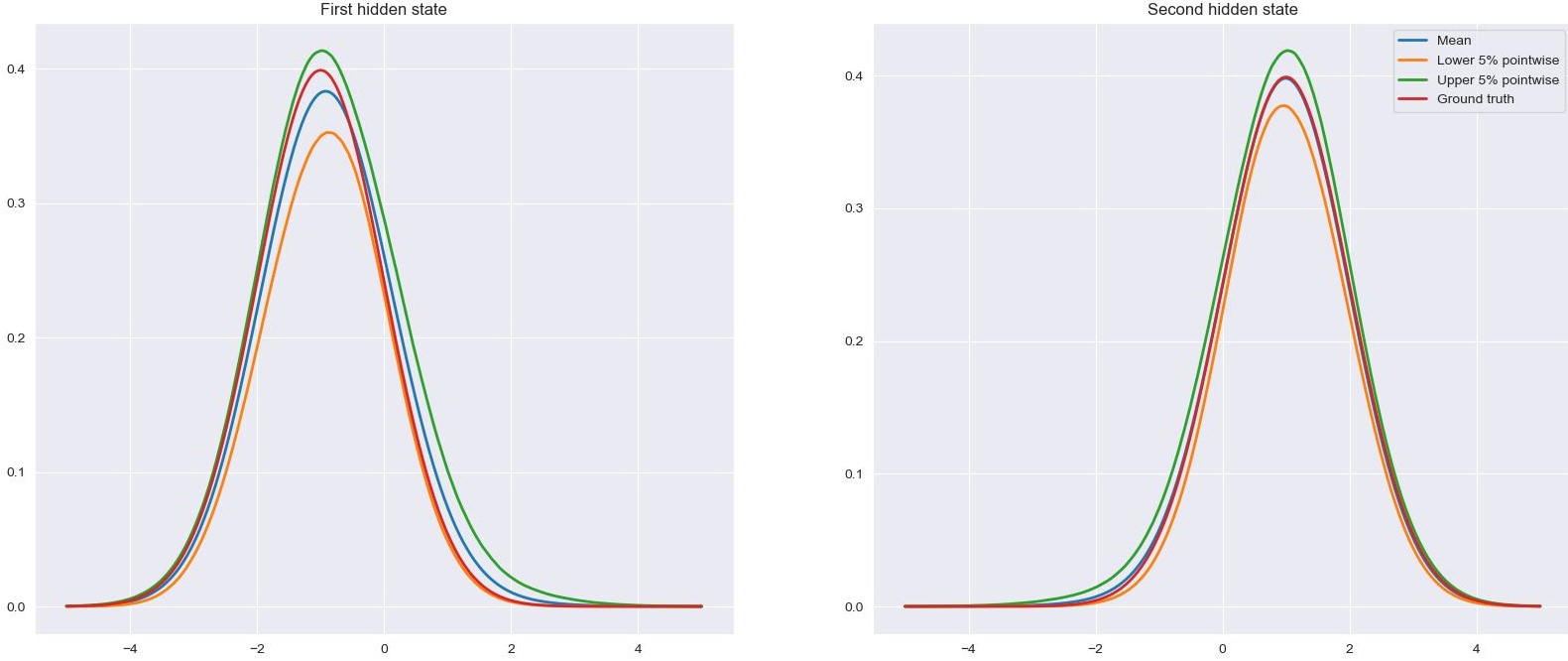} 
\includegraphics[width=\linewidth]{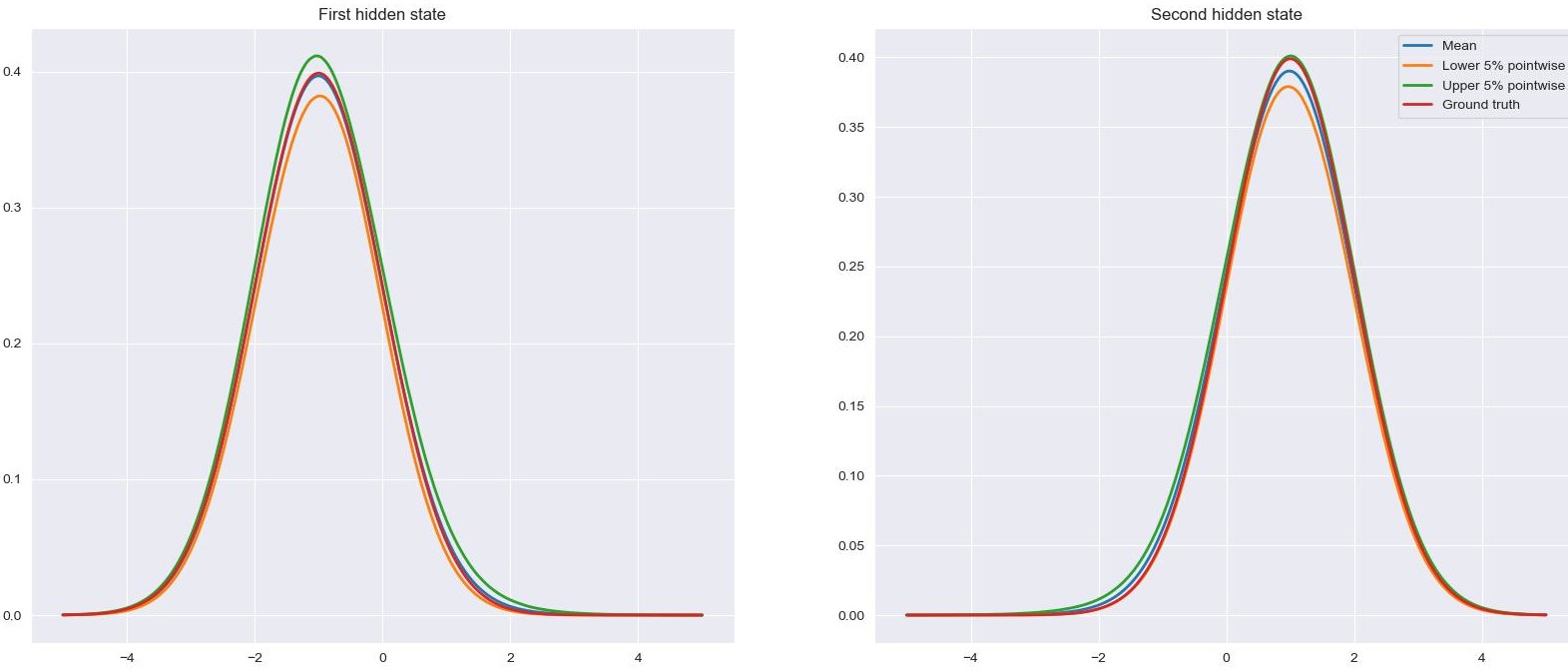} 
  \caption{Posterior mean and pointwise 90\% credible bands for emissions, with $n=2500$ (top) and $n=10000$ (bottom) under fully Bayesian posterior with prior $\Pi^\prime$.}
  \label{emissionsfullbayes} 
\end{figure}

\FloatBarrier

\section{Proofs of main results}\label{proofs}

We first prove Proposition \ref{scoreconvergence}.
\subsection{Proof of Proposition \ref{scoreconvergence}: Convergence of the scores}\label{scoreconvproof}

Throughout this section, \textbf{Assumptions} \ref{identifiability}-\ref{binsrefine} are assumed to hold.

\vspace{1ex}
Recall that $$\mathcal{H}_r = \{ h \in L^2(f^*_r\dd x),\, \int h(x) f^*_r\dd x=0\}$$ and $$\mathcal H_{r,M} =\{ h  = \sum_{m=1}^{\kappa_M}  \alpha_m \mathds{1}_{I_m}(y), \, \alpha_m \in \mathbb R, \, \int h f_{r,M}^*(y)dy =0 \}.$$ Define $\mathcal{P}$ to be the space spanned by the nuisance scores $\{H(h):h=(h_r)_{r\leq R},h_r\in \mathcal{H}_r \}$ in the full semiparametric model, and also $\mathcal{P}_M$ to be the space spanned by the nuisance scores $\{H_M(h_M):h_M=(h_{r,M})_{r\leq R},h_{r,M}\in\mathcal{H}_{r,M}\}$ in the $\mathcal{M}(\mathcal{I}_M)$. Here $H(h) = (H_{r}(h_r):r\in[R])$ with $H_{r}$ as in Equation \eqref{Scoreh}, and also $H_M(h_M) = (H_{r,M}(h_r):r\in[R])$ with $H_{r,M}$ as in Equation \eqref{ScorehM}. Denote also $\mathcal G_M^l = \sigma( Y^{(M)}_t, t\leq l)$  and $\mathcal F^j = \sigma( Y_t, t\leq l)$, for $l\in \mathbb Z$.

\vspace{1ex}
To prove Proposition \ref{scoreconvergence}, we prove convergence of the scores in  $\mathcal{P}_M$ to scores in $\mathcal{P}$ and we prove convergence of the score functions $S_{Q^*}^{(M)}$ to $S_{Q^*}$. Recall that the scores $S_{Q^*}$ in the model with known emissions are given by \eqref{ScoreQ}, and by analogy the $S_{Q^*}^{(M)}$ associated to data $Y^{(M)}$ are given by \eqref{ScoreQM}:
   For $l\in\mathbb{Z}$, we define the $\sigma$- algebras $\mathcal{G}^l=\sigma(Y_{-\infty:l})$ and $\mathcal{G}_M^l=\sigma(Y^{(M)}_{-\infty:l})$. 

\begin{proof}[Proof of Proposition \ref{scoreconvergence}]

Write $\mathcal{A}_M$ for the orthogonal projection onto $\mathcal{P}_M$ and $\mathcal{A}$ for the orthogonal projection onto $\mathcal{P}$. 
Define
$$\Tilde{S}_{Q^*}^{(M)}:=S_{Q^*}^{(M)}-\mathcal{A}_MS_{Q^*}^{(M)}\quad, \text{and} \quad \Tilde{S}_{Q^*}:={S}_{Q^*}-\mathcal{A}{S}_{Q^*}.$$
To prove Proposition \ref{scoreconvergence}, it is enough to prove that  $\Tilde{S}_{Q^*}^{(M)}$ converges to $\Tilde{S}_{Q^*}$. A major  difficulty here is to prove the convergence of $\mathcal{A}_MS$ to $\mathcal{A}S$ in $L^2(\mathbb P_*)$, with $S\in L^2(\mathbb{P}_*)$, because contrary-wise to what happens for mixture models in \cite{Gassiat2018} the sets $\mathcal P_M$ are not embedded. 

\vspace{1ex}
In Lemma \ref{knownnuisanceconv} we prove using a martingale argument that $S_{Q^*}^{(M)}$ converges to $S_{Q^*}$ in $L^2(\mathbb P_*)$ as $M$ goes to infinity. 
Then 
\begin{align*}
   \lVert \Tilde{S}_{Q^*}^{(M)}- \Tilde{S}_{Q^*} \rVert_{L_2} 
   &\leq \lVert S_{Q^*}^{(M)} - {S}_{Q^*} \rVert_{L_2} + \lVert \mathcal{A}_M({S}_{Q^*}-S_{Q^*}^{(M)})\rVert_{L^2} + \lVert (\mathcal A-\mathcal{A}_M)S_{Q^*}\rVert_{L^2}\\
   & =  \lVert (\mathcal A-\mathcal{A}_M)S_{Q^*}\rVert_{L^2}+o(1).
\end{align*}
We now prove that $ \lVert (\mathcal A-\mathcal{A}_M)S\rVert_{L^2} =  o(1)$,  for all $S \in L^2(\mathbb P_*)$. As mentioned earlier, the non trivial part of this proof comes from the fact that the sets 
$\mathcal P_M$ are not embedded. Hence we first prove that the set  $\mathcal P_M$ \textit{is close} to the set $\tilde{\mathcal P}_M$, where 
$$\tilde{\mathcal P}_M = \{H(h_M), \, h_M = (h_{r,M})_{r\leq R},h_{r,M}\in\mathcal{H}_{r,M}\}.$$
More precisely we prove in Lemma \ref{projectionconv} below  that $\mathcal{A}_MS-\tilde{\mathcal A}_M S=o_{L^2}(1)$, where $\tilde{\mathcal A}_M $ is the projection onto  
$\tilde{\mathcal P}_M$.  Since the partitions are nested, the $\tilde{\mathcal{P}}_M$ are nested, and reasoning as in  \cite{Gassiat2018}  we obtain that  $(\tilde{\mathcal{A}}_MS)_{M\in\mathbb{N}}$ is Cauchy, and so converges to some $\tilde{\mathcal{A}}S$. Then, arguing as in \cite{Gassiat2018}, we identify $\tilde{\mathcal{A}}S$ with $\mathcal{A}S$ in Lemma \ref{lemma:Ptildeconv}, by first arguing that $\tilde{\mathcal{A}}$ is a projection onto a subspace of $\mathcal{P}$, and then showing that elements of $\mathcal{P}$ are well approximated by elements of $\tilde{\mathcal{P}}$ by approximating the corresponding $h_r\in L^2_0(f_r^*)$ by histograms $(h_{r,M})_{M\in\mathbb{N}}$.  This terminates the proof of Proposition \ref{scoreconvergence}.
\end{proof}

\begin{lemma}\label{projectionconv}
For any element $S\in L^2(\mathbb{P_*})$, 
$$\lvert\mathcal{A}_MS - \tilde{\mathcal{A}}_MS\rvert_2  =o(1).$$
\end{lemma}
\begin{proof}

As will appear in the proof a key step  is Lemma \ref{Deconvolution} which says that for any sequence $h_M$ bounded in $L^2$, if $H_M(h_M) = o_{L^2}(1)$ then $h_M = o_{L^2}(1)$. Lemma \ref{lemma:HmH} also applies for $L^2$ bounded sequences, and implies that for any bounded sequence $h_M$, $(H_M-H)(h_M) = o_{L^2}(1)$.

\vspace{1ex}
To prove that $\mathcal{A}_MS - \tilde{\mathcal{A}}_MS = o_{L^2}(1)$, we decompose 
$$ \mathcal{A}_MS - \tilde{\mathcal{A}}_MS  = \mathcal{A}_MS -\tilde{\mathcal{A}}_M\mathcal{A}_MS + \tilde{\mathcal{A}}_M\mathcal{A}_MS  - \tilde{\mathcal{A}}_MS.$$
We first prove that 
$\tilde{\mathcal{A}}_M\mathcal{A}_MS - \mathcal{A}_MS = o_{L^2}(1)$
and then that $\tilde{\mathcal{A}}_M(\mathcal{A}_MS -S) = o_{L^2}(1)$. 
Define $h_M$ such that  $\mathcal{A}_MS=H_M(h_M)$, we have that $\|H_M(h_M)\|_2 \leq \|S\|_2$. Also, by Lemma \ref{lemma:HmH} and using the fact that orthogonal projections onto a space minimize the $L^2$ distance to that space (which includes $H(h_M)$)
$$\lVert \tilde{\mathcal{A}}_M\mathcal{A}_MS-\mathcal{A}_MS\rVert_2\leq \lVert (H-H_M)(h_M)\rVert.$$

\vspace{1ex}
Suppose now that, along a subsequence, $\lVert {h}_M \rVert_2\rightarrow\infty$. Then along this subsequence, $\frac{H_M({h}_M)}{\lVert {h}_M\rVert_2}=H_M(\frac{{h_M}}{\lVert {h}_M\rVert_2})\rightarrow 0$ in $L^2$ and we can apply Lemma \ref{Deconvolution} to conclude that $ \frac{{h_M}}{\lVert {h}_M\rVert_2} \rightarrow 0$ in $L^2$ along this subsequence, when in fact this subsequence has constant unit $L^2$ norm. We thus get a contradiction and conclude that $\sup_M\lVert {h}_M\rVert_2$ is bounded. Then by Lemma \ref{lemma:HmH}, we have that
$\mathcal{A}_M S = H_M(h_M) = H(h_M) + o_{L_2}(1)$. Thus $\lVert (H_M-H)(h_M)\rVert_2\rightarrow 0$ and $\lVert\tilde{\mathcal{A}}_M\mathcal{A}_MS-\mathcal{A}_MS\rVert_2 = o(1)$. 

\vspace{1ex}

We now control $\tilde{\mathcal{A}}_M(\mathcal{A}_MS -S)$. 
Considering $\tilde{\mathcal{A}}_M(S-\mathcal{A}_MS) = H(\tilde h_M)$ for some $\tilde h_M = (\tilde h_{r,M}, r\in [R])$ and $\tilde h_{r,M} \in L_0^2(f_r^*)$ in model $M$, we have
\begin{align*}
    \lVert\tilde{\mathcal{A}}_MS-\tilde{\mathcal{A}}_M\mathcal{A}_MS\rVert_{L^2}^2 &=\langle S-\mathcal{A}_MS,\tilde{\mathcal{A}}_M(S-\mathcal{A}_MS)\rangle\\
    & = \langle S-\mathcal{A}_MS, H(\tilde{h}_M)\rangle = \underbrace{\langle S-\mathcal{A}_MS, H_M(\tilde{h}_M)\rangle}_{=0 \textrm{   as }(S-\mathcal{A}_MS)\perp H_M(\tilde{h}_M)} +R_M
\end{align*}
where
$$\lvert R_M\rvert \leq \lVert S\rVert_2 \lVert (H-H_M)(\tilde{h}_M)\rVert_2 \rightarrow 0$$
with the convergence following as before using Lemma \ref{Deconvolution} and the boundedness of $\lVert H(\tilde{h}_M)\rVert_{L_2}$ by $\lVert S\rVert_{L_2}$. Thus
$\tilde{\mathcal{A}}_M(\mathcal{A}_MS -S) =  o_{L_2} (1)$
and Lemma \ref{projectionconv} is proved. 
\end{proof}


\begin{lemma}\label{Deconvolution}
Fix $r\in\{1,\dots, R\}$. Let $h_M = (h_{r,M}, r\in [R])$ be a sequence of step functions on the partitions of model $M$ with $\lVert h_M\rVert = \sum_{r}\lVert h_{r,M} \rVert _{L^2}\leq 1$ and $h_{r,M} \in \mathcal{H}_{r,M}$. Then
$$H_M(h_M)\stackrel{L^2(\mathbb{P}_*)}{\rightarrow}0 \implies h_M\stackrel{L^2(\mathbb{P}_*)}{\rightarrow}0.$$
\end{lemma}

\begin{proof}
Since $H_M(h_M)\rightarrow 0$,  by Lemma \ref{deconvolvehelp}, we have for all $k \geq 0$, 
$$D_k(h_M):=\sum_{r=1}^R\sum_{k'= -k}^0h_{r,M}(y_{k'})\mathbb{P}_*(X_{k'}=r|Y_{-k:0}) = o_{L_2}(1).$$
For shortness sake we write $g^* = g_{Q^*,\mathbf F^*}^{(2)}$ and we will abuse notation and write $y_j:=y_{-j}$ for $j\in\mathbb{N}$, so $y_1$ is $y_{-1}$ etc. Expanding $D_1(h_M)$, we thus have that
\begin{align*}
    &\sum_{r}h_{r,M}(y_0)f_r^*(y_0)\sum_s\dfrac{p_s^*Q_{sr}^*f_s^*(y_1)}{g^*(y_0,y_1)}+\sum_r\dfrac{h_{r,M}(y_1)f_r^*(y_1)}{g^*(y_0,y_1)}p_r^*\sum_sf_s^*(y_0)Q_{rs}^* \stackrel{L^2(\mathbb P_*)}{\longrightarrow} 0.
\end{align*}

\vspace{1ex}
Since for any function $H$, 
$$ \mathbb E_*\left[ \frac{ H^2(Y_{0}, Y_{1})}{g^*(Y_{0},Y_1)^2}\right] = \int_{[0,1]^2}\frac{ H^2(y_{0}, y_{1})}{g^*(y_{0},y_1)} dy_1dy_0 \geq \frac{ \|H\|_2^2}{ \|g^*\|_\infty} $$
 and since $\|g^*\|_\infty\leq \max_r \|f_r^*\|_\infty<\infty$ by continuity, we obtain
\begin{equation}\label{condexpconv}
    \sum_{r}h_{r,M}(y_0)f_r^*(y_0)\sum_s p_s^*Q_{sr}^*f_s^*(y_1)+\sum_r h_{r,M}(y_1)f_r^*(y_1) p_r^*\sum_sf_s^*(y_0)Q_{rs}^* \stackrel{L^2(\dd y^2)}{\rightarrow} 0 .
\end{equation}
Note further that the family of functions of $y_1$ given by
$y_1\mapsto (\sum_s p_s^*Q_{sr}^*f_s^*(y_1))_{r=1}^k$
is linearly independent.
Consider now a partition $\bar{I}_1,\bar{I}_2,\dots,\bar{I}_R$ of $[0,1]$ such that
$$F_{rj}= \int_{\bar I_j}f_r^*(y_1) \dd y_1 ,\quad r,j\in\{1,\dots R\},$$
defines a matrix $F=(F_{rj}^*)$ of rank $R$. Such a partition always exists by Lemma \ref{lemma:partition}. Define also
$$A_{rj}=\int_{I_j}\sum_s p_s^*Q_{sr}^*f_s^*(y_1) \dd y_1,\quad r,j\in\{1,\dots R\},$$
with $A=(A_{rj})$, define $D_p=\text{diag}(p_1^*,\dots,p_R^*)$ (which is invertible under \textbf{Assumption} \ref{ass:Qbound}) and recall $Q^*=(Q_{sr}^*)$. Then we can write $A^T=F^TD_pQ^*$ and note that $A$ also has rank $R$. Now, integrating the $y_1$ coordinate out of Equation \eqref{condexpconv}, we get for all $j=1,\dots k$
\begin{align*}
    \sum_r \underbrace{h_{r,M}(y_0)f_r^*(y_0)}_{\tilde{h}_{r,M}}A_{rj} &= -\sum_r \underbrace{\int_{\bar I_j}h_{r,M}(y_1)f_{r}^*(y_1) p_r^*\dd y_1}_{B_M(r,j)} \sum_{s} f_s^*(y_0)Q_{rs}^* \\ &\hspace{5ex}+ o_{L^2(\dd y_0)}(1) . \nonumber
\end{align*}
Let $\Tilde{h}_M=(\tilde{h}_{r,M})_{r=1}^R$, $h_M=(h_{r,M})_{r=1}^R$ and $B_M=(B_M(r,j))_{rj}$. Then the preceding display may be rewritten as
\begin{equation}\label{tildeh}
    \Tilde{h}_M(y_0)=\underbrace{-(A^T)^{-1}B_M^TQ^*}_{B_{0,M}} \mathbf f^*(y_0) + o_{L^2(\dd y_0)}(1).
\end{equation}


\vspace{1ex}
Now from the definition, $(B_M)_{M}$ is a bounded sequence in a finite dimensional space and so converges to some $B$ along a subsequence. Working on this subsequence, we then get the limit 
$$\Tilde{h}_M(y_0)\stackrel{L^2(\dd y_0)}{\rightarrow} \underbrace{-(A^T)^{-1}B^TQ^*}_{B_0}\mathbf f^*(y_0) := \Tilde{h}(y_0). $$

\vspace{1ex}
Consider now $D_2(h_M)$, wish again vanishes by assumption and Lemma \ref{deconvolvehelp}, and take the limit in $M$ along the subsequence where we just established convergence. As before, abuse notation by replacing $Y_{-k}$ by $Y_{k}$ for $k\in\mathbb{N}$. Expanding and multiplying through by joint marginals we get at the limit

\begin{align}\label{3obs}
    &\sum_{r} (B_0\mathbf f^*(y_0))_r \sum_{s_1,s_2}p_{s_2}^*Q_{s_2s_1}^*Q_{s_1r}^*f_{s_2}^*(y_2)f_{s_1}^*(y_1) \nonumber
    \\& +\sum_{r} (B_0\mathbf f^*(y_1))_r \sum_{s_1,s_2}p_{s_2}^*Q_{s_2 r}^*Q_{r s_1}^*f_{s_1}^*(y_0)f_{s_2}^*(y_2)\nonumber\\ 
    &\quad+\sum_{r} (B_0\mathbf f^*(y_2))_r \sum_{s_1,s_2}p_r^*Q_{r s_1}^*Q_{s_1s_2}^*f_{s_1}^*(y_1)f_{s_2}^*(y_0) =0.
\end{align}
 Rewriting \eqref{3obs} into matrix products we obtain

\begin{align*}
&\sum_{s_1}[\mathbf f^*(y_2)^T D_p Q^*]_{s_1} f_{s_1}^*(y_1)(Q^*B_0\mathbf f^*(y_0))_{s_1} \\ & +\sum_{s_1}[B_0 \mathbf f^*(y_1)]_{s_1}(Q^*\mathbf f^*(y_0))_{s_1}(\mathbf f^*(y_2)^TD_pQ^*)_{s_1} \\
&\quad +\sum_{s_1}[\mathbf f^*(y_2)^T B_0^TD_p Q^*]_{s_1} f_{s_1}^*(y_1)(Q^*\mathbf f^*(y_0))_{s_1} =0 .
\end{align*}
By linear independence of the the set of functions $\{f_{s}^*(y_1)\}_{s=1,\dots,R}$ we have that for all $s_1$ the coefficients vanish, so that for all $s_1$ we get
\begin{align*}
&[\mathbf f(y_2)^T D_p Q^*]_{s_1} (Q^*B_0\mathbf f(y_0))_{s_1}
+\sum_{r'}B_{0r's_1}(Q^*\mathbf f^*(y_0))_{r'}(\mathbf f^*(y_2)^TD_pQ^*)_{r'}  \\
 & \qquad +[\mathbf f^*(y_2)^T B_0^T D_pQ^*]_{s_1} (Q^*\mathbf f^*(y_0))_{s_1} =0.
\end{align*}
We can now repeat the argument but instead viewing the expression as a linear combination of the linearly independent set of functions $\{f_{s}^*(y_2)\}_{s=1,\dots,R}$. Expanding what precedes for each $s_1$, we get
\begin{align*}
&\sum_rf_r^*(y_2) p_r^* Q_{rs_1}^* (Q^*B_0\mathbf f^*(y_0))_{s_1}+\sum_{r,r'} B_{0r's_1}(Q^*\mathbf f^*(y_0))_{r'}(f_r^*(y_2))p_r^*Q_{rr'}^* \\
 & \qquad + \sum_r  f_r^*(y_2) (B_0^T D_pQ^*)_{r s_1} (Q^*\mathbf f^*(y_0))_{s_1} =0, 
\end{align*}
so that for each $r,s_1$ we get that
\begin{align*}
& p_r^* Q_{rs_1}^* (Q^*B_0\mathbf f^*(y_0))_{s_1}
+\sum_{r'}B_{0r's_1}(Q^*\mathbf f^*(y_0))_{r'}p_r^*Q_{rr'}^* +  (B_0^T D_pQ^*)_{r s_1} (Q^*\mathbf f^*(y_0))_{s_1} =0.
\end{align*}
Expanding in terms of the linearly independent functions $\{(Q^*\mathbf f^*)_r(y_0)\}_{r=1,\dots,R}$ we get

\begin{align*}
    & \sum_{r'} p_r Q_{rs_1}^* (Q^*B_0Q^{*-1})_{s_1r'}(Q^*\mathbf f^*)_{r'}(y_0)
+\sum_{r'}B_{0r's_1}(Q^*\mathbf f^*(y_0))_{r'}p_r^*Q_{rr'}^*\\
&\hspace{5ex}+  \sum_{r'}\mathds{1}_{s_1=r'}(B_0^TD_p Q^*)_{r r'} (Q^*\mathbf f^*(y_0))_{r'} =0,
\end{align*}
which then gives for each $r,r',s_1$ that

$$p_r^* Q_{rs_1}^* (Q^*B_0Q^{*-1})_{s_1r'}+B_{0r's_1}p_r^*Q_{rr'}^* + \mathds{1}_{s_1=r'}(B_0^TD_p Q^*)_{r r'} = 0. $$

In particular, setting $s_1=r'$ we get that for all $r,r'$,
$$p_r^*Q_{rr'}^*(Q^*B_0Q^{*-1})_{r'r'}+B_{0r'r'}p_r^*Q_{rr'}^* + (B_0^TD_p Q^*)_{r r'} =0.$$
By \textbf{Assumption} \ref{ass:Qbound}, $Q_{ij}^*>0 \hspace{1ex}\forall i,j$ and so we can divide through to get $$\frac{(B_0^TD_p Q^*)_{r r'}}{Q_{rr'}^*} =p_r^*a(r') , $$
for some $a$ depending on $r'$ alone. Write $D_a$ for the diagonal matrix whose diagonals are the $a(r')$. Then
$$(B_0^TD_p Q^*)_{r r'} = p_r^*Q_{rr'}^*a(r') =(D_pQ^*D_a)_{rr'}.$$


Equivalently we can write $B_0^T=(D_pQ^*)D_a(D_pQ^*)^{-1}$. Recalling the earlier definition of $B_0$, we had
$$B_0=-(A^T)^{-1}B^TQ^*,$$
with each column of $B$ expressible as the integral over $\bar I_j$, with respect to $y_1$, of $D_p\Tilde{h}(y_1)=D_pB_{0}\mathbf f^*(y_1)$, and thus expressible as
$ B = D_pB_0F.$ We can then substitute into the previous display to see that
$${B_0=-(A^T)^{-1}F^TB_0^TD_pQ^*},$$
but the expression of $A$ as $A^T=F^TD_pQ^*$ allows us to write
$$B_0=-Q^{*-1}D_p^{-1}(F^T)^{-1}F^TB_0^TD_pQ^*=-Q^{*-1}D_p^{-1}B_0^TD_pQ^*=-D_a.$$
This means that $B_0^T=-D_a=(D_pQ^*)D_a(D_pQ^*)^{-1}$. So
$Q^{-1}D_aQ^*=-D_a$ which gives
$$a(r)Q_{rs}^*= -Q_{rs}^*a(s).$$
Taking $r=s$ and dividing through by $Q_{rr}^*$ gives $a_r=0$ for all $r$, hence $B_0=0$ and $\Tilde{h}=0$. This means $\tilde{h}_m(y_0)\rightarrow 0$ in $L^2(\dd y_0)$ and so $h_{r,M}\rightarrow 0$ in $L^2(f_r(y_0)\dd y_0 )$ and so in $L^2(\mathbb{P}_*)$ under \textbf{Assumption} \ref{densityratio}, which is what we wanted to show.
\end{proof}


\begin{lemma}\label{knownnuisanceconv}
For $S_{Q^*}^{(M)},S_{Q^*} $ as in \eqref{ScoreQM} and \eqref{ScoreQ}, for any choice of $j$, we have that
$$S_{Q^*}^{(M)}(Y_{-\infty:j}^{(M)})\rightarrow S_{Q^*}(Y_{-\infty:j}) $$
in $L^2(\mathbb{P}_*)$.
\end{lemma}
\begin{proof}
Without loss of generality we take $j=0$. We first show that we can, uniformly in $M$, truncate the series in \eqref{ScoreQM} to an arbitrary degree of accuracy. we prove that uniformly in $M$, 
\begin{equation}\label{eq:probconv:trunc}
    \sum_{j \leq -J} \mathbb{P}_*({X_{j-1}=a},{X_j=b}|\mathcal{G}^{0}_M) - \mathbb{P}_*({X_{j-1}=a},{X_j=b}|\mathcal{G}^{-1}_M)=o_{L^2(\mathbb P_*)}(1),
\end{equation}
as $J$ goes to infinity. Using Lemma \ref{lemma:scoretails}, we have  that there exists $\rho<1$ and $C>0$  such that for all $j\leq -1$, for all $M$ 
\begin{align*}
\Bigg\lVert \mathbb{P}_*({X_{j-1}=a},{X_j=b}|\mathcal{G}^{0}_M)
    -\mathbb{P}_*({X_{j-1}=a},{X_j=b}|\mathcal{G}^{-1}_M)
    \Bigg\rVert_{L^2(\mathbb{P}_*)} \leq C\rho^{-j}.
\end{align*}
and similarly with $\mathcal{G}^0,\mathcal{G}^1$ in place of $\mathcal{G}^0_M,\mathcal{G}^1_M$. Therefore for all $\epsilon >0$ there exists $J(\epsilon)>0$ for which
\begin{align*} \sum_{j\leq -J(\epsilon)} &\Bigg\lVert \mathbb{P}_*({X_{j-1}=a},{X_j=b}|\mathcal{G}^{0}_M)
   -\mathbb{P}_*({X_{j-1}=a},{X_j=b}|\mathcal{G}^{-1}_M)
    \Bigg\rVert_{L^2(\mathbb{P}_*)}\\
    & \quad + \Bigg\lVert\mathbb{P}_*({X_{j-1}=a},{X_j=b}|\mathcal{G}^{0})- \mathbb{P}_*({X_{j-1}=a},{X_j=b}|\mathcal{G}^{-1})\Bigg\rVert_{L^2(\mathbb{P}_*)} < \epsilon .
\end{align*}

\vspace{1ex}
It now suffices to prove for all $j$ that 
\begin{equation}\label{eq:probconv}
    \mathbb{P}_*({X_{j-1}=a},{X_j=b}|\mathcal{G}^{l}_M)\stackrel{L^2(\mathbb{P}_*)}{\rightarrow}\mathbb{P}_*({X_{j-1}=a},{X_j=b}|\mathcal{G}^{l}).
\end{equation}
Firstly, we show that, for $l=-1,0$, the $\mathcal{G}_M^l$ form an increasing sequence of $\sigma$-algebras $\mathcal{G}_M^l\subset\mathcal{G}_{M^\prime}^l$ for $M<M^{\prime}$, and that the limiting $\sigma$-algebra $\mathcal{G}_{\infty}^l:=\sigma\left(\bigcup_{M\in\mathbb{N}}\mathcal{G}_M^l\right)$ is equal to $\mathcal{G}^l$.
We suppress the dependence on $l$ since the proofs are the same. For the sake of presentation we code  the coarsened observations $Y_i^{(M)}$  as projection onto the left endpoint of whichever interval of the partition $\mathcal{I}_M$ the $Y_i$ is contained: $Y_i^{(M)}=e_{j}$ if $Y_i\in I_j^{(M)}=(e_j,e_{j+1})$.

\vspace{1ex}
To see that the $\mathcal{G}_M$ form a filtration, we wish to show that every set in $\mathcal{G}_M$ is also in $\mathcal{G}_{M+1}$. Every set in these $\sigma-$algebras is generated by the preimages of the left endpoints of intervals under $Y_i^M$ for some $i$. Suppose the partition $\mathcal{I}_{M+1}$ is obtained from $\mathcal{I}_M$ by splitting one bin. The events $\{Y_i^M=e_j\}$ for $j=1,\dots, M$ and $e_j$ the endpoint of bin $j$ generate $\mathcal{G}_M$. For the split bin, if we split bin $M$ in $\mathcal{I}_M$ into bins $M$ and $M+1$ in $\mathcal{I}_{M+1}$, then we can express $\{Y_i^M=e_j\}=\{Y_i^{M+1}=e_j\}$ for $j<M$ and $\{Y_i^M=e_M\}=\{Y_i^{M+1}=e_M\}\cup \{Y_i^{M+1}=e_{M+1}\}$. Thus we can express these events as events in $\mathcal{G}_{M+1}$, and by induction the argument remains true for any nested sequence of intervals.

\vspace{1ex}
To see that each $\mathcal{G}_M\subset\mathcal{G}$ it suffices to note that $\{Y_i^M=e_j\}=\{Y_i\in[e_j,e_{j+1})\}\in\mathcal{G}$. This shows that $\mathcal{G}_{\infty}\subset\mathcal{G}$. To show that $\mathcal{G}\subset\mathcal{G}_\infty$, note that the Borel $\sigma-$algebra on $[0,1]$ is generated by the sets $(t,1]$. So it suffices to check that the sets $\{Y_i>t\}$ are in $\mathcal{G}_\infty$, as these sets generate $\mathcal{G}$. But for all $t\in (0, 1)$, there exists a sequence $e_{j(t,M)}\leq t$ converging to $t$ and such that $e_{j(t,M)+1}> t$ also converges to $t$ thus $\{Y_i^M>t\} = \{Y_i \geq e_{j(t,M)+1}\}$ increases to $\{Y_i>t\}$ and $\{Y_i>t\}=\lim_M\{Y_i^M>t\} \subset \mathcal{G}_\infty$. Therefore the conditional probabilities appearing in the histogram score expression (\ref{ScoreQM}) form  a bounded martingale in $L^2$ with $\mathcal{G}_\infty^l = \mathcal{G}^l$, $l=-1,0$. 
The $L^2$ martingale convergence theorem them implies that \eqref{eq:probconv} holds for any $j,l$, and Lemma \ref{knownnuisanceconv} is proved. 
\end{proof}

The proofs of the following results, which are used in the proofs of Lemmas \ref{projectionconv} and \ref{Deconvolution} are deferred to Section \ref{sup:techproofs} of the supplement.

\begin{lemma}\label{lemma:HmH}
If  $h_M=(h_{r,M})_{r\leq R}$ with $\sup_M\sum_r\lVert h_{r,M}\rVert_{L^{2}(f_r^*)}\leq L<\infty$, then
$$H_{M}(h_M) = H(h_M) +o_{L_2(\mathbb{P}_*)}(1).$$
\end{lemma}

\begin{lemma}\label{lemma:Ptildeconv}
Let $\tilde{\mathcal{A}}_M$ $\mathcal{\tilde{P}}_M$, $\mathcal{A}$, $\mathcal{P}$ be as in the statement of Lemma \ref{projectionconv}. Then, if $\tilde{\mathcal{A}}:=\lim_{M\rightarrow\infty}\tilde{\mathcal{A}}_M$, we have $\tilde{\mathcal{A}} = \mathcal{A}$.
\end{lemma}

\begin{lemma}\label{deconvolvehelp}
Let $h_M$ be a sequence of step functions on the partitions of model $m$ with $\lVert h_M\rVert = \sum_{r}\lVert h_{r,M} \rVert _{L^{2}}\leq 1$, such that $h_{r,M}\in L^{2}_0(f^*_r)$. For any $k\geq 0$,
$$H_M(h_M)\stackrel{L^2(\mathbb{P}_*)}{\rightarrow}0 \implies \sum_{r=1}^R\sum_{k'=-k}^0h_M(y_{k'})\mathbb{P}_*(X_{k'}=r|Y_{-k:0}) \stackrel{L^2(\mathbb{P}_*)}{\rightarrow}0.$$
\end{lemma}

\subsection{Proof of Theorem \ref{inversion}}

The proof of  Theorem \ref{inversion} works by contradiction. If there exist no $C>0$ such that for all $\textbf{f}$ with $\lVert g^{(3)}_{Q,\textbf{f}}-g^{(3)}_{Q,\textbf{f}^*}\rVert_{L^1}$ sufficiently small, 
$$\frac{\lVert g^{(3)}_{Q,\textbf{f}}-g^{(3)}_{Q,\textbf{f}^*}\rVert_{L^1}}{\sum_{r=1}^R\lVert f_r-f_r^*\rVert_{L^1}} \geq C$$
then there exists a sequence of emission densities $\textbf{f}^{(n)}$ such that $\lVert g^{(3)}_{Q,\textbf{f}}-g^{(3)}_{Q,\textbf{f}^*}\rVert_{L^1}= o(1)$ and 
$$\liminf_n\frac{\lVert g^{(3)}_{Q,\textbf{f}^n}-g^{(3)}_{Q,\textbf{f}^*}\rVert_{L^1}}{\sum_{r=1}^R\lVert f_r^n-f_r^*\rVert_{L^1}}=0.$$
For $j=2,3$, $r\in[R]$  and $\textbf{f}^{(n)}$  as above,  set
$$\Delta_{n,j} = \frac{\lVert g^{(j)}_{Q,\textbf{f}^{(n)}}-g^{(j)}_{Q,\textbf{f}^*}\rVert_{L^1}}{\sum_{r=1}^R\lVert f_r^n-f_r^*\rVert_{L^1}}; \quad h_r^n=\frac{f^n_r-f_r^*}{\sum_{r=1}^R\lVert f_r^n-f_r^*\rVert_{L^1}},$$
and write $\mathbf{h}^n=(h^n_r:r\in[R])$.


\vspace{1ex}
Since $\liminf_n\Delta_{n,3}=0$, there exists a subsequence $\phi_1(n)$ such that $\Delta_{\phi_1(n),3}\rightarrow 0$, which implies also that $\Delta_{\phi_1(n),2}\rightarrow 0$ since $\Delta_{n,2}\leq\Delta_{n,3}$.  We have by assumption that $\lVert g^{(j)}_{Q,\textbf{f}^{(\phi_1(n))}}-g^{(j)}_{Q,\textbf{f}^*}\rVert_{L^1}\rightarrow 0$, and by applying Lemma \ref{consistencyemissions} we can find a further subsequence $\phi_2(\phi_1(n))$ such that $\textbf{f}^{(\phi_2(\phi_1(n)))}\rightarrow \textbf{f}^*$ in $L^1$. For notational convenience we write $\Delta_{n,j}$ for the terms $\Delta_{\phi_2(\phi_1(n)),j}$ along this subsequence, and in general any index with $n$ is now interpreted as being along this subsequence. Since $\sum_{r=1}^R\lVert f_r^n-f_r^*\rVert_{L^1}\rightarrow 0$,  $\Delta_{n,2}\rightarrow 0$ implies that
$$ \sum_{r,s}(PQ^*)_{rs}(f_{r}^*(y_1)h_s^n(y_2)+h_{r}^n(y_1)f_{s}^*(y_2)) = o_{L^1}(1),$$
where $P=\diag(p_1^*,\dots,p_R^*)$. Partition $I_1,\dots,I_R$ such that $\textbf{F}^*$, as defined in Lemma \ref{consistencyemissions}, has rank $R$. Define also $\mathbf{H}^n=(H^{n}_{ir})_{ir}=\int_{I_i}h^n_r(y)\dd y$. We have for any integrable function $G$ any any $i\in[R]$ that
$$\int_{\mathbb{R}}\int_{\mathbb{R}} \lvert G(y_1,y_2) \rvert \dd y_1 \dd y_2 \geq \int_{\mathbb{R}}\int_{I_i} \lvert G(y_1,y_2) \rvert \dd y_1 \dd y_2 \geq \int_{\mathbb{R}}\left\lvert \int_{I_i}  G(y_1,y_2) \dd y_1 \right\rvert \dd y_2 , $$
which implies in particular that
$$\sum_{r,s} (PQ^*)_{rs}(F_{ir}^*h_s^n(y_2)+H^n_{ir}f_{s}^*(y_2)) =o_{L^1}(1).$$
This gives us
$$(\mathbf{F}^*)^T(PQ^*)\mathbf{h}^n = -\mathbf{H}_n^T(PQ)\mathbf{f}^* + o_{L^1}(1).$$
The sequence $\mathbf{H}_n$ is bounded and so converges to some $\mathbf{H}_0$ along a subsequence $\phi_3(n)$ which we pass to. Then since $\mathbf{F}^*$ is of rank $R$, we may write
$$\mathbf{h}^n = \underbrace{-((\mathbf{F}^*)^T(PQ))^{-1}(\mathbf{H}^{*})^T(PQ)}_{=B_1}\mathbf{f}^* + o_{L^1}(1),$$
and so $h_{r}^n\rightarrow (B_1\mathbf{f}^*)_{r}$ in $L^1$ for each $l=1,\dots, R$. Next, since we have that $\Delta_{n,3}\rightarrow 0$ in $L^1$, we get
$$\sum_{r,s,t}(PQ)_{rs}Q_{st}\left(f_{r}^*(y_1)f_{s}^*(y_2)h_t^n(y_3) + f_{r}^*(y_1)h_s^n(y_2)f_{t}^*(y_3)+h_{r}^n(y_1)f_{s}^*(y_2)f_{t}^*(y_3)\right)=o_{L_1(1)}.$$
Replacing $h_{r}^n$ by its limit $(B_1\mathbf{f}^*)_{r}$, we get that for Lebesgue-almost-all $y_1,y_2,y_3$
$$\sum_{r,s,t}(PQ)_{rs}Q_{st}\left(f_{r}^*(y_1)f_{s}^*(y_2)(B_1\mathbf{f}^*)_t(y_3) + f_{r}^*(y_1)(B_1\mathbf{f}^*)_s(y_2)f_{t}^*(y_3)+(B_0\mathbf{f}^*)_r(y_1)f_{s}^*(y_2)f_{t}^*(y_3)\right)=0$$
Under \textbf{Assumption} \ref{densityratio}, we have continuity of $f^*$ and so in fact the above display holds for all $y_1,y_2,y_3$.  This expression is of the form of Equation (\ref{3obs}) in the proof of Lemma \ref{Deconvolution}, and the same proof techniques show that $B_1=0$ and hence $\lVert \mathbf{h}^n\rVert = o_{L^1}(1)$. This yields the desired contradiction as $\sum_{r=1}^R\lVert {h}^n_r\rVert_{L^1}=1$ for all $n\in\mathbb{N}$. We conclude that no subsequence can exist for which $\Delta_{n,3}\rightarrow 0$, and hence that $\liminf_{n\rightarrow\infty} \Delta_{n,3}>0$, which terminates the proof of Theorem \ref{inversion}.

\vspace{1ex}
The proof of Theorem \ref{inversion} uses the following result which we prove in Section \ref{sup:sec:smoothproof} of the supplement.

\begin{lemma}\label{consistencyemissions}
Consider a sequence of densities $\mathbf f^{(n)}$ such that $g^{(3)}_{Q^*,\mathbf f^{(n)}}\stackrel{L^1(\dd y^3)}{\rightarrow} g^{(3)}_{Q^*,\mathbf f^*}$, then $\mathbf f^{(n)}\stackrel{L^1(\dd y)}{\rightarrow}\mathbf f^*$ along a subsequence.
\end{lemma}

\subsection{Proof of Theorem \ref{smoothingcontract}}\label{smoothingproof}
To prove Theorem \ref{smoothingcontract}, we make use of Lemma \ref{MLEinfluence} which allows us to approximate the MLE by an estimator which does not depend on the observation for which we control the respective smoothing probability. Intuitively, this result says that a single observation does not influence our MLE too much.

\begin{proof}[Proof of Theorem \ref{smoothingcontract}]
The starting point of the proof is Proposition 2.2 of \cite{de2017consistent}. Let $\theta = (Q, \mathbf f)$ and $\| Q - Q^*\| \leq M_n/\sqrt{n}$ so that $Q_{i,j} \geq \bar q/2>0$ with $\bar q = \min_{i,j}Q^*_{i,j}<\frac{1}{2}$, under \textbf{Assumption} \ref{ass:Qbound}. Set $\rho = 1 - \bar q/(1 - \bar q) $, then 
\begin{align*}
  \Delta_k &:=  \|\mathbb P_\theta( X_k = \cdot| Y_{1:n})  - \mathbb P_*( X_k = \cdot | Y_{1:n}) \|_{TV} \leq \frac{C^*}{ \bar q }   \left[ 
\rho^k\|\mu_Q - \mu^*\|  + \|Q - Q^*\|\right. \\ 
 & \quad \left. + \sum_{\ell=1}^n \frac{\bar q \rho^{|\ell-k|}}{c_*(y_\ell)}\max_s | f_s(y_\ell) - f_s^*(y_\ell)|  \right],
\end{align*} 
where  $c^*(y)=\min_{r\in[R]}\sum_{s\in[R]}Q^*_{rs}f^*_s(y)\geq \bar{q}\sum_s f_s^*(y)$  and $C^*$ is a constant depending only on $Q^* , \mathbf f^*$. Throughout the proof we use $C^*$ for a generic constant depending only on $Q^*, \mathbf f^*$.
 First for any $K_n \rightarrow \infty$, if 
$$B_n= \{\theta,\,  \lVert Q^* - \hat{Q}\rVert\leq \frac{K_n}{\sqrt{n}}; \, \max_s\lVert f_s^* - \hat{f}_s\rVert_{L^1} \leq K_n\epsilon_n \}\cap \{ \max_r\|f_r\|_2 \leq e^{ \gamma n \epsilon_n^2 } \}$$
for $\gamma$ defined in  \eqref{smoothingcond}, then Theorem \ref{L1contractemission} together with condition \eqref{smoothingcond} gives $\Pi_{cut}(B_n^c|Y_{1:n} ) = o_{\mathbb P_*}(1)$. Recall also that $n\epsilon_n^2 \rightarrow \infty$ so that for $z_n>0$
$$ \Pi_{cut}( \Delta_k > z_n \epsilon_n | Y_{1:n} ) \leq \Pi_{cut}( \{\Delta_k > z_n \epsilon_n\}\cap B_n | Y_{1:n} ) +o_{\mathbb P_*}(1).$$

Since $\lVert \dd\Pi_1(Q|Y_{1:n}) -\phi_n(Q|Y_{1:n})\dd Q \rVert_{TV}\rightarrow 0$ where $\phi_n(Q|Y_{1:n})$ is a Gaussian density centred at the MLE of variance $\Tilde{J}^{-1}/n$, we have 
$ \|\Pi_{cut} - \Tilde{\Pi}_{cut}\|_{TV} = o_{\mathbb P_*}(1) $, where
 $\Tilde{\Pi}_{cut}$ is the measure $\phi_n(Q|Y_{1:n})\dd Q\dd\Pi_2(\textbf{f}|Q,Y_{1:n})$.
Moreover on $B_n$,
\begin{equation}\label{smoothingG}
\Delta_k \leq \frac{C^* K_n }{\sqrt{n} } 
 + \sum_{\ell=1}^n \frac{\bar q \rho^{|\ell-k|}}{c_*(y_\ell)}\max_s | f_s(y_\ell) - f_s^*(y_\ell)| .
 \end{equation}
To prove Theorem \ref{smoothingcontract}, we thus need to control the  sum in \eqref{smoothingG} under the posterior distribution.
 In \cite{abraham2021multiple}, the authors use either a control in sup-norm of $\hat f_s- f_s^*$ or a split their data into a \textit{training set} used to estimate  $\mathbf f^*$, and a \textit{test set} used to estimate the smoothing probabilities. We show here that with the Bayesian approach we do not need to split the data in two parts nor do we need a sup-norm control. We believe that our technique of proof might still be useful for other approaches.

\vspace{1ex}
 We split the  sum in \eqref{smoothingG}  into $|\ell-k| \leq T_{n}$ and $ |\ell -k| > T_{n}$  where $T_{n}\epsilon_n  $ converges arbitrarily slowly to 0. 
We first control for all $r\in [R]$ and any $z_n>0$ going to infinity
\begin{align}\label{eq:control_smooth_sum}
    \tilde \Pi_{cut}&\left[ \left. B_n \cap \left\{\sum_{|\ell-k|\leq T_{n}} \frac{ \rho^{|\ell-k|} | f_r(y_\ell) - f_r^*(y_\ell)| }{c_*(y_\ell)} \geq z_n\epsilon_n \sum_{|\ell-k| \leq T_n} \rho^{|\ell-k|/2} \right\} \right|Y_{1:n} \right] \nonumber\\
    &
     \leq \sum_{|\ell-k|\leq T_{n}}  \tilde \Pi_{cut}\left[\left.B_n \cap \left\{  \frac{\ | f_r(y_\ell) - f_r^*(y_\ell)| }{c_*(y_\ell)} \geq z_n\epsilon_n  \rho^{-|\ell-k|/2}\right\} \right|Y_{1:n} \right] .
\end{align}
To bound \eqref{eq:control_smooth_sum}, we need to study the quantities $ \pi_{n,\ell}(z) = \Pi_{cut}\left[\left. B_n \cap A_n(z; \ell)  \right|Y_{1:n} \right] $ where  $z>0$ and
$$A_n(z; \ell) =\left\{ \frac{\lvert f_r(y_\ell)-f^*_r(y_\ell)\rvert}{\sum_sf^*_s(y_\ell)}>z \epsilon_n\right\}. $$
Write
$L_n(Y_{1:n},\theta) = \sum_s f_s(y_\ell)L_{-\ell}(\theta,s), $
where $L_{-\ell}(\theta,s)$ depends on $Y^{-\ell}_{1:n}$ and is given by
$$L_{-\ell}(\theta,s)=L(y_{1:l-1},\theta)\mathbb{P}_{\theta}(X_l=s|Y_{1:l-1})L(y_{\ell+1:n}|X_l=s,\theta).$$
Note that
\begin{align*}
    L(y_{\ell+1:n}|X_{\ell}=s,\theta) & =\sum_{s'}Q_{ss'}L(y_{\ell+1:n}|X_{\ell+1}=s',\theta) \\
    &\in \left[\frac{\bar{q}}{2}\sum_{s'}L(y_{\ell+1:n}|X_{\ell+1}=s',\theta),\sum_{s'}L(y_{\ell+1:n}|X_{\ell+1}=s',\theta)\right],
\end{align*}
and also that $\mathbb{P}_\theta(X_l=s | Y_{1:l-1})\in [\frac{\bar{q}}{2},1]$, which together implies, denoting $\bar{L}_{-\ell}(\theta)=\max_s L_{-\ell}(\theta,s)$, that
$$L_{-\ell}(\theta,s)\in \left[\frac{\bar{q}^2}{4}\bar{L}_{-\ell}(\theta), \bar{L}_{-\ell}(\theta)\right].$$
This implies that the cut posterior mass of $A_n(z,\ell)\cap B_n$ is bounded as
\begin{equation}
   \Pi_{cut}(A_n(z; \ell)\cap B_n | Y_{1:n}) \leq \dfrac{4}{\bar{q}^2}\dfrac{\int_{B_n}1_{A_n(z; \ell)}(f_r)\bar{L}_{-\ell}(\theta)(\sum_sf_s(y_\ell))\dd\Pi_2(\textbf{f}|Q)\phi_n(Q|Y_{1:n})dQ}{\int_{B_n}\bar{L}_{-\ell}(\theta)(\sum_sf_s(y_\ell))\dd\Pi_2(\textbf{f}|Q)\phi_n(Q|Y_{1:n})dQ}. 
\end{equation}

\vspace{1ex}
Applying Lemma \ref{MLEinfluence}, we also have uniformly in $\ell \in [k-T_n, k+T_n]\cap \mathbb N $, $\lVert\phi_n(Q|Y_{1:n})\dd Q-\phi_n(Q|Y^{-\ell}_{1:n})\dd Q\rVert_{TV}\rightarrow 0$ where $\phi_n(Q|Y^{-\ell}_{1:n})$ is a Gaussian centred at the estimator defined in Lemma \ref{MLEinfluence} which is a function of $Y^{-\ell}_{1:n}$ alone, also of variance $\Tilde{J}^{-1}/n$. Hence, 
with $$\tilde{\Pi}_{cut}^{(\ell)}(A|Y_{1:n})=\dfrac{\int_AL_n(Y_{1:n},\theta)\dd\Pi_2(\textbf{f})\phi(Q|Y_{1:n}^{(-\ell)})\dd Q}{D_n(\ell)}$$
with  $D_n(\ell)=\int_{B_n}\bar{L}_{-\ell}(\theta)(\sum_sf_s(y_\ell))\dcut$,
we have 
$$\max_{|\ell-k|\leq T_n} \|\tilde{\Pi}_{cut}^{(\ell)}(\cdot |Y_{1:n})- \tilde{\Pi}_{cut}(\cdot|Y_{1:n})\|_{TV} = o_{\mathbb P_*}(1).$$

\vspace{1ex}
Hence, it suffices to control each summand of \eqref{eq:control_smooth_sum} with $\tilde{\Pi}_{cut}$ replaced by $\tilde{\Pi}^{(\ell)}_{cut}$. We first focus on lower bounding the denominator $D_n(\ell)$. Writing $\Omega_n=\{\sum_sf_s(y_\ell)\geq\sum f_s^*(y_\ell)/2\}$, it is immediate that
$$D_n(\ell) \geq \int_{B_n}\bar{L}_{-\ell}(\theta)1_{\Omega_n}\dcut \left(\sum_s f_s^*(y_\ell)/2\right). $$
We will show that we can replace the above bound with a bound in probability where the integrand does not feature the indicator $1_{\Omega_n}$. This is similar in spirit to concentration results based around the framework of \cite{ghosal2007convergence} (including our own Theorem \ref{cutcontract}), where a key ingredient of the proof is to lower bound the denominator by a suitable quantity. We have that
\begin{align*}
&\mathbb{P}_*\left(\int_{B_n}\bar{L}_{-\ell}(\theta)1_{\Omega_n}\dcut) < \frac{1}{2}\int_{B_n}\bar{L}_{-\ell}(\theta)\dcut\right) \\
& = \mathbb{P}_*\left(\int_{B_n}\bar{L}_{-\ell}(\theta)1_{\Omega_n^c}\dcut \geq \frac{1}{2}\int_{B_n}\bar{L}_{-\ell}(\theta)\dcut\right) \\
& \leq 2\mathbb{E}_*\left[\mathbb{E}_*\left[\dfrac{\int_{B_n}\bar{L}_{-\ell}(\theta)1_{\Omega_n^c}\dcut }{\int_{B_n}\bar{L}_{-\ell}(\theta)\dcut}\bigg|Y^{-\ell}_{1:n}\right]\right] \\
& \leq 2\mathbb{E}_*\left[\dfrac{\int_{B_n}\bar{L}_{-\ell}(\theta)\mathbb{P}({\Omega_n^c}|Y^{-\ell}_{1:n})\dcut }{\int_{B_n}\bar{L}_{-\ell}(\theta)\dcut}\right].
\end{align*}
Now we can bound the probability in the integrand by
\begin{align*}
    \mathbb{P}_*(\Omega_n^c | Y^{-\ell}_{1:n}) & = \mathbb{P}_*\left(\sum_sf_s(y_\ell)<\frac{1}{2}\sum_s f_s^*(y_\ell)|Y^{-\ell}_{1:n}\right) \\ &=\mathbb{P}_*\left(\sum_s(f_s^*(y_\ell)-f_s(y_\ell))\geq \sum_s f_s^*(y_\ell)/2 \bigg |Y^{-\ell}_{1:n}\right) \\ & \leq 2\int\dfrac{\sum_s\lvert f_s^*(y_\ell)-f_s(y_\ell)\rvert}{{\sum_s}f_s^*(y_\ell)}f^*(y_\ell|Y^{-\ell}_{1:n})\dd y_\ell.
\end{align*}
We can bound the conditional density by
$$f^*(y_\ell|Y^{-\ell}_{1:n})\leq \sum_sf_s^*(y_\ell) \quad \text{so that}   \quad  \mathbb{P}_*(\Omega_n^c | Y^{-\ell}_{1:n}) \leq 2\sum_s\lVert f_s^*-f_s\rVert_{L_1} \leq 2R K_n\epsilon_n,
$$
which in turns implies that
\begin{align*}
\mathbb{P}_*\bigg(&\int_{B_n}\bar{L}_{-\ell}(\theta)1_{\Omega_n}\dcut < \\
& \hspace{10ex}\frac{1}{2}\int_{B_n}\bar{L}_{-\ell}(\theta)\dcut\bigg) \leq 4R K_n\epsilon_n.
\end{align*}
Thus, on this event of probability at least $1-4RK_n\epsilon_n$, we can bound the denominator as
\begin{equation}\label{Dnl}
D_n(\ell) \geq \frac{1}{4}\sum_sf_s^*(y_\ell)\underbrace{\int_{B_n}\bar{L}_{-\ell}(\theta)\dcut}_{=:M(Y^{-\ell}_{1:n})},
\end{equation}
and so, on this event, 
\begin{align*}
    \tilde{\Pi}^{(\ell)}_{cut}(A_n(z;\ell)\cap B_n | Y_{1:n}) & \leq \dfrac{16}{\bar{q}^2}\dfrac{\int_{B_n}1_{A_n(z;\ell)}(\sum_sf_s(y_\ell))\bar{L}_{-\ell}(\theta)\dcut}{\sum_sf_s^*(y_\ell)M(Y^{-\ell}_{1:n})} \\
    & \leq \dfrac{16}{\bar{q}^2}\Bigg[\underbrace{\dfrac{\int_{B_n}1_{A_n(z;\ell)}(\sum_s|f_s(y_\ell)-f_s^*(y_\ell)|)\bar{L}_{-\ell}(\theta)\dcut}{\sum_sf_s^*(y_\ell)M(Y^{-\ell}_{1:n})}}_{J_1}\\
    &+\underbrace{\dfrac{\int_{B_n}1_{A_n(z;\ell)}\bar{L}_{-\ell}(\theta)\dcut}{M(Y^{-\ell}_{1:n})}}_{J_2} \Bigg].
\end{align*}
and using $\mathbb{E}_*[J_1] = \mathbb{E}_*[\mathbb{E}_*[J_1|Y^{-\ell}_{1:n}]]$, together with the fact that $f(y_\ell|Y_{1:n}^{(-\ell)})\dd y_\ell\leq \sum_sf_s^*(y_\ell) \dd y_\ell$, as well as the bound on $\lVert f_s-f_s^*\rVert_{L^1}$ on $B_n$, we obtain 
$$\mathbb{E}_*[J_1]\leq \mathbb{E}_*\left[ \dfrac{\int_{B_n}\sum_s\lVert f_s-f_s^*\rVert_1\bar{L}_{-\ell}(\theta)\dcut}{M(Y^{-\ell}_{1:n})}\right]\leq RK_n\epsilon_n.$$
To bound $J_2$ note that $1_{A_n(z;\ell)} \leq |f_r(y_\ell)-f_r^*(y_\ell)| /(z\epsilon_n\sum_rf_r^*(y_\ell) )$ so that 
\begin{align*}
    J_2&\leq\dfrac{\int_{B_n}\lvert f_r(y_\ell)-f^*_r(y_\ell)\rvert\bar{L}_{-\ell}(\theta)\dcut}{z\epsilon_n M(Y^{-\ell}_{1:n})\sum_sf^*_s(y_\ell)} ,
\end{align*}
and bounding as before with $J_1$, we obtain 
 $\mathbb{E}_*[J_2]\leq\frac{RK_n}{z}.$
 This implies that 
 \begin{align*}
\mathbb P_* &\left[ \sum_{|\ell-k|\leq T_{n}} \Tilde{\Pi}_{cut}^{(\ell)}(A_n(\rho^{-|\ell-k|/2} z_n;\ell)\cap B_n |Y_{1:n} ) > \epsilon \right] \\
&\leq \sum_{|\ell-k|\leq T_{n}}\mathbb P_*\left[ D_n(\ell) \leq \frac{c_*(y_\ell) M(Y_{1:n}^{-\ell})}{4}\right]  + T_{n}\dfrac{16RK_n\epsilon_n}{ \epsilon\bar{q}^2} + \sum_{|\ell-k|\leq T_{n}} \frac{\rho^{|\ell-k|/2} K_n}{\epsilon z_n} \\& \lesssim \frac{T_{n} K_n\epsilon_n}{ \epsilon} + \frac{ K_n}{ z_n\epsilon } .
 \end{align*}
Recall that  $T_n \leq \epsilon^2[K_n\epsilon_n]^{-1} $ and $z_n \geq K_n/\epsilon^2$ leads to 
 \begin{equation} \label{small_ell}
 \mathbb P_* \left[ \sum_{|\ell-k|\leq T_{n}} \Tilde{\Pi}_{cut}^{(\ell)}(A_n(\rho^{-|\ell-k|/2} z_n;\ell)\cap B_n |Y_{1:n} ) > \epsilon \right]\leq c_1 \epsilon ,
 \end{equation}
 for some constant $c_1$ independent on $n$ and $\epsilon$.

\vspace{1ex}
Returning to the original goal of bounding \eqref{smoothingG}, and having now bounded the sum over ${|\ell-k| \leq T_{n}}$, we study the sum over ${|\ell-k| > T_{n}}$. For this, we will bound the sum under $\Tilde{\Pi}_{cut}$ defined previously. First note that 
\begin{align*}
    \mathbb P_*\left[ \exists \ell, \,  c_*(y_\ell) \leq e^{-L_n n\epsilon_n^2} \right] &\leq  n \int_{c_*(y)\leq e^{-L_n n\epsilon_n^2}}\sum_{s=1}^Rp^*_sf_s^*(y)dy \\ & \leq n \int_{c_*(y)\leq e^{-L_n n\epsilon_n^2}} \sqrt{c^*(y)}e^{_-L_nn\epsilon_n^2/2} \dd y \\  &\leq C^* n e^{-L_n n\epsilon_n^2/2} ,
\end{align*}
where we use that $c_*(y)\geq \sum_sp^*_sf_s^*(y) $, and that $\sqrt{c_*}$ is $L^1$ bounded by some $C^*$ from condition \eqref{tail:frtrue}. 
As soon as $L_n n\epsilon_n^2 \geq L_0 \log n$ with $L_0$ large enough, possible since $n\epsilon_n^2\gtrsim \log n$ by assumption, we have $\mathbb{P}_*\left[ \exists \ell, \,  c_*(y_\ell) \leq e^{-L_n n\epsilon_n^2} \right]=o(1)$ and we can bound $c_*(y_\ell) \geq e^{-L_n n\epsilon_n^2}$ for all $\ell$.

\vspace{1ex}
Next, we define
$$\bar A_n(\ell) =\left\{ \theta; \sum_{|\ell-k|> T_{n}}  \rho^{|\ell-k|} | f_r(y_\ell) - f_r^*(y_\ell)| \leq z_n \epsilon_n e^{-L_nn\epsilon_n^2} \right\}.$$
We write
$$ \tilde{\Pi}_{cut}(B_n\cap \bar A_n^c (\ell)|Y_{1:n}) = \frac{ \int_{B_n\cap  \bar A_n^c (\ell)} e^{\ell_n(\theta) - \ell_n(\theta^*) }\dd\Pi_2(\mathbf f) \phi_n(Q|Y_{1:n})dQ }{\int_{B_n} e^{\ell_n(\theta) - \ell_n(\theta^*) }\dd\Pi_2(\mathbf f) \phi_n(Q|Y_{1:n})dQ  } =: \frac{ N_n(\bar A_n(\ell)^c)}{ D_n }, $$
where we write the denominator as $D_n$ and write $N_n(\Bar{A}_n(\ell)^c)$ for the numerator.

\vspace{1ex}
Using Lemma \ref{ELBO} so that $\mathbb P_*[ D_n < e^{-L_nn\epsilon_n^2} ] \lesssim L_n^{-1}$ together with $\phi_n(Q|Y_{1:n}) \lesssim n^{R(R-1)/2} $, we obtain with probability greater than $1-O(1/L_n)$, with $L_n$ going arbitrarily slowly to infinity, 
\begin{align*}
    \tilde{\Pi}_{cut}(B_n\cap \Bar{A}_n(\ell)^c |Y_{1:n}) &\leq 
    e^{L_n n\epsilon_n^2} \int_{B_n}\mathds{1}_{\Bar{A}_n(\ell)^c}e^{\ell_n(\theta) - \ell_n(\theta^*) }\dd\Pi_2(\mathbf f) \phi_n(Q|Y_{1:n})dQ \\
    &\hspace{-8ex}\leq \frac{C_*  e^{2L_nn\epsilon_n^2}n^{R(R-1)/2}}{ z_n\epsilon_n } \sum_{|\ell-k|>T_n} \rho^{|\ell-k|}\int_{B_n}| f_r(y_\ell) - f_r^*(y_\ell)| e^{\ell_n(\theta)- \ell_n(\theta^*) }\dd\Pi_2(\textbf f ) dQ\\
    &\hspace{-8ex} \leq C_*  e^{3L_nn\epsilon_n^2}\sum_{|\ell-k|>T_n} \rho^{|\ell-k|}\int_{B_n}| f_r(y_\ell) - f_r^*(y_\ell)| e^{\ell_n(\theta)- \ell_n(\theta^*) }\dd\Pi_2(\mathbf f) dQ,
\end{align*}
where in the final line we again use that $\log n\lesssim n\epsilon_n^2$ to consolidate the exponential. This implies that $\tilde{\Pi}_{cut}(B_n\cap \Bar{A}_n(\ell)^c |Y_{1:n})$ is controlled under $\mathbb{P}_*$ as
\begin{align}\label{large_ell}
\mathbb P_*&\left[ \tilde{\Pi}_{cut}(B_n\cap \Bar{A}_n(\ell)^c |Y_{1:n}) \geq 2/L_n \right] \leq o(1)\nonumber \\
&+ \mathbb P_*\left[ \sum_{|\ell-k|>T_n}\rho^{|\ell-k|}\int_{B_n} | f_r(y_\ell) - f_r^*(y_\ell)|e^{\ell_n(\theta)- \ell_n(\theta^*)}d\Pi_2(\mathbf f)dQ >\frac{ e^{-3L_n n\epsilon_n^2}}{ L_n} \right]\nonumber \\
& \leq \mathbb P_*\left[ \sum_{|\ell-k|>T_n}\int_{B_n} | f_r(y_\ell) - f_r^*(y_\ell)|e^{\ell_n(\theta)- \ell_n(\theta^*)}d\Pi_2(\mathbf f)dQ > \frac{\rho^{-T_n} e^{-3L_n n\epsilon_n^2}}{ L_n} \right] + o(1) \nonumber \\
&  \leq \frac{ C_* L_n \rho^{T_n}e^{3L_nn\epsilon_n^2} }{L_n} \sum_{|\ell-k|>T_n} \int_{B_n}\mathbb E_{Q,\mathbf f}[| f_r(y_\ell) - f_r^*(y_\ell)|]\dd\Pi_2(\mathbf f) \dd Q + o(1) \nonumber \\
    &\leq  C_* L_n \rho^{T_n}e^{3L_nn\epsilon_n^2}  \sum_{s=1}^R \int_{B_n}\int f_s(y)(f_r(y) +f_r^*(y)) \dd y  \dd\Pi_2(\mathbf f)\nonumber + o(1)  \\
    & \leq C_* L_n \rho^{T_n}e^{3L_nn\epsilon_n^2}  \sum_s \int_{B_n} \|f_s\|_2(\|f_r\|_2 +\|f_r^*\|_2)d\Pi_2(\mathbf f) + o(1) \nonumber \\ & \leq C_* \rho^{T_n}e^{4L_nn\epsilon_n^2}  + o(1) = o(1),
\end{align}
with the final equality holding as soon as $T_n\log(1/\rho)\geq  5 n L_n\epsilon_n^2$. We recall that we required earlier in the proof, in order to establish Equation \eqref{small_ell}, that $T_n\epsilon_n\rightarrow 0$ arbitrarily slowly. We thus use the assumption $n\epsilon_n^3\rightarrow 0$ to choose $T_n$ so that both conditions may hold simultaneously.

\vspace{1ex}
Finally combining \eqref{large_ell} with \eqref{small_ell} and \eqref{smoothingG}, and since $K_n$ and $L_n$ can be chosen to go arbitrarily slowly to infinity and $n\epsilon_n^2 \gtrsim \log n$, we obtain that for any $z_n$ going to infinity, 
$$ \Pi_{cut}( \Delta_k > z_n \epsilon_n| Y_{1:n}) = o_{\mathbb P_*}(1),$$
which terminates the proof. 
\end{proof}

\section{Conclusion and discussion}\label{conclusion}
In this paper we use the cut posterior approach of \cite{jacob2017better} for inference in semiparametric models, which we apply to the nonparametric Hidden Markov models. A difficulty with the Bayesian approach in high or infinite dimension is that it is difficult (if not impossible) to construct priors on these complex models which allow for \textit{good} simultaneous inference on a collection of parameters of interest, where \textit{good} means having good frequentist properties (and thus leading to some robustness with respect to the choice of the prior). As mentioned in Section \ref{sec:intro}, a number of examples have been exhibited in the literature where \textit{reasonable} priors lead to poorly behaved posterior distribution for some functionals of the parameters. We believe that the cut posterior approach is an interesting direction to pursue in order to address this general problem, and we demonstrate in the special case of semiparametric Hidden Markov models that it leads to interesting properties of the posterior distribution and is computationally tractable.
\vspace{1ex}

Our approach is based on a very simple prior $\Pi_1$ on $Q, \mathbf f$ used for the estimation of $Q$ based on finite histograms with a small number of bins. This enjoys a Bernstein-von Mises property, so that credible regions for $Q$ are also asymptotic confidence regions. Moreover by choosing  $M$ large (but not too large) we obtain efficient estimators. Proving efficiency for semiparametric HMMs is non trivial and our proof has independent interest.  Another original and important contribution of our paper is an inversion inequality (stability estimate) comparing the $L^1$ distance between $g_{Q, \mathbf f_1}^{(3)}$ and $g_{Q, \mathbf f_2}^{(3)}$ and the $L^1$ distance between $ \mathbf f_1$ and $ \mathbf f_2$. Finally another interesting contribution is our control  of the error of the estimates of  the smoothing probabilities using a Bayesian approach, which is based on a control under the posterior distribution  of 
$$\frac{ 1}{n}\sum_{i=1}^n \frac{ |f_r(y_i) - f_r^* (y_i) | }{  \sum_r f_r^*(y_i)}, $$ 
despite the double use of the data (in the posterior on $f_r$ and in the empirical distance above). It is a rather surprising result, which does not hold if $f_r$ is replaced with $\hat f_r$ constructed using $y_1, \dots, y_n$, unless a sup-norm bound on $\hat f_r-f_r^*$ is obtained. 

\vspace{1ex}
Section \ref{simulation} demonstrates clearly the estimation procedures and highlights the importance of choosing a small number of bins in practical situations. Here, there remains open the question of how to choose this number in a principled way, and an interesting extension of the work would follow along the lines of the third section of \cite{Gassiat2018}, in which the authors produce an oracle inequality to justify a cross-validation scheme for choosing the number of bins. 

\vspace{1ex}
Finally, the paper deals with the case where $R$ is known. This assumption is rather common both in theory (see for instance the works of \cite{Gassiat2016}, \cite{de2017consistent} or \cite{Anandkumar2012}) and in practice (for instance in genomic applications as in \cite{yau2011bayesian}). The identifiability results of \cite{alexandrovich2016nonparametric} and \cite{gassiat2016inference} do show that $R$ can be identified as well, leaving open the possibility to jointly estimate $R$, together with $Q$ and $\textbf{f}$. We leave this direction of research for future work.

\section*{Acknowledgements}


The project leading to this work has received funding from the European Research Council
(ERC) under the European Union’s Horizon 2020 research and innovation programme (grant agreement No 834175).
The project is also partially funded by the EPSRC via the  CDT StatML. 

\pagebreak

\setcounter{section}{0}

\renewcommand{\thepage}{S\arabic{page}}

\renewcommand{\thesection}{S\arabic{section}}

\renewcommand{\thetable}{S\arabic{table}}

\renewcommand{\thefigure}{S\arabic{figure}}

\renewcommand{\thetheorem}{S\arabic{section}.\arabic{theorem}}

\renewcommand{\theprop}{S\arabic{section}.\arabic{prop}}

\renewcommand{\thelemma}{S\arabic{section}.\arabic{lemma}}

\renewcommand{\theassume}{S\arabic{assume}}

\renewcommand{\theequation}{S\arabic{section}.\arabic{equation}}

\newcommand{\thealgorithm}{S\arabic{section}.\arabic{algorithm}}

\renewcommand{\therk}{S\arabic{rk}}

\renewcommand{\thedefn}{S\arabic{section}.\arabic{defn}}

\begin{frontmatter}
\title{Supplement to ``Efficient Bayesian estimation and use of cut posterior in semiparametric hidden Markov models"}
\runauthor{Moss \& Rousseau}
\runtitle{Efficiency and cut posterior in semiparametric HMMs}

\begin{aug}
\author[A]{\fnms{Daniel} \snm{Moss}\ead[label=e1]{daniel.moss@stats.ox.ac.uk}},
\author[A]{\fnms{Judith} \snm{Rousseau}\ead[label=e2]{judith.rousseau@stats.ox.ac.uk}}
\address[A]{Department of Statistics,
University of Oxford, United Kingdom.
\printead{e1,e2}}
\end{aug}

\begin{abstract}
The supplementary material contains a number of proofs, a presentation of the general contraction theory for cut posteriors, and some further details on simulations. In Section \ref{sup:sec:mainproofs} we present the proofs of Theorem \ref{BvM} and Proposition \ref{L1contractmarginal}. In Section \ref{sup:sec:cutcontract}, we develop the general contraction theory for cut posteriors which is a straightforward adaptation \cite{ghosal2007convergence}, and has general interest beyond the setting of hidden Markov models. In Section \ref{sup:techproofs}, we collect a number of technical results. In Section \ref{sup:sec:dirmix}, we document the assumptions required for the application of our contraction theory when the emissions are modelled as Dirichlet process mixtures of Gaussians. In Section \ref{sup:sec:keyresults}, we gather some key results from the literature. In Section \ref{extrasims}, we provide some further details of our computations.
\end{abstract}

\begin{keyword}
\kwd{Bernstein-von Mises}
\kwd{Contraction rates}
\kwd{Cut posterior}
\kwd{Efficiency}
\kwd{Hidden Markov Models}
\kwd{Inversion inequality}
\kwd{Semiparametric estimation}
\end{keyword}

\end{frontmatter}
All sections, theorems, propositions, lemmas, definitions, remarks, assumptions, algorithms, figures and equations presented in the supplement are designed with a prefix S. Regarding the others, we refer to the material of the main text \cite{main}.

\section{Proofs of main results}\label{sup:sec:mainproofs}
In the following section we detail the proofs of Theorem \ref{BvM} and Proposition \ref{L1contractmarginal}.

\subsection{Proof of Theorem \ref{BvM}}

The proof of Theorem \ref{BvM} follows from Proposition \ref{scoreconvergence} and Lemma \ref{infogrows}.

\begin{lemma}\label{infogrows}
Let $(\mathcal{I}_M)_M$ be a sequence of embedded partitions, let $M$ be such that $I_M$ is admissible for $\mathbf{f}^*$, and grant Assumptions \ref{identifiability}, \ref{ass:Qbound} and \ref{densityratio}. Then the matrix

$$\Tilde{J}_{M+1}-\Tilde{J}_M$$
is positive semi-definite.
\end{lemma}
\begin{proof}
Consider the MLE for $\theta=(Q(\theta),\underline{\omega}_M(\theta))$ for parameters in $\mathcal{Q}\times\Omega_M$ in the model with partition $\mathcal{I}_M$.

\vspace{1ex}
By the remarks at the end of the proof of Theorem \ref{MLEM}, the MLE is regular. This means that, for all parameter sequences $\theta_n$ of the form $\theta_0+n^{-\frac{1}{2}}h$ with $h\in[-H,H]$ for some fixed $H$, $\sqrt{n}(\hat{\theta}^{(n)}-\theta_n)$ converges in distribution under $P_{\theta_n}$ to a variable, say $Z$, of fixed distribution, say $\mu_Z$. Explicitly, this means that for Borel sets $A$ we have
$$ \mathbb{P}_{\theta_n}\left(\sqrt{n}(\hat{\theta}^{(n)}-\theta_n)\in A\right)\rightarrow \mu_Z(A) \text{ as }n\rightarrow\infty.$$

Here $\theta_n$ is a sequence of parameters in $\mathcal{Q}\times\Omega_M$. However, any likelihood-based estimator depends on the observations only through the probabilities of each bin assignment, by considering the multinomial model with the count data. Thus, we can replace $\mathbb{P}_{\theta_n}$ with any law from the full semiparametric model for which the transition matrix $Q_n$ is the same and the functions $f_n$ have the corresponding $\underline{\omega}_M(\theta_n)$ as the bin assignment probabilities, and the above convergence remains true.

\vspace{1ex}

Since $\mathcal{I}_{M+1}$ is made by splitting bins in $\mathcal{I}_M$, a collection of weights $\omega^{(M+1)}=(\omega_m^{(M+1)}:m\leq\kappa_{M+1})$ for the partition $\mathcal{I}_{M+1}$ is identified with weights $\omega^{(M)}$ for the partition $\mathcal{I}_M$. We set $\omega_m^{(M+1)}=\sum_{m^\prime\in\mathrm{Ch}_{M}(m)}\omega_{m^\prime}^{(M)}$, where $\mathrm{Ch}_M(m)$ satisfies $I^{(M)}_m=\bigcup_{m^\prime\in \mathrm{Ch}_M(m) }I^{(M+1)}_{m^\prime}$. This implies that the MLE for the model $\mathcal{M}(\mathcal{I}_M)$ may be obtained by maximising the likelihood in model $\mathcal{M}(\mathcal{I}_{M+1})$ subject to the constraint that the quantities
$ \dfrac{\omega_{m^\prime}^{(M+1)}}{\lvert I_{m^\prime}^{(M+1) }\rvert}$ coincide for all $m^\prime\in\mathrm{Ch}(m)$. Regularity of the constrained MLE within $\mathcal{I}_{M+1}$ is the inherited from its regularity as a global MLE in $\mathcal{\mathcal{I}_{M}}$ by the preceding discussion.

\vspace{1ex}
Moreover, its asymptotic variance, which is given by the inverse Fisher information $\tilde{J}_M^{-1}$ for the model with partition $\mathcal{I}_M$, is lower bounded by the inverse Fisher information for the model with partition $\mathcal{I}_{M+1}$ by Lemma \ref{CR}, hence $ \Tilde{J}_M^{-1} - \Tilde{J}_{M+1}^{-1} $ is positive semi definite.
\end{proof}

\begin{proof}[Proof of Theorem \ref{BvM}]
We have by inspecting the part of the proof of Lemma \ref{infogrows} concerning regularity, that under any sequence of the form $\mathbb{P}^n_*=\mathbb{P}_{\theta^*(a,\mathbf{h})}$, where $\theta^*(a,\mathbf{h})=(Q^*+\frac{a}{\sqrt{n}},f_r^*(1+\frac{1}{\sqrt{n}}h_r):r\in[R])$ with $a\in\mathbb{R}^{R(R-1)}$ and $h_r\in L^2_0(f_r^*)$,
$$ R_{M,n} := d_{BL} \left( \mathbb{P}^n_*\Big(\sqrt{n}\big(\hat{Q}^{(M)}_n-(Q^*+\frac{a}{\sqrt{n}})\big)\in \cdot \Big) , \mu_M \right) \rightarrow 0 ,$$
where $\mu_M$ is Gaussian of covariance equal to $\tilde{J}_{M}^{-1}$ by Theorem \ref{MLEM}. Here $d_{BL}$ is the bounded Lipschitz metric which metrizes weak convergence. By taking $M_n\rightarrow\infty$ sufficiently slowly, we have $R_{M_n,n}\rightarrow 0$ also. Since for any $t\in\mathbb{R}^{R(R-1)}$, $(t^T\tilde{J}_{M_n}t)^{-1}$ is decreasing by Lemma \ref{infogrows} and bounded below by $(t^T\tilde{J}t)^{-1}$ the efficient variance, it converges and the limit is equal to $(t^T\tilde{J}t)^{-1}$ by Proposition \ref{scoreconvergence}. This implies weak convergence of the measures $\mu_{M_n}$ to $\mu$, the centred Gaussian measure whose variance is $\Tilde{J}^{-1}$, and applying the triangle inequality in the metric $d_{BL}$ proves the first claim.

\vspace{1ex}
For the Bernstein-von Mises result, we have from Theorem \ref{BvMM} that

$$\left\lVert \Pi_M(\sqrt{n}(Q-\hat{Q}_n) | Y_{1:n}) - \mathcal{N}(0,\Tilde{J}_M^{-1}) \right\rVert_{TV} =\tilde{R}_{M,n},$$
where $\tilde{R}_{M,n}\rightarrow 0$ in $\mathbb{P}_*-$probability as $n\rightarrow\infty$. Refining the sequence from before so that $M_n\rightarrow\infty$ sufficiently slowly that we additionally have $\tilde{R}_{M_n,n}\rightarrow 0$, and noting that $\tilde{J}_{M_n}\rightarrow \tilde{J}$, we get
\begin{align*}
    &\left\lVert \Pi_{M_n}(\sqrt{n}(Q-\hat{Q}_n) | Y_{1:n}) - \mathcal{N}(0,\Tilde{J}_{\infty}^{-1}) \right\rVert_{TV} \\
    &\leq \left\lVert \Pi_{M_n}(\sqrt{n}(Q-\hat{Q}_n) | Y_{1:n}) - \mathcal{N}(0,\Tilde{J}_{M_n}^{-1}) \right\rVert_{TV} + \left\lVert \mathcal{N}(0,\Tilde{J}^{-1}) - \mathcal{N}(0,\Tilde{J}_{M_n}^{-1}) \right\rVert_{TV} \\
    & = o_{\mathbb{P}_*}(1).
\end{align*}

\end{proof}

\subsection{Proof of Proposition \ref{L1contractmarginal}}
The proof of Proposition \ref{L1contractmarginal} is a direct consequence of the following result, which shows that Theorem \ref{cutcontract} applies to semi-parametric HMMs of the type described in Section \ref{sec:cut:gene}.  For the construction of $\Pi_1$, a monotonic transformation is implicitly used (if necessary) in order to consider histograms on $[0,1]$

\begin{prop}
Let $(Y_t)_{t\geq 1}$ be observations from a finite state space HMM with transition matrix $Q$ and emission densities $\mathbf f=(f_r)_{r=1,\dots,R}$. Consider the cut posterior based on $\Pi_1$ associated to the partition $\mathcal I _M$ and $\Pi_2$. 

\vspace{1ex}
$(i)$ Under \textbf{Assumptions} \ref{identifiability}-\ref{binsrefine}, \textbf{Assumption} \ref{pi1control} is satisfied with $\mathcal{T}_n = \{ Q; \|Q - Q^*\| \leq z_n/\sqrt{n}\}$ with $z_n\rightarrow\infty$ sufficiently slowly and $\phi_n $ the restriction to $\mathcal{T}_n$ of the Gaussian density centered at $\hat Q_{n,M}$ with variance $\tilde J_M^{-1}/n$.

\vspace{1ex}
$(ii)$ Choosing $\Pi_2$ such that \textbf{Assumption} \ref{emissioncontractcondition} is verified for suitable $\epsilon_n,\tilde{\epsilon}_n$ with $n\epsilon_n^2\gtrsim\log n$, the assumptions of Theorem \ref{cutcontract} are verified for the same $\epsilon_n$ and any $K_n\rightarrow\infty$.
\end{prop}

\begin{proof}
$(i)$ The idea of the proof is as follows: We know from Theorem \ref{BvMM} that $\Pi_1(\cdot|Y_{1:n})$ is close in total variation distance to the normal distribution centred at the estimator $\hat{Q}$. Then the restriction of this distribution to a ball of radius $z_n$ centred at $Q^*$, where $M_n\rightarrow\infty$ slowly, will be close in total variation to the original normal distribution with high probability as $n\rightarrow\infty$.

\vspace{1ex}
Let $p_{1,n}$ be the normal density of variance $\tilde{J}_M^{-1}/n$ centred at $\hat{Q}$, where $\hat{Q}$ is the estimator from Theorem \ref{BvMM}. Let $p_{2,n}$ be the restriction of this normal density to $\{Q:\lVert Q-Q^*\rVert \leq z_n/\sqrt{n} \}$ where $z_n\rightarrow\infty$. Fix $\epsilon>0$ and choose $n$ large enough that $\mathbb{P}_*(\lVert \hat{Q}-Q^*\rVert>cx^\prime_n/\sqrt{n})<\epsilon$ where $1>c>0$ is a constant. Let $p_{3,n}$ be the density obtained by restriction of $p_{1,n}$ to $\{Q:\lVert \hat{Q}-Q^*\rVert \leq (1-c)z_n /\sqrt{n}\}$. Then for sufficiently large $n$, we have with probability exceeding $1-\epsilon$ that the support of $p_{3,n}$ is a subset of the support of $p_2$, in which case
$$\lVert p_{1,n} - p_{2,n}\rVert_{L^1} \leq \lVert p_{1,n} - p_{3,n} \rVert_{L^1}.$$
But the variance of $p_{1,n}$ is of the order $O(n^{-1})$ and $z_n\rightarrow\infty$ so, as $n\rightarrow\infty$, the support of $p_{3,n}$ approaches the support of $p_{1,n}$ and $\lVert p_{1,n} - p_{3,n} \rVert_{L^1}\rightarrow 0$. But then $\lVert p_{1,n} - p_{2,n}\rVert_{L^1}\rightarrow 0$. Thus, with probability exceeding $1-\epsilon$, we have that $\lVert p_{1,n} - p_{2,n}\rVert_{L^1}\rightarrow 0$ as $n\rightarrow\infty$. Since $\epsilon>0$ is arbitrary we conclude $\lVert p_{1,n} - p_{2,n}\rVert_{L^1}=o_{\mathbb{P}_*}(1)$.

\vspace{1ex}
Taking $\phi_n=p_{2,n}$ as above for any $z_n\rightarrow\infty$, we verify Assumption \ref{pi1control} for any $\epsilon_n$ for which $n\epsilon_n^2 \gtrsim \log n$.

\vspace{1ex}
$(ii)$ This is a consequence of Theorem 3.1 in \cite{vernet2015conc}. The choice of $\mathcal{T}_n$ also provides verification of the hypothesis of Lemma 3.2 of \cite{vernet2015conc} and establishes the required control on KL divergence described in \eqref{KLcut}. To satisfy the testing assumption \eqref{testcut}, it suffices to note that for large enough $n$, $\mathcal{T}_n$ is a subset of $\{Q:Q_{i,j}\geq \frac{Q^*_{ij}}{2}\hspace{1ex}\forall i,j\leq R\}$ under \textbf{Assumption} \ref{ass:Qbound}. 
\end{proof}
\section{General theorem for cut posterior contraction}\label{sup:sec:cutcontract}
In this section, we present a general theory for contraction of cut posteriors which is developed in the style of the usual theory for Bayesian posteriors of \cite{ghosal2007convergence}. The main result of this section is Theorem \ref{cutcontract}, from which Proposition \ref{L1contractmarginal} follows.

\vspace{1ex}
Consider a general semiparametric model in which there is a finite dimensional parameter $\vartheta$ and an infinite-dimensional parameter $\eta$. Suppose we wish to estimate pair $(\vartheta,\eta)$ governing the law $\mathbb{P}_{\vartheta,\eta}^n$ of a random sample $Y_{1:n}\in\mathcal{Y}^n$. Assume the data is generated by some true distribution $\mathbb{P}_*^n=\mathbb{P}^n_{\vartheta^*,\eta^*}$ and consider the following two models:

\vspace{1ex}
\textbf{Model 1:} Consider a model $\mathcal{T} \times \mathcal{F}_1$ on pairs $(\vartheta,\eta)$ such that $\vartheta^*\in \mathcal{T}$ but for which we do not require that $\eta^*\in\mathcal{F}_1$. Consider a joint prior $\Pi_1$ over this space yielding a marginal posterior on $\vartheta$ given by
$$\Pi_1(\dd\vartheta | Y_{1:n})=\int_{\mathcal{F}_1}\Pi_1(\dd\vartheta, \dd\eta | Y_{1:n}).$$

\vspace{1ex}
\textbf{Model 2:} Consider a model on $\eta$, conditional on $\vartheta$, with a prior $\Pi_{2}(\cdot|\vartheta)$. We denote the parameter set for $\eta$ by $\mathcal{F}_{2}$ and we assume that $\eta^*\in \mathcal{F}_2$. We obtain in this model, a conditional posterior distribution $\Pi_2(\dd\eta|Y_{1:n},\vartheta)$. 

\begin{defn}
The \textit{cut posterior} on $(\vartheta,\eta)$ is given by
$$\Pi_{\text{cut}}(\dd\vartheta,\dd\eta|Y_{1:n})=\Pi_1(\dd\vartheta |Y_{1:n})\Pi_2(\dd \eta | \vartheta, Y_{1:n}).$$

\end{defn}
Write $\mathbb{P}_{\vartheta,\eta}^n$ for the law of the observations $Y_{1:n}$. Define for some $\ell>0$ and some loss function  $d(\cdot,\cdot)$ on $\mathbb{P}_{\vartheta,\eta}^\ell$, $$B_\ell(\epsilon)=\{ (\vartheta,\eta) : d(\mathbb{P}_{\vartheta,\eta}^\ell, \mathbb{P}_{\vartheta^*,\eta^*}^\ell) <  \epsilon\}, \quad M,\epsilon>0,$$ with $B_\ell^c(\epsilon)$ its complement. Define also the Kullback-Leibler neighbourhoods of $  \mathbb{P}^n_* = \mathbb{P}_{\vartheta^*,\eta^*}^n$ as
$$V_{n}(\mathbb{P}^n_*,\epsilon)=\{ (\vartheta,\eta) : \mathcal{K}( \mathbb{P}_{\vartheta^*,\eta^*}^n |\mathbb{P}^n_{\vartheta,\eta} ) < n\epsilon^2\},$$
where $\mathcal{K}( \mathbb{P}_{\vartheta^*,\eta^*}^n |\mathbb{P}^n_{\vartheta,\eta} ) = \mathbb{E}^n_*\log\frac{\dd  \mathbb{P}_{\vartheta^*,\eta^*}^n}{\dd \mathbb{P}_{\vartheta,\eta}^n}$ is the Kullback-Leibler divergence.



We now present a general theorem to characterize cut-posterior contraction result in the spirit of the now classical result of \cite{ghosal2007convergence}. 
 Our main additional assumption for the cut setup is that we have sufficiently good control over $\Pi_1$, similar to the kind established in Section \ref{efficiency} in the HMM setting, which we detail now.

\vspace{1ex}
Denote by $\pi_1( \cdot | Y_{1:n})$ the marginal posterior density of $\vartheta$ with respect to some measure $\mu_1$ on $\mathcal{T}$, associated to the prior $\Pi_1$ on $\vartheta, \eta$. 
\begin{assume}\label{pi1control}
 For all sequences $z_n\rightarrow\infty$ there exist $\mathcal{T}_n\subset\mathcal{T}$ with $\vartheta^*\in\mathcal{T}_n$, $\epsilon_n = o(1)$ and  a non-negative, random function $\phi_n$ on $\mathcal{T}_n$ with $\mathrm{supp}(\phi_n)\subset\mathcal{T}_n$, such that    $\lVert\pi_1(\cdot|Y_{1:n})-\phi_n(\cdot|Y_{1:n})\rVert_{L^1(\mu)}=o_{\mathbb{P}_*^n}(1)$,
$$\mu(\mathcal{T}_n)\mathbb E_*^n\left(\int_{\mathcal{T}_n}\phi_n^2(\vartheta)\dd\mu(\vartheta)\right) = O(z_n);\quad \exists K>0  \textrm{ }\mathrm{ s.t. }\textrm{ }  \mathbb P^n_*[\sup_{\vartheta \in \mathcal T_n}\phi_n(\vartheta)> e^{Kn\epsilon_n^2}] = o(1) $$
\end{assume}
Assumption \ref{pi1control} is mild, for instance if $\phi_n$ is a Gaussian distribution with mean $\hat \vartheta$ and variance $\mathbf i_*/n$, for some semi definite matrix $\mathbf i_*$ where $\hat \vartheta -  \vartheta^* = O_{\mathbb P_*^n}(1/\sqrt{n})$, then $\mathcal{T}_n = \{ \|\vartheta -  \vartheta^*\| \leq K_n/n\}$ with $K_n\rightarrow\infty $ arbitrarily slowly
$$\mathbb{E}^n_*\left(\int_{\mathcal{T}_n}\phi_n^2(\vartheta)\dd\mu(\vartheta)\right) \lesssim  n^{d/2}, \quad \mathrm{vol}(\mathcal{T}_n)\lesssim n^{-d/2} K_n^{d/2}, \quad \sup_{\vartheta \in \mathcal T_n}\phi_n(\vartheta) \lesssim n^{d/2} = o( e^{K_nn\epsilon_n^2}) $$
as soon as $n\epsilon_n^2 \gtrsim \log n$.

\begin{theorem}\label{cutcontract}
Let \textbf{Assumption} \ref{pi1control} hold with $n\epsilon_n^2\rightarrow\infty$ and assume that there exist $C>0$ such that for any  $\vartheta\in \mathcal{T}_n$,  sets $S_{n}(\vartheta) \subset \mathcal F_2$ satisfying
\begin{equation}\label{KLcut}
\bigcup_{\vartheta\in\mathcal{T}_n}\{\vartheta\}\times S_{n}(\vartheta)\subset V_{n}(\mathbb{P}^n_*,\epsilon_n), \quad \inf_{\vartheta \in \mathcal T_n} \Pi_2(S_n(\vartheta)|\vartheta)\geq e^{-Cn\epsilon_n^2}.
\end{equation}
 Assume also that there exist  $L_n\rightarrow\infty$,  $\mathcal{F}_{2,n}(\vartheta)\subset\mathcal{F}_2$ and $\psi_n:\mathcal{Y}^n\longrightarrow\{0,1\}$   such that
    $$\sup_{\vartheta \in \mathcal T_n}\Pi_2(\mathcal{F}_{2,n}(\vartheta)^c |\vartheta) \leq e^{-L_nn\epsilon_n^2} , \quad \text{set} \quad \mathcal{U}_n=\bigcup_{\vartheta\in\mathcal{T}_n}\{\vartheta\}\times\mathcal{F}_{2,n}(\vartheta),$$
   \begin{equation}\label{testcut}\mathbb{E}^n_*\psi_n =o(1), \hspace{3ex} \sup_{\stackrel{(\vartheta,\eta)\in \mathcal{U}_n}{d(\mathbb P_{\vartheta,\eta}^n, \mathbb{P}_{\vartheta^*,\eta^*}^n) >  K_n\epsilon_n}}\mathbb{E}_{\vartheta,\eta}^n(1-\psi_n) \leq C^\prime e^{-2L_nn\epsilon_n^2}, \quad 
   \end{equation}
for some $K_n,K>0$. Then, as $n\rightarrow\infty$,
$$\Pi_{cut}(B_\ell^c(K_n\epsilon_n) | Y_{1:n}) = o_{\mathbb{P}_*^n}(1).$$

\end{theorem}

\begin{rk}
Typically $K_n=O(\sqrt{L_n}$), when the tests (\ref{testcut}) are constructed as a supremum of local tests. Given some $L_n\rightarrow\infty$ for which the conditions are satisfied, we can choose $L_n\rightarrow\infty$ arbitrarily slowly, and consequently we can choose $K_n\rightarrow\infty$ arbitrarily slowly.
\end{rk}

The proof of Theorem \ref{cutcontract} which we present below is a rather simple adaptation of \cite{ghosal2007convergence}. As in  \cite{ghosal2017fundamentals}, we have simplified the common Kullback-Leibler neighbourhood assumption involving variances of the log-likelihood ratio using the technique of  Lemma 6.26 therein.

\begin{rk}
By placing an additional assumption on neighbourhoods for higher-order Kullback-Leibler variations, the assumption on the existence of the sequence $L_n\rightarrow\infty$ can be replaced with an assumption that $L_n=L$ for a constant $L>1+C$ - the proof is similar but uses a different (but standard) technique for proving the evidence lower bound. In this case we can choose $K_n$ a constant.
\end{rk}

\begin{proof}[Proof of Theorem \ref{cutcontract}]
The proof is an adaptation of  the  proof on posterior contraction rates, as in  \cite{ghosal2007convergence}. 
Under \textbf{Assumption} \ref{pi1control}, it suffices to prove the claim when we replace the cut posterior $\Pi_{cut}$ with the distribution $$\tilde{\Pi}(\dd\vartheta,\dd\eta)=\phi_n(\vartheta)\dd\mu(\vartheta) \Pi_2(\dd\eta | \vartheta, Y_{1:n} ).$$ Write also $\bar{B}_\ell=B_\ell^c(K_n\epsilon_n)\cap(\mathcal{T}_n\times\mathcal{F}^R)$, the subset of $B^c_\ell(K_n\epsilon_n)$ on which $\tilde{\Pi}$ is supported, and $\delta_n=\mu(\mathcal{T}_n)$.

\vspace{1ex}
Writing $J_{n}(\vartheta)$ for the random variable $\Pi_2(\bar{B}_\ell|\vartheta, Y_{1:n})$, we have
$J_{n}(\vartheta)=\dfrac{N_n(\vartheta)}{D_n(\vartheta)}$ where $$N_n(\vartheta)=\int_{\bar{B}_\ell}e^{l_n(\vartheta,\eta)-l_n(\vartheta^*,\eta^*)}\dd\Pi_2(\eta|\vartheta);\hspace{2ex}D_n(\vartheta)=\int_{\mathcal{F}_2}e^{l_n(\vartheta,\eta)-l_n(\vartheta^*,\eta^*)}\dd\Pi_2(\eta|\vartheta).$$

Write also $\Tilde{\Pi}(\Bar{B}_\ell | Y_{1:n}) = \Tilde{\Pi}(\bar{B}_\ell | Y_{1:n})\psi_n+\Tilde{\Pi}(\bar{B}_\ell | Y_{1:n})(1-\psi_n)1_{\Omega_n^c} + \Tilde{\Pi}(\bar{B}_\ell | Y_{1:n})(1-\psi_n)1_{\Omega_n},$
where $\Omega_n=\{\sup_{\vartheta \in \mathcal T_n}\phi_n(\vartheta)\leq e^{Kn\epsilon_n^2}\}$. By the testing assumption, the first term vanishes in probability, while the second term vanishes under Assumption \ref{pi1control}. For the remaining term, we have for $\tilde{L}_n\rightarrow\infty $ to be chosen later, that
\begin{align*}
   1_{\Omega_n}(1-\psi_n)\int_{\mathcal{T}_n}J_n(\vartheta)\phi_n(\vartheta)\dd\mu(\vartheta) \leq & \int_{\mathcal{T}_n}1_{\{D_n(\vartheta)\leq e^{-(C+\tilde{L}_n)n\epsilon_n^2}\}}\phi_n(\vartheta)\dd\mu(\vartheta) &(I_1) \\
   +&(1-\psi_n)\int_{\mathcal{T}_n}1_{\Omega_n}N_n(\vartheta)e^{(C+\tilde{L}_n)n\epsilon_n^2}\phi_n(\vartheta)\dd\mu(\vartheta)  , &(I_2)
\end{align*}
since $J_n(\vartheta)\leq 1$. We can bound $\mathbb{E}_*[I_1]$ as
\begin{align*}
  \mathbb{E}_*[I_1]&\leq \left(\int_{\mathcal{T}_n}\mathbb{P}_*{\{D_n(\vartheta)\leq e^{-(C+\tilde{L}_n)n\epsilon_n^2}\}}\dd\mu(\vartheta) \right)^{\frac{1}{2}} \left[E_*\left(\int_{\mathcal{T}_n}\phi_n^2(\vartheta)\dd\mu(\vartheta)\right)\right]^{\frac{1}{2}},
\end{align*}
by an application of Cauchy-Schwartz. Now, noting that by Lemma \ref{ELBO} and under the assumption of sufficient prior mass on the $S_n(\vartheta)$, we have for any $\tilde{L}_n\rightarrow\infty$ that

$$\mathbb{P}_*{\{D_n(\vartheta)\leq e^{-(\tilde{L}_n +C)n\epsilon_n^2}\}}= O({\tilde{L}_n}^{-1}),$$
and so $$\mathbb{E}_*[I_1]\lesssim \sqrt{\delta_n}\tilde{L}_n^{-1/2}\left(\int_{\mathcal{T}_n}\phi_n^2(\vartheta)\dd\mu(\vartheta)\right)^{\frac{1}{2}}=o(1)$$  under \textbf{Assumption} \ref{pi1control} by taking $z_n=o(\tilde{L}_n)$.
We bound $\mathbb{E}_*[I_2]$ by
\begin{align*}
  \mathbb{E}_*[I_2] &\leq  e^{(C+\tilde{L}_n)n\epsilon_n^2}  \mathbb{E}_*\left(\int_{\mathcal{T}_n} \int_{\bar{B}_\ell}[e^{l_n(\vartheta,\eta)-l_n(\vartheta^*,\eta^*)}(1-\psi_n)\phi_n]\dd\Pi_2(\eta|\vartheta) \dd\mu(\vartheta) \right)\\
   &\leq e^{(C+\tilde{L}_n)n\epsilon_n^2}   \left(\int_{\mathcal{T}_n} \int_{\bar{B}_\ell}\mathbb{E}_{\vartheta,\eta}[(1-\psi_n)\phi_n1_{\Omega_n}]\dd\Pi_2(\eta|\vartheta) \dd\mu(\vartheta) \right) .
\end{align*}
Using Fubini's theorem again alongside the assumption on type II errors and on the sieves, and the deterministic bound on $\phi_n(\vartheta)$ over $\Omega_n$ from \textbf{Assumption} \ref{pi1control}, we bound what precedes by 
\begin{align*}
\leq & e^{(K+C+\tilde{L}_n)n\epsilon_n^2}  \left(\int_{\mathcal{T}_n} \int_{\bar{B}_\ell} \mathbb{E}_{\vartheta,\eta}[(1-\psi_n) (1_{\eta\in\mathcal{F}_{2,n}(\vartheta)}+1_{\eta\in\mathcal{F}_{2,n}(\vartheta)^c})] \dd\Pi_2(\eta|\vartheta) \dd\mu(\vartheta) \right)\\
\leq \hspace{1ex} & e^{(K+C+\tilde{L}_n)n\epsilon_n^2} \Bigg\{\left(\int_{\mathcal{T}_n} \int_{\bar{B}_\ell}C^\prime e^{-L_nn\epsilon_n^2}   \dd\Pi_2(\eta|\vartheta) \dd\mu(\vartheta) \right) + \int_{\mathcal{T}_n}  e^{-L_nn\epsilon_n^2}\dd\mu(\vartheta) \Bigg\} \\
\leq \hspace{1ex} & e^{(K+C+\tilde{L}_n)n\epsilon_n^2}  \left(C^\prime e^{-L_nn\epsilon_n^2} + e^{-L_nn\epsilon_n^2} \right).
\end{align*}
Choosing $\tilde{L}_n=o(L_n)$, we get the required convergence.
\end{proof}

Lemma \ref{ELBO} provides the lower bound on the denominator used in the proof of Theorem \ref{cutcontract}. The proof is standard but we include it for completeness, it follows almost exactly the proof of Lemma 6.26 in \cite{ghosal2017fundamentals}.

\begin{lemma}\label{ELBO}
Let $A_n(\vartheta)=\{D_n(\vartheta)\geq \Pi_2(S_{n}(\vartheta)|\vartheta)e^{-n\tilde{L}_n\epsilon_n^2}\}$ where $\tilde{L}_n\rightarrow\infty$ and $S_{n}(\vartheta)$ is such that there exists $\mathcal{T}_n$ with
$$\bigcup_{\vartheta\in\mathcal{T}_n}\{\vartheta\}\times S_{n}(\vartheta)\subset V_{0}(P_0,\epsilon_n).$$ Then  $$\sup_{\vartheta\in\mathcal{T}_n}\mathbb{P}_*({A}_n(\vartheta)^c)=O({\tilde{L}_n}^{-1})=o(1)$$ as $n\rightarrow\infty$.
\end{lemma}

\begin{proof}
We show that for each $\vartheta$ that, with probability tending to one,
$$\int_{\mathcal{F}_2}e^{l_n(\vartheta,\eta)}-e^{l_n(\vartheta^*,\eta^*)}\dd\Pi_2(\eta|\vartheta) \geq \Pi_2(S_{n}(\vartheta)|\vartheta) e^{-nM_n\epsilon_n^2},$$
for any $M_n\rightarrow\infty$. It suffices to show the above equation holds when we restrict the integral to $S_{n}(\vartheta)\subset\mathcal{F}_2$, for which we have $\{\vartheta\}\times S_{n}(\vartheta)\subset V_n(\mathbb{P}_*,\epsilon_n)$. By dividing both sides by $\Pi_2(S_{n}(\vartheta)|\vartheta)$ we see that it suffices to show that, for any probability measure $\nu$ supported on $S_n=S_n(\vartheta)$, that
$$\int_{S_n}e^{l_n(\vartheta,\eta)}-e^{l_n(\vartheta^*,\eta^*)}\dd\nu(\eta) \geq  e^{-nM_n\epsilon_n^2}.$$
By applying the logarithm to both sides and using Jensen's inequality, it suffices to show that with high probability
$$Z:= \int_{S_n}l_n(\vartheta,\eta)-l_n(\vartheta^*,\eta^*)\dd\nu(\eta) \geq  -nM_n\epsilon_n^2.$$
Arguing as in \cite{ghosal2017fundamentals}, we obtain
$$\mathbb{P}_*(Z < -nM_n\epsilon_n^2) = O(M_n^{-1}) = o(1) $$
as $n\rightarrow\infty$ for any $M_n\rightarrow\infty$. Since the bound does not depend on $\vartheta$ and uses only that $\{\vartheta\}\times S_n(\vartheta)\subset V_0(P_0,\epsilon_n)$, we obtain the desired bound on the supremum.

\end{proof}

\section{Proofs of technical results}\label{sup:techproofs}
Section \ref{sup:techproofs} is devoted to a number of technical proofs which are required the main results, but are reasonably standard in their approach.

\vspace{1ex}
In Section \ref{invertibleFI}, we prove that the Fisher information matrix is invertible for general discrete state-space, discrete observation HMMs. This is necessary to apply the results of \cite{bickel1998asymptotic} and \cite{DeGunst2008} in the proofs of Theorems \ref{MLEM} and \ref{BvMM}.

\vspace{1ex}
In Section \ref{sup:sec:filowerbound}, we gather some properties of the Fisher information matrix, showing a Cramer-Rao bound for estimation in HMMs, and showing a local uniform convergence result for the expected information from $n$ observations.

\vspace{1ex}
In Section \ref{TechnicalLemmas}, we collect technical lemmas used for the deconvolution argument which is the key part of the proof of Theorem \ref{BvM}.

\vspace{1ex}
In Section \ref{sup:sec:admisspart}, we state a result which implies the existence of an admissible partition.

\vspace{1ex}
In Section \ref{sup:sec:smoothproof}, we collect technical lemmas which feature in the proofs of Theorems \ref{inversion} and \ref{smoothingcontract}

\subsection{Non-singularity of Fisher Information}\label{invertibleFI}
In what follows, we establish invertibility of the Fisher Information matrix for general multinomial Hidden Markov models.
\begin{prop}\label{prop:invertibleFI}
Consider a multinomial HMM with latent states $X_t\in\{1,\dots,R\}$ and discrete observations $Y_t$ taking values in the set of basis vectors of $\mathbb{R}^{\kappa}$, which we denote $\{e_1,\dots,e_M\}$. Denote $Q\in\mathbb{R}^{R\times R}$ the transition matrix and $\Omega\in\mathbb{R}^{\kappa\times R}$ be the matrix whose columns $\omega_r=(\omega_{mr})_{m=1}^M$ are the emission probabilities for the $r^{th}$ state.

\vspace{1ex}
Denote $\theta=(Q,\Omega)\in\mathbb{R}^p$ the HMM parameter and write $J(\theta)$ for the Fisher Information matrix with entries given by

$$ J(\theta)_{ij} = -\lim_{n\rightarrow\infty}\underbrace{\mathbb{E}_{\theta}\left[\dfrac{\partial^2}{\partial\theta_i\partial\theta_j}\log p_\theta(Y_1,\dots,Y_n)\right]}_{{-J_n(\theta)}_{ij}}. $$
Then, if $Q,\Omega$ have rank $R$, $J$ is non-singular.
\end{prop}

The idea of the proof is to exhibit estimators with $L^2$ risk of order $\frac{1}{\sqrt{n}}$. We then show that the local asymptotic minimax result of \cite{gassiat2013revisiting} implies that the existence of such estimators guarantees a non-singular Fisher information. We use the spectral estimators proposed in \cite{Anandkumar2014}, the control of which is established in Section \ref{sec:spectral}.

\begin{proof}[Proof of Proposition \ref{prop:invertibleFI}]
By an application of the van Trees inequality, analogously to Equation (12) in Theorem 4 of \cite{gassiat2013revisiting}, we obtain for the HMM that (for $p$ the parameter dimension)

\begin{align*}
    &\int_{B_p(0,1)}\mathbb{E}^{(n)}_{\theta^*+ch/\sqrt{n}}\left[\left(U^T(\hat{\theta}^{(n)}_{}-(\theta^*+\frac{ch}{\sqrt{n}}))\right)^2\right]q(h)\dd h \\
    & \geq U^T\left(\dfrac{n}{c^2}J_q+n\int_{B_p(0,1)}\frac{1}{n}J_n(\theta^*+\frac{ch}{\sqrt{n}})q(h)\dd h\right)^{-1}U,
\end{align*}
which holds for any vector $U$. Here $J_n$ is the joint Fisher information for $n$ observations as in Proposition \ref{uniformFI} and $q$  is a density on $\mathbb{R}^p$ such that $J_q:=\int_{\mathbb{R}^P}\nabla q \nabla q^T\dfrac{1_{\{q>0\}}}{q}\dd x$ is non-singular. Rescaling we get that, for any vector $U$,
\begin{align*}
    &\int_{B_p(0,1)}\mathbb{E}^{(n)}_{\theta^*+ch/\sqrt{n}}\left[\left(\sqrt{n}U^T(\hat{\theta}^{(n)}_{}-(\theta^*+\frac{ch}{\sqrt{n}}))\right)^2\right]q(h)\dd h \\
    & \geq U^T\left(\dfrac{1}{c^2}J_q+\int_{B_p(0,1)}\frac{1}{n}J_n(\theta^*+\frac{ch}{\sqrt{n}}) q(h) \dd h\right)^{-1}U.
\end{align*}
Taking a limit in $n$ and applying Lemma \ref{uniformFI} stated below gives that the limit inferior of the left hand side is at least

$$\lim_nU^T\left(\dfrac{1}{c^2}J_q+\int_{B_p(0,1)}\frac{1}{n}J_n(\theta^*+\frac{ch}{\sqrt{n}}) q(h) \dd h\right)^{-1}U=U^T\left(\dfrac{1}{c^2}J_q+J(\theta^*)\right)^{-1}U.$$

Call the matrix on the right hand side which we invert $J_c$. It is indeed invertible for sufficiently large $c$ as $J_q$ is invertible, and the set of invertible matrices is open. Denote its matrix square root by $J_c^{\frac{1}{2}}$. Now suppose $\exists V^*$ such that $(V^*)^TJ(\theta^*)V^*=0$. Then by writing $V=J_c^{-\frac{1}{2}}U$ we get
\begin{align*}
    \sup_{\lVert U\rVert =1}U^TJ_c^{-1}U&=\sup_{\lVert U\rVert\neq 0}\dfrac{U^TJ_c^{-1}U}{U^TU} \\
    &=\sup_{\lVert V\rVert\neq 0}\dfrac{V^TV}{V^TJ_cV}\geq \dfrac{{V^*}^TV^*}{{V^*}^TJ_cV^*}=Ac^2.
\end{align*}
With $A$ a fixed constant not depending on $c$.
Taking the limit as $c\rightarrow\infty$ gives that (upper bounding also the averaging over the law $q\dd h$ by the supremum over $h$)

$$\liminf_{c\rightarrow\infty}\sup_{\lVert U\rVert =1}\liminf_{n\rightarrow\infty}\sup_{\lVert h\rVert <1}\mathbb{E}^{(n)}_{\theta^*+ch/\sqrt{n}}\left[\left(\sqrt{n}U^T(\hat{\theta}^{(n)}_{}-(\theta^*+\frac{ch}{\sqrt{n}}))\right)^2\right]= +\infty,$$
or equivalently,

$$\liminf_{c\rightarrow\infty}\sup_{\lVert U\rVert =1}\liminf_{n\rightarrow\infty}\sup_{\lVert \theta-\theta^*\rVert <\frac{c}{\sqrt{n}}}\mathbb{E}^{(n)}_{\theta}\left[\left(\sqrt{n}U^T(\hat{\theta}^{(n)}_{}-\theta)\right)^2\right]= +\infty,$$

which contradicts the local uniform bound of Proposition \ref{estimators}.
\end{proof}

\vspace{1ex}
The following proposition establishes the existence of an estimator with suitable risk, as required for the arguments of Proposition \ref{invertibleFI}.

\begin{prop}\label{estimators} Let $\theta=(Q,\Omega)$ be the parameter for the HMM described in Proposition \ref{prop:invertibleFI} and let $\theta^*$ be such that the identifiability conditions of \textbf{Assumption} \ref{identifiability} hold. Then there exists an estimator $\hat{\theta}$ such that, for any $\epsilon>0$ sufficiently small, we have, as $n\rightarrow\infty$,
$$\sup_{\lVert \theta-\theta^*\rVert < \epsilon } \mathbb{E}_\theta\lVert\hat{\theta}-\theta\rVert^2_2 \leq \frac{C}{n}(1+o(1)),$$ up to label-swapping.
\end{prop}

\begin{proof}
Let $\hat{\theta}$ be the spectral estimator constructed in Section \ref{sec:spectral}. We have that
$$\mathbb{E}_\theta \lVert \hat{\theta}-\theta\rVert_2^2=\int_{0}^\infty t\mathbb{P}_\theta( \lvert \hat{\theta}-\theta\rVert>t) \dd t.$$
Setting $t=\frac{Cx}{\sqrt{n}}$ we get
$$\mathbb{E}_\theta \lVert \hat{\theta}-\theta\rVert_2^2=\int_{0}^{\infty} \frac{C^2x}{n}\mathbb{P}_\theta( \lVert \hat{\theta}-\theta\rVert>\frac{Cx}{\sqrt{n}}) \dd x .$$
The bound of Lemma \ref{labelfix} is valid for $x\geq 1$, and so
$$\mathbb{E}_\theta\lVert \hat{\theta}-\theta\rVert_2^2 \leq \frac{C^2}{n} + \frac{C^2}{n}\int_{0}^\infty xe^{-x^2}\dd x\leq\frac{3C^2}{2n} .$$
Since the bound is valid on a small neighbourhood around $\theta$, it holds in supremum over that neighbourhood.
\end{proof}

\subsubsection{Construction of spectral estimators}\label{sec:spectral}
Spectral estimation of HMMs has been addressed by a number of previous works \cite{hsu2012spectral,Anandkumar2012,Anandkumar2014,de2017consistent,abraham2021fundamental,abraham2021multiple}. In \cite{abraham2021multiple}, the authors exhibit estimators for emissions in parametric HMMs with a $\sqrt{n}$ rate in probability, but we require convergence in expectation. In \cite{de2017consistent}, a convergence in expectation is shown but the rate contains a logarithmic factor which is not sufficient for our use in the proof of Proposition \ref{invertibleFI}. In \cite{abraham2021fundamental}, the authors exhibit estimators for which the concentration is sub-Gaussian, which would permit integration to an in-expectation bound, but their work only concerns two-state HMMs.

\vspace{1ex}
We will construct estimators which are very similar to those constructed in these works. To get the convergence properties we require, we use as a starting point the sub-Gaussian concentration of the empirical estimator of $E_{ijk}=\mathbb{P}(Y_1=i,Y_2=j,Y_3=k)$. To extract the HMM parameters, we follow the approach set out in \cite{Anandkumar2014}, who prove guarantees for their \textit{Tensor power method} which outputs, from an estimate $\hat{T}$ of a given symmetric, orthogonal tensor $T$, estimates of the associated eigenvalues and eigenvectors whose error is of the same order as the input error $\lVert\hat{T}-T\rVert$.

\vspace{1ex}
Since $T$ is required to be symmetric and orthogonal, we cannot directly apply their algorithm to extract estimates with appropriate concentration properties from estimates $\hat{E}$ of the tensor $E$. Instead, we construct an estimate $\hat{T}$ of a symmetrised, orthogonalised version $T$ of $E$, the eigenvalues and eigenvectors of which can then be extracted through the Tensor power method and related back to the parameters of the HMM. This approach is similar to that taken in \cite{anandkumar2012spectral,hsu2013learning}; \cite{hsu2013learning} consider the setting of Spherical Gaussians, which requires a different symmetrisation step but whose orthogonalisation procedure we adopt here. Our symmetrisation step is instead based on Theorem 3.6 of \cite{Anandkumar2014}.

\paragraph*{The construction}

First define the following matrices and tensors:
$$ E_{\theta}^{(12)} = \mathbb{E}_{\theta}[Y_1\otimes Y_2],\hspace{3ex} E_{\theta}^{(13)} = \mathbb{E}_{\theta}[Y_1\otimes Y_3], \quad E_{\theta}^{(123)} = \mathbb{E}_{\theta}[Y_1 \otimes Y_2 \otimes Y_3 ] ,$$
$$\hat{E}^{(12)}=\frac{1}{n-1}\sum_{i=1}^{n-1}Y_i\otimes Y_{i+1},\quad \hat{E}^{(13)}=\frac{1}{n-2}\sum_{i=1}^{n-2}Y_i\otimes Y_{i+2},\quad
\hat{E}^{(123)}=\frac{1}{n-2}\sum_{i=1}^{n-2}Y_i\otimes Y_{i+1} \otimes Y_{i+2}.$$
Next, we will symmetrise about $Y_3$ as follows: define
$$\Tilde{Y}_1=\underbrace{E^{(21)}_\theta(E^{(12)}_\theta)^{-1}}_{A_\theta}Y_1, \quad \Tilde{Y}_2=\underbrace{E^{(31)}_\theta(E^{(21)}_\theta)^{-1}}_{B_\theta} Y_2,$$
where $E^{(ij)}_\theta$ is the transpose of $E^{(ji)}_\theta$. Define $\Tilde{E}_\theta^{(ij)}$ and $\Tilde{E}_\theta^{(123)}$ in the same way as $E_\theta^{(ij)}, E_\theta^{(123)}$, but with $Y_1,Y_2$ replaced by $\Tilde{Y}_1,\Tilde{Y}_2$ respectively; these matrices are then symmetric. Define estimators $\check{E}^{(12)}$, $\check{E}^{(13)}$ and $\check{E}^{(123)}$ of the symmetric quantities as follows: Split the sample in two and from the first sample, produce estimates  $\hat{E}^{(\cdot)} $. This yields estimates $\hat{A}$, $\hat{B}$ of $A_\theta,B_\theta$. Then define
$$\check{E}^{(12)} = \frac{1}{n-\lceil\frac{n}{2}\rceil}\sum_{i=\lceil\frac{n}{2}\rceil}^{n-1}\hat{A}Y_i\otimes\hat{B}Y_{i+1},\quad \check{E}^{(13)} = \textbf{}\sum_{i=\lceil\frac{n}{2}\rceil}^{n-1}\hat{A}Y_i\otimes Y_{i+2},$$
and
$$\check{E}^{(123)} = \frac{1}{n-\lceil\frac{n}{2}\rceil - 1}\sum_{i=\lceil\frac{n}{2}\rceil}^{n-1}\hat{A}Y_i\otimes \hat{B}Y_{i+1}\otimes Y_{i+2}.$$
We remark that Theorem 3.6 of \cite{Anandkumar2014} gives us the decomposition of the symmetric tensor $\Tilde{E}_\theta^{(123)}$ as
$$ \Tilde{E}_\theta^{(123)} = \sum_{i=1}^R p_i \mu_i\otimes\mu_i\otimes\mu_i ,$$

where $p_i=\mathbb{P}_\theta(X_2=i)$ is taken to be the stationary distribution and $\mu_i= \mathbb{P}_\theta(Y_3 = \cdot | X_2=i) = (\Omega Q^T)_i$. In a slight abuse of notation, we surpress the dependence on $\theta$ in our notation for $p_i,\mu_i$.

\vspace{1ex}
This justifies that indeed $\Tilde{E}^{(123)}_\theta$ is a symmetric tensor. However, the $\mu_i$ are generally not orthonormal, which is required for the application of Theorem 5.1 of \cite{Anandkumar2014} (the tensor power method). Define, for a three-way tensor $T$ and matrices $P,Q,R$,
$$ [T(P,Q,R)]_{ijk} = \sum_r\sum_s\sum_tT_{rst}P_{ri}Q_{sj}R_{tk}. $$

Next, define $\Tilde{E}^{(123)}_{W,\theta}$ as in Section 4.3.1 of \cite{Anandkumar2014} as follows: Take $W=W_\theta$ as the matrix of left orthonormal singular vectors of $\tilde{E}^{(12)}$ so that $W^T\Tilde{E}^{(12)}_\theta W=I$ and define
$$ \mu_i^{(W)}=\sqrt{p_i}W^T\mu_i,$$
which are orthonormal. Then take
$$ \Tilde{E}^{(123)}_{W,\theta} = \Tilde{E}^{(123)}_\theta(W,W,W) = \sum_{i=1}^R \frac{1}{\sqrt{p_i}}{\mu}_i^{W}\otimes {\mu}_i^{W}\otimes {\mu}_i^{W}. $$
Now set $\hat{W}$ as the matrix of left orthonormal singular vectors of $\check{E}^{(12)}$, with corresponding $\check{E}^{(123)}_{\hat{W}} = \check{E}^{(123)}(\hat{W},\hat{W},\hat{W})$.

\vspace{1ex}
At this point, we have at hand an estimate $\check{E}^{(123)}_{\hat{W}}$ of the orthogonal, symmetric tensor $\Tilde{E}^{(123)}_{W,\theta}$ so we apply the Tensor power method of \cite{Anandkumar2014}. This returns estimates $\hat{\mu}_I^W$, $\hat{p}_i$ of $\mu_i^W$ and $p_i$. Now set $\hat{\mu}_i=\frac{1}{\sqrt{\hat{\nu}_i}}\hat{W}^\dag\hat{\mu}_i^W$, which esimates (up to label-swapping) the columns of $\Omega Q^T$.

\vspace{1ex}
By repeating the above procedure, but instead symmetrising about $Y_2$ (so defining a corresponding $\tilde{Y}_1,\tilde{Y}_3$) we produce estimates (up to label-swapping) of the columns of $\Omega$. By permuting these estimates relative to some consistent estimator (such as the MLE, which is consistent without assuming invertibility of the Fisher Information, see e.g. \cite{Douc2004}) we finally produce estimates $(\hat{Q},\hat{\Omega})$ of $(Q,\Omega)$.

\vspace{1ex}
In the following section, we control the error introduced at each stage of this construction.

\subsubsection{Control of spectral estimators}\label{sec:spectralcontrol}
The control on the estimators constructed in \ref{sec:spectral} is ultimately proved in Lemma \ref{labelfix}; this control is established through the following sequence of technical lemmas. The estimator we will use are those constructed by the tensor power method in \cite{Anandkumar2014}. It is convenient to assume $\kappa=R$ in what follows so that certain matrices are square - the extension to the general case is straightforward using SVD and low rank approximations but is omitted for ease of exposition.

\vspace{1ex}
The next lemma makes explicit the sub-Gaussian concentration of emperical tensors when the data is generated according to a geometrically ergodic Markov chain. It resembles closely Lemma 1 of \cite{abraham2021fundamental}.

\begin{lemma}[Concentration of empirical tensors]\label{Hoeffding}
Define the matrices and tensors
$$ E_{\theta}^{(12)} = \mathbb{E}_{\theta}[Y_1\otimes Y_2],\hspace{3ex} E_{\theta}^{(13)} = \mathbb{E}_{\theta}[Y_1\otimes Y_3], \quad E_{\theta}^{(123)} = \mathbb{E}_{\theta}[Y_1 \otimes Y_2 \otimes Y_3 ] .$$
Define further $$\hat{E}^{(12)}=\frac{1}{n-1}\sum_{i=1}^{n-1}Y_i\otimes Y_{i+1},\quad \hat{E}^{(13)}=\frac{1}{n-2}\sum_{i=1}^{n-2}Y_i\otimes Y_{i+2},\quad
\hat{E}^{(123)}=\frac{1}{n-2}\sum_{i=1}^{n-2}Y_i\otimes Y_{i+1} \otimes Y_{i+2}.$$
Then there exists a constant $C=C(R,Q^*)$, such that for all $x\geq 1$, sufficiently small $\epsilon>0$, and any superscripts,

$$ \sup_{\|\theta - \theta^*\| \leq \epsilon} \mathbb{P}_{\theta}(\lVert \hat{E}^{(\cdot)} - E_{\theta}^{(\cdot)} \rVert ) > \frac{Cx}{\sqrt{n}} )\leq e^{-x^2} .$$
\end{lemma}

\begin{proof}
The proof is essentially the same as the proof of Lemma 1 in \cite{abraham2021fundamental}. In that paper, the authors use the concentration result of \cite{paulin2015concentration} which provides a bound in terms of the pseudo-spectral gap $\gamma_{ps}$ of the chain $Z_n=(X_n,X_{n+1},X_{n+2},Y_n,Y_{n+1},Y_{n+2}$), together with Proposition 3.4 of the same reference which relates $\gamma_{ps}$ to the mixing time $t_{{mix}}$ of the same chain. Arguing as in \cite{abraham2021fundamental}, we see that $t_{mix}$ is uniformly bounded on a neighbourhood of the true parameter, since the eigenvalues of the transition matrix vary continuously in its entries. For the remainder of the proof, we follow the argument of \cite{abraham2021fundamental}.
\end{proof}

An important step in the application of the Tensor power method is to pre-process the empirical tensors controlled in Lemma \ref{Hoeffding} so that they are symmetric and orthogonal. The symmetrisation step involves a linear transformation of the first two observations so that the three-way expectation is symmetric - see Theorem 3.6 of Anandkumar et al. (2014). Note that $E^{(12)}_\theta=\sum_{i=1}^Rp_iv_i\otimes w_i$ where $p_i=\mathbb{P}_\theta(X_2=i)$, $v_i=\mathbb{P}_\theta(Y_1=\cdot|X_2=i)$ and $w_i=\mathbb{P}_\theta(Y_2=\cdot|X_2=i)$, and by writing $v_i,w_i$ in terms of $Q$ and $\Omega$, we see that $E^{(12)}_\theta$ has rank $R$ and so is indeed invertible.
\begin{lemma}[Symmetrisation preserves concentration]\label{symmetry}

Define $$\Tilde{Y}_1=\underbrace{E^{(21)}_\theta(E^{(12)}_\theta)^{-1}}_{A_\theta}Y_1, \quad \Tilde{Y}_2=\underbrace{E^{(31)}_\theta(E^{(21)}_\theta)^{-1}}_{B_\theta} Y_2,$$
where $E^{(ij)}_\theta$ is the transpose of $E^{(ji)}_\theta$. Define $\Tilde{E}_\theta^{(ij)}$ and $\Tilde{E}_\theta^{(123)}$ in the same way as $E_\theta^{(ij)}, E_\theta^{(123)}$  in Lemma \ref{Hoeffding} but with $Y_1,Y_2$ replaced by $\Tilde{Y}_1,\Tilde{Y}_2$ respectively; these matrices are then symmetric. Define estimators $\check{E}^{(12)}$, $\check{E}^{(13)}$ and $\check{E}^{(123)}$ of the symmetric quantities as follows: Split the sample in two and from the first sample, produce estimates  $\hat{E}^{(\cdot)} $. This yields estimates $\hat{A}$, $\hat{B}$ of $A_\theta,B_\theta$. Then define
$$\check{E}^{(12)} = \frac{1}{n-\lceil\frac{n}{2}\rceil}\sum_{i=\lceil\frac{n}{2}\rceil}^{n-1}\hat{A}Y_i\otimes\hat{B}Y_{i+1},\quad \check{E}^{(13)} = \textbf{}\sum_{i=\lceil\frac{n}{2}\rceil}^{n-1}\hat{A}Y_i\otimes Y_{i+2},$$
and
$$\check{E}^{(123)} = \frac{1}{n-\lceil\frac{n}{2}\rceil - 1}\sum_{i=\lceil\frac{n}{2}\rceil}^{n-1}\hat{A}Y_i\otimes \hat{B}Y_{i+1}\otimes Y_{i+2}.$$
Then there exists a constant $\Tilde{C}=\Tilde{C}(R,W^*)$ such that for all $x\geq 1$, sufficiently small $\epsilon>0$ and any superscripts,
$$\sup_{\|\theta-\theta^*\|\leq \epsilon}\mathbb{P}_\theta(\lVert \check{E}^{(\cdot)} - \Tilde{E}_\theta^{(\cdot)} \rVert  > \frac{\Tilde{C}x}{\sqrt{n}} )\leq e^{-x^2} .$$
\end{lemma}

\begin{proof}
Write $\mathcal{E}$ for the event that $\hat{E}^{(12)}$ is invertible. Choose $r$ such that the ball of radius $r$ around $\theta^*$ is contained in the set of $\theta$ for which the corresponding $Q=Q(\theta)$ is invertible. Then for all $\theta$ in a ball of radius at most $\frac{r}{2}$ inside the ball on which Lemma \ref{Hoeffding} holds, the ball of radius $\frac{r}{2}$ centred at any $Q(\theta)$ contains only those $\theta^\prime$ for which $Q^\prime(\theta^\prime)$ is invertible. Hence, by choosing $x=c\sqrt{n}$ where $c=c(r)$, we can apply Lemma \ref{Hoeffding} for sufficiently large $n$ to get that $\mathbb{P}_{\theta}(\mathcal{E}^c)\leq e^{-c^2n}=o(1)$ as $n\rightarrow\infty$.

\vspace{1ex}
Working on $\mathcal{E}$, Lipshitz continuity of matrix inversion and multiplication implies that $\hat{A}$ and $\hat{B}$ satisfy an exponential bound of the kind shown for $\hat{E}^{(\cdot)}$, as a corollary of Lemma \ref{Hoeffding}. This implies that $\check{E}^{(12)}$ and $\check{E}^{(13)}$ satisfy a corresponding bound when approximating the analogous quantities with $\hat{A},\hat{B}$ replaced with $A,B$. These quantities in turn approximate $\Tilde{E}^{(\cdot)}$ by the same arguments as presented in Lemma \ref{Hoeffding}.

\vspace{1ex}
For the estimates $\hat{A}$ and $\hat{B}$, we can formally take $\hat{A}=\hat{B}=I$ on the event that $\hat{E}^{(12)}$ is not invertible.

\vspace{1ex}
We deduce that for $\Tilde{C}$ suitably chosen
$$ \mathbb{P}_\theta(\lVert \check{E}^{(\cdot)} - \Tilde{E}_\theta^{(\cdot)} \rVert > \frac{\Tilde{C}x}{\sqrt{n}} )\leq e^{-x^2} +o(1), $$
and by enlarging the constant by some factor greater than one, we get that for sufficiently large $n$, any $x\geq 1$, and sufficiently small $\epsilon>0$ that
$$ \sup_{\|\theta-\theta^*\|\leq \epsilon}\mathbb{P}_\theta(\lVert \check{E}^{(\cdot)} - \Tilde{E}^{(\cdot)}_\theta \rVert  > \frac{\Tilde{C}x}{\sqrt{n}} )\leq e^{-x^2} .$$
\end{proof}
Lemma \ref{symmetry} tells us we can construct an approximation to the symmetrised tensors $\Tilde{E}^{(\cdot)}$. In the next result, we show that we can further extend this to an approximation to a symmetrised, orthogonalised tensor, at which point the Tensor power method may be applied to see that the parameter values extracted from this approximate tensor well approximate the parameter values associated to the true orthogonal, symmetric tensor (which are themselves related in an analytic sense to the parameters of interest). The orthogonalisation or `whitening' step is detailed in Lemma \ref{whitening} and follows closely the ideas of appendix C of \cite{hsu2013learning}.

\vspace{1ex}
Before we proceed, we remark that Theorem 3.6 of \cite{Anandkumar2014} gives us the decomposition of the symmetric tensor $\Tilde{E}_\theta^{(123)}$ as
$$ \Tilde{E}_\theta^{(123)} = \sum_{i=1}^R p_i \mu_i\otimes\mu_i\otimes\mu_i ,$$

where $p_i=\mathbb{P}_\theta(X_2=i)$ is taken to be the stationary distribution and $\mu_i= \mathbb{P}_\theta(Y_3 = \cdot | X_2=i) = (\Omega Q^T)_i$. In a slight abuse of notation, we surpress the dependence on $\theta$ in our notation for $p_i,\mu_i$.

\vspace{1ex}
This justifies that indeed $\Tilde{E}^{(123)}_\theta$ is a symmetric tensor. However, the $\mu_i$ are generally not orthonormal, which is required for the application of Theorem 5.1 of \cite{Anandkumar2014} (the tensor power method). First, let us define for a three-way tensor $T$ and matrices $P,Q,R$,
$$ [T(P,Q,R)]_{ijk} = \sum_r\sum_s\sum_tT_{rst}P_{ri}Q_{sj}R_{tk}. $$
We are now ready to show that the concentration of Lemma \ref{Hoeffding} is preserved on further reduction:

\begin{lemma}[Whitening preserves concentration]\label{whitening}
Define $\Tilde{E}^{(123)}_{W,\theta}$ as in Section 4.3.1 of \cite{Anandkumar2014} as follows: Take $W=W_\theta$ as a linear transformation such that $W^T\Tilde{E}^{(12)}_\theta W=I$ (made explicit in the proof) and define
$$ \mu_i^{(W)}=\sqrt{p_i}W^T\mu_i,$$
which are orthonormal. Then take

$$ \Tilde{E}^{(123)}_{W,\theta} = \Tilde{E}^{(123)}_\theta(W,W,W) = \sum_{i=1}^R \frac{1}{\sqrt{p_i}}{\mu}_i^{W}\otimes {\mu}_i^{W}\otimes {\mu}_i^{W}. $$
Then there exists an estimate $\hat{W}$, with corresponding $\check{E}^{(123)}_{\hat{W}} = \check{E}^{(123)}(\hat{W},\hat{W},\hat{W})$ such that, for some $\Tilde{C}^\prime=\Tilde{C}^\prime(R,Q^*)$, for any $x\geq 1$ and sufficiently small $\epsilon>0$

$$ \sup_{\lVert\theta-\theta^*\rVert\leq\epsilon} \mathbb{P}_\theta(\lVert \check{E}^{(123)}_{\hat{W}} - \Tilde{E}^{(123)}_{W,\theta} \rVert > \frac{\Tilde{C}^\prime x}{\sqrt{n}} )\leq e^{-x^2}. $$

\end{lemma}

\begin{proof}
The proof essentially follows Appendix C.6 in \cite{hsu2013learning}. In particular, we take $W,\hat{W}$ to be the matrix of left orthonormal singular vectors of $\tilde{E}^{(12)}$ and $\check{E}^{(12)}$ respectively. Then their Lemma 12 gives that

$$\lVert \check{E}^{(123)}_{\hat{W}} - \Tilde{E}^{(123)}_{W,\theta} \rVert \leq \lVert \check{E}^{(123)}(\hat{W},\hat{W},\hat{W})-\tilde{E}^{(123)}_\theta(\hat{W},\hat{W},\hat{W})  \rVert + \frac{6}{\sqrt{p_{\min}}}\frac{\lVert \check{E}^{(12)}-\tilde{E}^{(12)}_\theta\rVert}{\zeta_R(\tilde{E}_\theta^{(12)})}$$
where $p_{\min}$ is the smallest element of the stationary distribution (which under \textbf{Assumption} \ref{identifiability} is uniformly bounded below on a ball around the true parameter) and $\zeta_R(\tilde{E}^{(12)})$ is the $R^{th}$ largest singular value of $\tilde{E}_\theta^{(12)}$, which is also positive under this assumption at $\theta=\theta^*$ (and uniformly bounded below on some neighbourhood of the true parameter) as $\tilde{E}^{(12)}_{\theta^*}$ has full rank. Then on an event $\mathcal{E}$ of arbitrarily large probability the map $M\mapsto M(\hat{W},\hat{W},\hat{W})$ is Lipshitz with sufficiently large constant $K>0$; on $\mathcal{E}$ the first term on the right-hand side is then bounded by $K\lVert \check{E}^{(123)}-\tilde{E}^{(123)}_\theta\rVert$ and the result follows.
\end{proof}

Having constructed a theoretical orthogonal tensor $\Tilde{E}_{W,\theta}^{(123)}$ based on the (unknown) distribution of our observations, and having at hand an estimator $\check{E}^{(123)}_{\hat{W}}$ for this tensor with good concentration properties, Theorem 5.1 of \cite{Anandkumar2014} gives guarantees for an algorithm which recovers the $\mu_i^W$ (and the $w_i$). Note that the probability of failure (their $\eta$) for the algorithm described may be tuned by the user, and so for a theoretical estimator we may consider the algorithm in the limit $\eta\rightarrow 0$ (corresponding to running the algorithm for an arbitrarily long number of iterations). A direct application of the theorem then provides the following result:

\begin{lemma}[Recovery of orthogonalised parameters]\label{orthogparams}
Suppose Algorithm 1 of \cite{Anandkumar2014} is called iteratively, as in the statement of Theorem 5.1 of the same reference. Let $\hat{\mu}_i^W$ be the estimates of $\mu_i^W$ returned by this process and $\hat{p}_i$ the estimates of $p_i$ returned (by also applying $x\mapsto x^{-2}$ to the estimated weights). Then there exists $\Tilde{C}^{\prime\prime}=\Tilde{C}^{\prime\prime}(R,Q)$ and a permutation $\tau\in\mathcal{S}_R$ such that for all $x\geq 1$ and sufficiently small $\epsilon>0$

$$ \sup_{\lVert\theta-\theta^*\rVert\leq\epsilon}\mathbb{P}_\theta(\max_{i=1,\dots,R}\lVert \hat{\mu}_i^W - {}^\tau\mu_i^W \rVert > \frac{\Tilde{C}^{\prime\prime }x}{\sqrt{n}} , \max_{i=1,\dots,R}\lVert \hat{p}_i - {}^\tau p_i \rVert > \frac{\Tilde{C}^{\prime\prime }x}{\sqrt{n}} )\leq e^{-x^2}. $$
\end{lemma}

The previous result follows easily from the fact that the order of error in the output corresponds to the order of the error in the input tensors, together with the smoothness of $x\mapsto x^{-2}$, so the proof is omitted. It remains to show that these estimates can be `de-whitened' to produce estimates for the vector of interest $\mu_i$, which we show in the following result.

\begin{lemma}[Recovery of original parameters]\label{params}
Let $W^\dag$ be the Moore-Penrose pseudo-inverse of $W^T$. Then
$$\mu_i^{(\theta)}=\mu_i=\frac{1}{\sqrt{p_i}}W^\dag\mu_i^W, $$
and the $\hat{\mu}_i$ constructed as
$$\hat{\mu}_i=\frac{1}{\sqrt{\hat{\nu}_i}}\hat{W}^\dag\hat{\mu}_i^W $$
satisfy
$$ \sup_{\lVert\theta-\theta^*\rVert\leq\epsilon}\mathbb{P}_\theta(\max_{i=1,\dots,R}\lVert \hat{\mu}_i - {}^\tau\mu_i^{(\theta)} \rVert > \frac{Cx}{\sqrt{n}} )\leq e^{-x^2}, $$
for some $\tau\in\mathcal{S}_R$ and $C=C(R,Q^*)$, for all $x\geq 1$ and for sufficiently small $\epsilon>0$.

\end{lemma}

\begin{proof}
The first part of the statement is given by Theorem 4.3 of \cite{Anandkumar2014}. The control over $\hat{\mu}_i$ then follows from the control over $\check{E}^{(12)}$ established in Lemma \ref{symmetry} and the Lipshitz continuity of the singular vectors and Moor-Penrose pseudo-inverse on a neighbourhood of the true parameter, as well as the control over $\hat{\nu}_i$ and $\hat{\mu}_i^W$ established in Lemma \ref{params}.
\end{proof}

The estimates produced target (up to label-swapping) the columns of $\Omega Q^T$. By adapting the procedure to symmetrise about the second view (rather than the third), we can produce estimates (up to label-swapping) of $\Omega$. Although the labelling is arbitrary, the permutation of Lemma \ref{params} may differ in the two cases.

\vspace{1ex}
To alleviate this issue, we use any consistent estimator (such as the maximum-likelihood estimator, which is consistent without assumptions on the Fisher information, see Theorem 1 of \cite{Douc2004}). Because the transition matrix $Q$ has full rank, the $\mu_i$ (and associated vectors from symmetrisation about the second view) are well-separated and so we can fix a labelling relative to this estimator.

\begin{lemma}[Fixing of labels]\label{labelfix}
Let $\hat{\Xi}_3$ be the matrix whose columns are the $\hat{\mu}_i$ as in Lemma \ref{params} (with symmetrisation about the third view as detailed) and $\hat{\Xi}_2$ be the corresponding matrix when symmetrising about the second view. Moreover, let $\Check{\Xi}_3$,$\Check{\Xi}_2$ be the corresponding estimates formed from the MLE, that is
$$ \Check{\Xi}_3 = \Check{\Omega}_{MLE}\Check{Q}^T_{MLE}, \hspace{3ex} \Check{\Xi}_2 = \Check{\Omega}_{MLE}. $$
For $i=3,2$,  pick $\hat{\tau}_i\in\mathcal{S}_R$ such that
$$\lVert {}^{\hat{\tau}_i}\hat{\Xi}_i - \check{\Xi}_i \rVert = \min_{\tau\in\mathcal{S}_R} \lVert {}^{\tau}\hat{\Xi}_i - \check{\Xi}_i \rVert, $$
where $\lVert \cdot \rVert$ is the max of the squared norms of the columns. Define
$$\hat{\Omega} = {}^{\hat{\tau}_2}\hat{\Xi}_2, $$ 
and
$$\hat{Q}^T = \hat{\Omega}^{-1}{}^{\hat{\tau}_3}\hat{\Xi}_3 .$$
Recall $\theta=(Q,\Omega)$ and define $\hat{\theta}=(\hat{Q},\hat{\Omega})$ .Then there exists a $\tau\in\mathcal{S}_R$ (the permutation accounting for the difference between the truth and the MLE) such that, for all $\epsilon>0$ small enough,
$$ \sup_{\lVert \theta-\theta^*\rVert < \epsilon}\mathbb{P}_\theta(\max_{i=1,\dots,R}\lVert \hat{\theta}_i - {}^\tau \theta_i \rVert > \frac{Cx}{\sqrt{n}} )\leq e^{-x^2}, $$
for some $C=C(R,Q)$.
\end{lemma}

\subsection{Fisher Information and asymptotic lower bound}\label{sup:sec:filowerbound}
The following result gives a lower bound on the asymptotic variance of regular estimators and is used in the proof of Lemma \ref{infogrows}.
\begin{lemma}\label{CR}
Let $\hat{\theta}_n$ be a regular estimator in the histogram model with partition $\mathcal{I}_M$, with $M$ fixed. Let $Z$ denote the random variable of law equal to the limit distribution of the scaled and centred estimates $\sqrt{n}(\hat{\theta}_n-\theta^*)$ under $\mathbb{P}_*$ as $n\rightarrow\infty$. Then
$$\text{Cov}(Z)\geq J(\theta^*),$$
where $J(\theta)$ is the Fisher information matrix at the parameter $\theta$.
\end{lemma}
\begin{proof}
Consider estimation of the one-dimensional parameter $\lambda^T\theta$. Recall from Lemma \ref{uniformFI} that for any sequence $\theta_n\rightarrow\theta$, we have $J(\theta)=\lim_{n\rightarrow\infty}\frac{1}{n}J_n(\theta_n)$, with $J_n$ the joint Fisher information for $n$ observations, as defined in Lemma \ref{uniformFI}.

\vspace{1ex}
We follow the arguments of Gill and Levit \cite{gill1995applications}. Let $\pi$ be a fixed prior density on $[-1,1]$ and $J(\pi)=\mathbb{E}[{\left(\log\pi(\theta)\right)^\prime} ^2]$. For a given $H>0$, let $\pi(H,n)$ be the rescaling of this prior to the interval $A=[\theta_0-n^{-\frac{1}{2}}H,\theta_0+n^{-\frac{1}{2}}H]$ for given $H>0$. Then, applying the van Trees inequality we get
$$\mathbb{E}[\lambda^T\hat{\theta}^n-\lambda^T\theta]^2\geq\dfrac{1}{\mathbb{E}\lambda^TJ_n(\theta)\lambda+nJ(\pi)/H^2},$$
where the left hand expectation is taken over the joint law of the parameter and the data given the parameter (with the parameter distributed according to $\pi(H,n)$) and the expectation in the denominator is over $\theta$ having that law also. Dividing through gives

$$\mathbb{E}[\sqrt{n}(\hat{\theta}^n-\theta)]^2\geq\dfrac{1}{\frac{1}{n}\mathbb{E}\lambda^TJ_n(\theta)\lambda+J(\pi)/H^2}.$$
\vspace{1ex}

Taking first $n\rightarrow\infty$ then $H\rightarrow\infty$ we obtain the result, applying Lemma \ref{uniformFI} to get convergence of the expectation term in the denominator to the Fisher information.
\end{proof}

The previous result requires the following convergence property of the Fisher information matrices. We also use it in the proof of Proposition \ref{invertibleFI}.

\begin{lemma}\label{uniformFI}
Let $\theta_n\rightarrow\theta$. Denote $J_n=J_n^{(M)}$ the Fisher information for $n$ observations in the model with $M$ bins given by $$J_n(\theta)=\mathbb{E}_\theta[-\nabla_\theta^2\log L_n(\theta)], $$ and $J$ the Fisher information for the model. Then $$\frac{1}{n}J_n(\theta_n)\rightarrow J(\theta).$$
\end{lemma}
When $\theta_n=\theta$ the result is simply the definition of $J$. The interest is hence in establishing local uniform convergence.
\begin{proof}
Theorem 3 of \cite{Douc2004} establishes the result in the case of observed information\footnote{The result as stated in the reference conditions on the value $X_0=x_0$ - the general statement follows and can be found explicitly in Theorem 13.24 in \cite{douc2014nonlinear}.},

$$-\frac{1}{n}\nabla_\theta^2l_n(\theta_n) \rightarrow J(\theta),$$
a.s. under $\mathbb{P}_\theta$. Moreover, the Fisher and Louis identities \cite{Louis1982} show that boundedness of the complete observed information and complete scores (which take simple forms) implies boundedness of the observed information. Write $\mathcal{Y}=[M]^\mathbb{N}$ and $\mu$ for the counting measure on this space with associated densities $p_\theta(y)$. We have since the observed information is bounded and $p_{\theta_n}(y)\rightarrow p_\theta(y)$ pointwise and so in $L^1(\mu)$ by Scheff\'e's lemma \cite{scheffe1947useful},
\begin{align*}
    \frac{1}{n}J_n(\theta_n) &= \int_{\mathcal{Y}}-\frac{1}{n}\nabla^2_\theta l_n(\theta_n)p_{\theta_n}(y)\dd \mu(y) \\
    & = \int_{\mathcal{Y}}-\frac{1}{n}\nabla^2_\theta l_n(\theta_n)p_{\theta}(y)\dd \mu(y) +o(1) \\
    & = J(\theta^*) + o(1).
\end{align*}
\end{proof}

\subsection{Technical results for the deconvolution argument}\label{TechnicalLemmas}

Here, we collect the technical results required to prove Lemma \ref{Deconvolution}. 

\vspace{1ex}
Let us recall some notation for what follows. For $r\in\{1,\dots, R\}$, and $g\in L^2(f^*_r)$, define
$$H_{r,M}(g)=g(y_0)\mathbb{P}_*(X_0=r|\mathcal{G}^0_m) + \sum_{j=-\infty}^{-1}g(y_j) (\mathbb{P}_*(X_j=r|\mathcal{G}^0_M) - \mathbb{P}_*(X_j=r|\mathcal{G}^{-1}_M)),$$
and define $H_r(g)$ similarly but with the conditioning being on the sigma algebras $\mathcal{G}^0$ and $\mathcal{G}^{-1}$. Then define, for $g = (g_1, \cdots, g_r)$, $$H_M(g)=\sum_rH_{r,M}(g_r); \quad H(g)=\sum_rH_r(g_r).$$ These are the score functions in the submodels with fixed transition matrix, and emission densities varying along the path characterised by $g$.

\vspace{1ex}

Our first technical result allows us to eliminate certain terms when we make the main deconvolution argument. We recall $\mathcal{G}^l=\sigma(Y_{-\infty:l})$.
\begin{lemma}\label{sup:zeromeans}
For $g\in \mathcal{H}_r$ and $j<k\leq 0$, we have
$$\mathbb{E}_*[g(Y_j)\mathbb{P}_*(X_j=r|\mathcal{G}^0) | Y_{k:0} ] = 0 . $$
\end{lemma}

\begin{proof}
We first note that
$$\mathbb{E}_*[g(Y_j)\mathbb{P}(X_j=r|\mathcal{G}^0) | Y_{k:0} ] = \mathbb{E}_*[g(Y_j)1_{X_j=r}| Y_{k:0}]=\mathbb{E}_*[\mathbb{E}_*[g(Y_j)1_{X_j=r}| Y_{k:0},X_j]|Y_{k:0}],$$
by using usual properties of conditional expectations. Given $X_j$, $Y_j$ and $X_j$ are conditionally independent of $Y_{k:0}$ and thus
$\mathbb{E}_*[g(Y_j)1_{X_j=r}| Y_0,X_j]=\mathbb{E}_*[g(Y_j)1_{X_j=r}| X_j].$
Since $X_j$ has values in the discrete set $[R]$, we can explicitly write the conditional expectation as
$$\mathbb{E}_*[g(Y_j)1_{X_j=r}| X_j]=\sum_{r^\prime=1}^R\mathbb{E}_*[g(Y_j)1_{X_j=r}| X_j=r^\prime]1_{X_j=r^\prime}=\mathbb{E}_*[g(Y_j)| X_j=r]1_{X_j=r}$$
But since $g\in L^2(f^*_r\dd x)$ has zero expectation against the emission distribution for state $r$, the right hand side vanishes and hence so does its expectation given $Y_{k:0}$.
\end{proof}

In the following section we  relate the score functions $H_M$ in the binned model to score functions $H$ in the full model where the perturbation $g$ is in the `direction' of a histogram.

\subsubsection{ On the difference $H_M - H$ } \label{sec:HMH}
\begin{lemma}\label{sup:lemma:HmH}
If  $h_M=(h_{r,M})_{r\leq R}$ with $\sup_M\sum_r\lVert h_{r,M}\rVert_{L^{2}(f_r^*)}\leq L<\infty$, then
$$H_{M}(h_M) = H(h_M) +o_{L^2}(1).$$
\end{lemma}
\begin{proof}
We have for each $r$,
\begin{align}\label{eq:HmHscores}
    H_{r,M}(h_{r,M})&=h_{r,M}(y_0)\mathbb{P}_*(X_0=r|\mathcal{G}^0_M) + \sum_{j=-\infty}^{-1}h_{r,M}(y_j) (\mathbb{P}_*(X_j=r|\mathcal{G}^0_M) - \mathbb{P}_*(X_j=r|\mathcal{G}^{-1}_M)); \nonumber \\
H_r(h_{r,M})&=h_{r,M}(y_0)\mathbb{P}_*(X_0=r|\mathcal{G}^0) + \sum_{j=-\infty}^{-1}h_{r,M}(y_j) (\mathbb{P}_*(X_j=r|\mathcal{G}^0) - \mathbb{P}_*(X_j=r|\mathcal{G}^{-1})).
\end{align}
We will show that $H_{r,M}(h_{r,M})-H_{r}(h_{r,M})=o_{L^2}(1)$.

\vspace{1ex}
Let $\epsilon>0$ be arbitrary and $J=J(L,\epsilon)>0$ be such that $2L\rho^{J-1}<\frac{\epsilon}{3}(1-\rho)$. Then Lemma \ref{lemma:scoretails} implies that
\begin{align*}
    \bigg\lVert\sum_{j=-\infty}^{-J}h_{r,M}(y_j) &(\mathbb{P}_*(X_j=r|\mathcal{G}^0) - \mathbb{P}_*(X_j=r|\mathcal{G}^{-1})) \\ +\sum_{j=-\infty}^{-J}& h_{r,M}(y_j) (\mathbb{P}_*(X_j=r|\mathcal{G}^0_M) - \mathbb{P}_*(X_j=r|\mathcal{G}^{-1}_M))\bigg \rVert_{L^2(\mathbb{P}_*)} \leq \frac{2\epsilon}{3} .
\end{align*}
We have now reduced the problem to the case where the sums in (\ref{eq:HmHscores}) are over $j>-J$. Since $\lVert h_M\rVert_{L^2}=O(1)$, to show convergence of the remaining difference between $H_{r,M}$ and $H_r$ it suffices to show that $\mathbb{P}_*(X_j=r|\mathcal{G}^l_M)\rightarrow \mathbb{P}_*(X_j=r|\mathcal{G}^l)$ in $L^\infty(\mathbb{P}_*)$ for $l=0,1$ and $-J\leq j \leq 0$.
\vspace{1ex}

We use Lemma \ref{lemma:forgetting}. Let $K>J$. We then obtain, for all $j\geq -J$ and $l=0,-1$,
\begin{align*}
    \mathbb{P}_*(X_j=r|\mathcal{G}^l)&=\sum_{s=1}^R\mathbb{P}_*(X_j=r|Y_{-K:l},X_K=s)\mathbb{P}_*(X_{-K}=s|\mathcal{G}^l_\infty)   \\
    &=\sum_{s=1}^R \mathbb{P}_*(X_j=r|Y_{-K:0})(1+O(\rho^{j+K})) \mathbb{P}_*(X_{-K}=s|\mathcal{G}^l_\infty) \\
    & =\left(1+O(\rho^{j+K})\right) \sum_{s=1}^R \mathbb{P}_*(X_j=r|Y_{-K:0})\mathbb{P}_*(X_{-K}=s|\mathcal{G}^l_\infty) \\
    &=\left(1+O(\rho^{j+K})\right)\mathbb{P}_*(X_j=r|Y_{-K:0}).
\end{align*}
Similarly, $\mathbb{P}_*(X_j=r|\mathcal{G}^l_M)=(1+O(\rho^{j+K}))\mathbb{P}_*(X_j=r|Y^{(M)}_{-K:0})$. Then
\begin{align*}
    &\sup_{Y_{-\infty:0}}\lvert \mathbb{P}_*(X_j=r|\mathcal{G}^l_M) - \mathbb{P}_*(X_j=r|\mathcal{G}^l)\rvert\leq \sup_{Y_{-\infty:0}} \lvert \mathbb{P}_*(X_j=r|\mathcal{G}^l_M) - \mathbb{P}_*(X_j=r|Y^{(M)}_{-K:l})\rvert \\&\quad + \sup_{Y_{-\infty:0}}\lvert \mathbb{P}_*(X_j=r|Y_{-K:l}^{(M)}) - \mathbb{P}_*(X_j=r|Y_{-K:l})\rvert
    + \sup_{Y_{-\infty:0}}\lvert \mathbb{P}_*(X_j=r|Y_{-K:l}) - \mathbb{P}_*(X_j=r|\mathcal{G}^l)\rvert.
\end{align*}
By choosing $K$ sufficiently larger than $J$, we have seen that the first and third terms are small, and bounded above by $\frac{\epsilon}{9JL}$ uniformly in $j\geq -J$. It remains to bound the second term. We have
\begin{align*}
    &\lvert \mathbb{P}_*(X_j=r|Y_{-K:l}^{(M)}) - \mathbb{P}_*(X_j=r|Y_{-K:l})\rvert \\&=\mathbb{P}_*(X_j=r)\left\lvert\dfrac{\mathbb{P}_*(Y_{-K:l}^{(M)}|X_j=r)}{\mathbb{P}_*(Y_{-K:l}^{(M)})}-\dfrac{\mathbb{P}_*(Y_{-K:l}|X_j=r)}{\mathbb{P}_*(Y_{-K:l})}\right\rvert.
\end{align*}
We write 
$$ \mathbb{P}_*(Y_{-K:l}^{(M)}|X_j=r)= \sum_{x_{-K},\cdots, x_l}\mathds{1}_{x_j=r} \prod_{t=-K}^lf_{\omega_{x_t}^*}(Y_t)\mathbb P_*(X_{-K:l}=x_{-K:l}|X_j=r),$$
and similarly for $Y_{-K:l}$. 
The true emissions $f_r^*$ are continuous functions on the compact set $[0,1]$, hence they are uniformly continuous and hence are  approximated in $L^\infty([0,1])$ by the histograms $f_{\omega_r^*}$ with $\omega_r^*=(\int_{I_m}f_r^*(y)\dd y:m\in[\kappa_M])$. Since there are a fixed number of such terms, we can choose $M$ large enough that, for each $j$, the preceding disaplay is bounded above by $\frac{\epsilon}{9JL}$. Putting everything together and using H\"{o}lder's inequality, we get that $$\lVert H_{r,M}(h_{r,M})-H_r(h_{r,M})\rVert_{L^2(\mathbb{P}_*)} \leq \epsilon .$$ The general case follows by summing over $r$, which concludes the proof.
\end{proof}

\subsubsection{Deconvolution argument } \label{sec:deconv}
The next lemma essentially provides a reduction of the deconvolution argument.
\begin{lemma}\label{sup:deconvolvehelp}
Let $h_M$ be a sequence of step functions on the partition $\mathcal{I}_M$ with $\lVert h_M\rVert = \sum_{r}\lVert h_{r,M} \rVert _{L^{2}}\leq 1$, such that $h_{r,M}\in \mathcal{H}_{r,M}$. For any $k\geq 0$,
$$H_M(h_M)\stackrel{L^2(\mathbb{P}_*)}{\rightarrow}0 \implies \sum_{r=1}^R\sum_{k'=-k}^0h_M(y_{k'})\mathbb{P}_*(X_{k'}=r|Y_{-k:0}) \stackrel{L^2(\mathbb{P}_*)}{\rightarrow}0.$$
\end{lemma}
\begin{proof}
For notational simplicity, we will write for any index $j$ and $\sigma$-algebra $\mathcal{G}$ $$h_M(Y_j)\mathbb{P}_*(X_0|\mathcal{G})=\sum_rh_{r,M}(Y_j)\mathbb{P}_*(X_0=r|\mathcal{G}).$$

By Lemma \ref{sup:lemma:HmH}, we have $H_M(h_M)\stackrel{L^2(\mathbb{P}_*)}{\rightarrow}0 \iff H(h_M)\stackrel{L^2(\mathbb{P}_*)}{\rightarrow}0.$ Denote by $Y^k$ the vector $(Y_{-\infty : k})$. Note that $H(h_M)(Y_{-\infty:0}) = G_M(Y^0)$ has the same distribution as $G_M(Y^{j})$ $\forall j \leq 0$, by stationarity of the latent chain. 
Define
\begin{align*}
    D_0^{(M)}(h^{(M)},J) = \sum_{j=J}^0G_M(Y^{j}) &= \sum_{j=J}^0h_M(Y_j)\mathbb{P}_*(X_j|\mathcal{G}^0) \\ &+\sum_{j=-\infty}^{J-1}h_M(Y_j) (\mathbb{P}_*(X_j|\mathcal{G}^0) - \mathbb{P}_*(X_j|\mathcal{G}^{J-1}))=o_{L^2}(1).
\end{align*}
Using Lemma \ref{sup:zeromeans} with $k=0$ to eliminate the terms which condition on $\mathcal{G}^0$ we obtain
\begin{align*}
    \mathbb{E}_*(D_0^{(M)}(h^{(M)},J)|Y_0)  &= h_M(Y_0) \mathbb{P}_*(X_0|Y_0) - \sum_{j=-\infty}^{J-1}E(h_M(Y_j)  \mathbb{P}_*(X_j|\mathcal{G}^{J-1}) |Y_0) \\ & =o_{L^2}(1).
\end{align*}

Using Lemma \ref{lemma:scoretails} we obtain, for a particular $\rho<1$ which does not depend on $M$, that

\begin{align*}
    &\sum_{j=-\infty}^{J-1}\mathbb{E}_*(h_M(y_j)  \mathbb{P}_*(X_j|\mathcal{G}^{J-1}) |Y_0) \\ &= \underbrace{\sum_{j=-\infty}^{J-1}\mathbb{E}_*(h_M(y_j)  \mathbb{P}_*(X_j|\mathcal{G}^{J-1}) )}_{=0} + O(\rho^{J}) = O(\rho^{J}).
\end{align*}
with the right-hand sum vanishing because the $\mathbb{E}_*(h_M(y_j)  \mathbb{P}_*(X_j|\mathcal{G}^{J-1}) ) = 0$ by Lemma \ref{sup:zeromeans}.
The term $O(\rho^{J})$ is uniform in $M$ since $\|h_M\|_{L^2(\mathbb P_*)}\leq 1$ and by choosing $J$ large enough this term can be made arbitrarily small. 
We conclude that $$h_M(y_0)\mathbb{P}_*(X_0|Y_0) = o_{L^2}(1),$$ as $M\rightarrow\infty$. By instead applying Lemma \ref{sup:zeromeans} with $k<0$, we can argue similarly that, for all $k\geq 0$
$$D_k^{(M)}(h_M):=\sum_{r=1}^R\sum_{k'=-k}^0h_M(y_{k'})\mathbb{P}_*(X_{k'}=r|Y_{-k:0}) = o_{L^2}(1).$$
\end{proof}
\subsubsection{ On the convergence on $\tilde{\mathcal{A}}_M$ } \label{sec:tildeAM}
The proof of the following result is very similar to the proof of Lemma 1 in \cite{Gassiat2018}, which used that the spaces of score functions for mixture models were nested to show convergence.
\begin{lemma}
Let $\tilde{\mathcal{A}}_M$ $\mathcal{\tilde{P}}_M$, $\mathcal{A}$, $\mathcal{P}$ be as in the statement of Lemma \ref{projectionconv}. Then $\mathcal{A}=\tilde{\mathcal{A}}:=\lim_{M\rightarrow\infty}\tilde{\mathcal{A}}_M$.
\end{lemma}

\begin{proof}
We have by the arguments of \cite{Gassiat2018} that for any $S\in L^2(\mathbb{P}_*)$ $(\tilde{\mathcal A}_MS)$ is Cauchy, as the spaces $\tilde{\mathcal P}_M$ are nested. By completeness of $L^2(\mathbb{P}_*)$ we then establish convergence to some element $\mathcal{\tilde{A}}S\in L^2(\mathbb{P}_*)$. Since $\tilde{\mathcal A}_M$ are projections, $\tilde{\mathcal A}$ is also a projection onto its image, which is some subspace
$$\tilde{\mathcal P} \subset \text{Cl}\left(\bigcup_{M\in\mathbb{N}}\tilde{\mathcal P}_M\right). $$
Recall that $\mathcal{P}$ is defined as the closure of the linear span of the $H(h)$ functions, as in the statement of Lemma \ref{projectionconv}. By definition, $\mathcal{P}$ is closed, and since $\tilde{\mathcal{P}}_M\subset\mathcal{P}$ for all $M$, we must have $\text{Cl}\left(\bigcup_{M\in\mathbb{N}}\tilde{\mathcal P}_M\right)\subset\mathcal{P}$. We show that any element in $H(h)\in\mathcal{P}$ can be written as the $L^2$ limit of some sequence in $\bigcup_{M\in\mathbb{N}}\tilde{\mathcal P}_M$. First choose a sequence $h_M\rightarrow h$ in $L^2(\mathbb{P}_*)$ such that $\lVert h_M\rVert_{L^2}\leq 2\lVert h\rVert_{L^2}$ for all $M$, which is possible under \textbf{Assumption} \ref{binsrefine}. We now wish to show that for all $r\in[R]$
\begin{align}\label{eq:tildeconv}
    &(h_{r,M}(y_0)-h_r(y_0))\mathbb{P}_*(X_0=r|\mathcal{G}^0) \nonumber \\ & + \sum_{j=-\infty}^{-1}(h_r(y_j)-h_{r,M}(y_j))(\mathbb{P}_*(X_j=r|\mathcal{G}^0)-\mathbb{P}(X_j=r|\mathcal{G}^{-1}))
\end{align}
vanishes in $L^2$. By Lemma \ref{lemma:scoretails}, we can choose $J$ sufficiently large negative that $$\sum_{j=-\infty}^{J}(h_r(y_j)-h_{r,M}(y_j))(\mathbb{P}_*(X_j=r|\mathcal{G}^0)-\mathbb{P}_*(X_j=r|\mathcal{G}^{-1})) < \frac{\epsilon}{2}.$$
Since $\lVert h_{r,M}\rVert_{L^2}\leq 2\lVert h_r\rVert_{L^2}$, $J$ can be chosen so that the above holds uniformly in $M$.

\vspace{1ex}
To control the finitely many remaining terms, it suffices to note the $L^\infty$ boundedness of the probabilities and the $L^2$ convergence of $h_M$ to $h$, and so for $M$ sufficiently large the remaining finite sum is bounded by $\frac{\epsilon}{2}$, and so (\ref{eq:tildeconv}) has $L^2$ norm at most $\epsilon$. Since $\epsilon$ was arbitrary, we conclude that $\mathcal{P}\subset\mathcal{\tilde{P}}$, and so the two spaces coincide and $\mathcal{A}=\tilde{\mathcal{A}}$ as required.
\end{proof}

\subsection{Admissible partitions}\label{sup:sec:admisspart}
The following lemma is well-known in the identifiability literature. We recall it here for completeness.
\begin{lemma}\label{lemma:partition}
Let $F_1^*,\dots,F_R^*$ be linearly independent distribution functions on $[0,1]$. Then there exists a partition $I_{1},\dots,I_R$ of $[0,1]$ which is admissible in the sense of Definition \ref{admispartition}.
\end{lemma}

\subsection{Technical results for the proofs of Theorems \ref{inversion} and \ref{smoothingcontract}}\label{sup:sec:smoothproof}
We employ the following lemma (Lemma \ref{consistencyemissions} of the main text) in the proof of Theorem \ref{inversion} to eliminate certain terms.

\begin{lemma}\label{sup:consistencyemissions}
Consider a sequence of emissions $\mathbf f^{(n)}$ such that, $g^{(3)}_{Q^*,\mathbf f^{(n)}}\rightarrow g^{(3)}_{Q^*,\mathbf f^*}$ in $L^1(\dd x^3)$. Then $\mathbf f^{(n)}\stackrel{L^1}{\rightarrow}\mathbf f^*$ along a subsequence.
\end{lemma}
\begin{proof}
Write $P=\text{diag}(p^*)$. Since $\|g^{(3)}_{Q^*,\mathbf f^{(n)}}-g^{(3)}_{Q^*,\mathbf f^*}\|_{L^1} = o(1)$, $\|g^{(2)}_{Q^*,\mathbf f^{(n)}}-g^{(2)}_{Q^*,\mathbf f^*}\|_{L^1}=o(1)$ and
$$\sum_{r,s}(PQ^*)_{rs}(f_r^{(n)}(y_1)f_s^n(y_2)-f_{r}^*(y_1)f_{s}^*(y_2)) = o_{L^1}(1).$$
Construct a partition $I_1,\dots I_R$ of $\mathbb R$ such that $F^*_{ir}=\int_{I_i}f_{r}^*\dd y$ is of rank $R$ and define $F^{(n)}_{ir}=\int_{I_i}f_{r}^{(n)}(y)\dd y$. Then we get for all $i$,
$$\sum_{r,s}(PQ^*)_{rs}(F_{ri}^{(n)}f_s^{(n)}(y)-F_{ri}^*f_{s}^*(y)) = o_{L^1}(1).$$
Since $g^{(3)}_{Q^*,\mathbf f^{(n)}}\rightarrow g^{(3)}_{Q^*,\mathbf f^*}$, Theorem 2.3 of \cite{vernet2015posterior} establishes weak convergence of the emissions, which then implies $\mathbf F^n\rightarrow \mathbf F^*$ and so
$$(\mathbf F^*)^T(PQ^*)\mathbf f^{(n)}=(\mathbf F^*)^T(PQ^*)\mathbf f^*+o_{L^1}(1),$$ which implies $\mathbf f^{(n)}=\mathbf f^*+o_{L^1}(1).$
\end{proof}

In the proof of Theorem \ref{smoothingcontract}, we use the fact that the approximating distribution for the posterior $\Pi_1(\cdot|Y_{1:n})$ can be chosen to not depend on the observation $Y_{\ell}$. The difficulty in the proof compared to the iid setting arises from the dependence structure, but by the exponential forgetting (Lemma \ref{lemma:forgetting}) we can prove that the $1/\sqrt{n}$ scaled joint score functions coincide up to $o(1)$ terms, and hence so do the $\sqrt{n}-$scaled MLEs.

\begin{lemma}\label{MLEinfluence}
Grant \textbf{Assumptions} \ref{identifiability}-\ref{binsrefine}, let $\hat{\theta}=(\hat{Q},\hat{\omega})$ denote the MLE in the model with admissible partition $\mathcal{I}_M$ and define $\hat{\theta}^{-\ell}=(\hat{Q}^{-\ell},\hat{\omega}^{-\ell})$ by 
\begin{equation}\label{sup:eq:estimate_no_l}
    \hat{\theta^{-\ell}} = \theta^* + \frac{J_M^{-1}}{\sqrt{n}}l_\theta^{(-\ell)}(Y_{1:n}^{(-\ell)}) ,
\end{equation}
where $Y_{1:n}^{-\ell}=(Y_1,\dots,Y_{\ell-1},Y_{\ell+1},\dots,Y_n)$ and
\begin{equation}\label{sup:eq:conditionalsumlikel}
    l_\theta^{(-\ell)}(Y_{1:n}^{(-\ell)})=\sum_{i=1}^{\ell-1}l_\theta(Y_i|Y_{1:i-1}) + \sum_{i=\ell+1}^nl_\theta(Y_i|Y_{1:i-1}^{(-\ell)}).
\end{equation} Let $\mathcal{L}_n\subset[n]$ be an index set whose cardinality satisfies $\lvert\mathcal{L}_n\rvert \log \lvert\mathcal{L}_n\rvert = o(\sqrt{n})$. Then, for all $\epsilon>0$,
$$\mathbb{P}_*(\exists l\in\mathcal{L}: \hspace{0ex} \sqrt{n}\lvert \hat{Q}^{(-l)}-Q^*\rvert > \epsilon ) \rightarrow 0,$$
as $n\rightarrow\infty$
\end{lemma}
\begin{proof}
We may write the log-likelihood for the parameter $\theta=(Q,\omega)$ as
\begin{equation}\label{sup:eq:conditionalsumlike}
    l_\theta(Y_1,\dots,Y_n)=\sum_{i=1}^nl_\theta(Y_i|Y_{1:i-1}),
\end{equation}
where the $i=1$ term is the unconditional likelihood. Note that \eqref{sup:eq:conditionalsumlikel} corresponds to the log-likelihood based on observing $Y_{1:n}^{(-\ell)}$. Standard arguments for showing asymptotic normality (see e.g. \cite{bickel1998asymptotic}), using Taylor expansions of the derivative of the expression \eqref{sup:eq:conditionalsumlike}, give
$$ \sqrt{n}(\hat{\theta} - \theta^*) = \frac{{J}_M^{-1}}{\sqrt{n}}\nabla_{\theta^*} l_{\theta^*}(Y_{1:n}) + R_n, $$
for $R_n=O_{\mathbb{P}_*}(\sup_{\bar{\theta}\in [\theta^*,\hat{\theta}]}\lVert \frac{1}{n}D^2_{\bar\theta}l_{\bar\theta}(Y_{1:n}) + J_M\rVert )$ the error in approximating the Fisher Information with the negative log-likelihood. The uniform law of large numbers for the observed information as given in Theorem 3 of \cite{Douc2004}, together with the consistency of the MLE (see e.g. \cite{leroux1992maximum}), then implies that $R_n=o_{\mathbb{P}_*}(1)$.

\vspace{1ex}
For the estimator $\hat{\theta}^{(-\ell)}$, we have for each $l\in\mathcal{L}$, by the definition in \eqref{sup:eq:estimate_no_l}, that
$$ \sqrt{n}(\hat{\theta}^{-\ell} - \theta^*) = \frac{{J}_M^{-1}}{\sqrt{n}}\nabla_{\theta^*} l_{\theta^*}^{-\ell}(Y_{1:n}^{-\ell}). $$
Hence
$$\sqrt{n}(\hat{\theta}^{-\ell} -\hat{\theta}) = \frac{{J}_M^{-1}}{\sqrt{n}}(\nabla_{\theta^*} l_{\theta^*}^{-\ell}(Y_{1:n}^{-\ell}) - \nabla_{\theta^*} l_{\theta^*}(Y_{1:n})) - R_n$$
Note that the $R_n$ term does not depend on $\ell$, and vanishes in $\mathbb{P}_*-$probability. We will show that, for each $\ell\in\mathcal{L}_n$,  $\frac{1}{\sqrt{n}}(\nabla_{\theta^*} l_{\theta^*}^{-\ell}(Y_{1:n}^{-\ell}) - \nabla_{\theta^*} l_{\theta^*}(Y_{1:n}))$, so that a union bound implies that
$$ \mathbb{P}_*(\exists l\in\mathcal{L} : \frac{1}{\sqrt{n}}(\nabla_{\theta^*} l_{\theta^*}^{-\ell}(Y_{1:n}^{-\ell}) - \nabla_{\theta^*} l_{\theta^*}(Y_{1:n})) > \epsilon) \leq \lvert \mathcal{L}_n\rvert o_{\mathbb{P}_*}(\lvert\mathcal{L}_n\rvert ^{-1}) = o_{\mathbb{P}_*}(1), $$
from which the result follows. The argument is based on the exponential forgetting properties, and the expansion of the score considered in \cite{Douc2004}. For the full log-likelihood \eqref{sup:eq:conditionalsumlike}, we have through the Fisher identity as in \cite{Douc2004} that the log-likelihood of $\theta$ given $Y_{1:n}$ and $X_0=x$, denoted $l_n(\theta,X_0=x)$, has gradient at $\theta=\theta^*$ given by
$$ \frac{1}{\sqrt{n}}\nabla_{\theta^*} l_n(\theta,X_0=x) = \frac{1}{\sqrt{n}}\sum_{k=1}^n\Delta_{k,0,x},$$
with
\begin{align*}
    \Delta_{k,0,x} &= \mathbb{E}_{\theta^*}\left[\sum_{i=1}^k \nabla_{\theta^*} \log (Q_{X_{i-1},X_i}\omega_{X_i,Y_i} ) | Y_{0:k},X_0=x \right] \\ &\quad- \mathbb{E}_{\theta^*}\left[\sum_{i=1}^{k-1} \nabla_{\theta^*} \log (Q_{X_{i-1},X_i}\omega_{X_i,Y_i} ) | Y_{0:k-1},X_0=x \right],
\end{align*}
which can be seen by writing the log-likelihood as a telescoping sum and using the Fisher identity to write each term as a conditional expectation of the full likelihood. We can do the same for the likelihood without $Y^\ell$: Define
\begin{align*}
    \Delta^{(-\ell)}_{k,0,x} &= \mathbb{E}_{\theta^*}\left[\sum_{i=1}^k \nabla_{\theta^*} \log (Q_{X_{i-1},X_i}\omega_{X_i,Y_i} ) | Y_{0:k}^{-\ell},X_0=x \right] \\ &\quad- \mathbb{E}_{\theta^*}\left[\sum_{i=1}^{k-1} \nabla_{\theta^*} \log (Q_{X_{i-1},X_i}\omega_{X_i,Y_i} ) | Y_{0:k-1}^{(-\ell)},X_0=x \right].
\end{align*}
We note that the term $\Delta^{(-\ell)}_{l,0,x}=0$ which corresponds to the missing contribution from $Y^l$ in the expansion of the likelihood. Then we can write
$$ \frac{1}{\sqrt{n}}\nabla_{\theta^*} l_n^{(-\ell)}(\theta,X_0=x) = \frac{1}{\sqrt{n}}\sum_{k<\ell-1}\Delta^{(-\ell)}_{k,0,x} + \frac{1}{\sqrt{n}}\sum_{k=\ell+1}^{\ell+1+L_n}\Delta^{(-\ell)}_{k,0,x} + \frac{1}{\sqrt{n}}\sum_{\ell>L_n}\Delta^{(-\ell)}_{k,0,x}.$$
The first sum remains unchanged compared to the expression without the missing data point. Lemma 8 of \cite{Douc2004} shows that we may replace the $\Delta_{k,0,x}$ by $\Delta_{k,0}$ where the latter is defined conditioning only on the $Y_{0:k}$ - the same argument made there shows the analogous result for the $\Delta_{k,0,x}^{(-\ell)}$ with respect to the analogously defined $\Delta_{k,0}^{(-\ell)}$. We note that, under \textbf{Assumptions} \ref{ass:Qbound} and \ref{densityratio}, the $\Delta_{k,0}$ and $\Delta_{k,0}^{-\ell}$ are bounded uniformly in $X_{(\cdot)},Y_{(\cdot)}$, and so for any $L_n$ to be chosen later, we have that
$$ \frac{1}{\sqrt{n}}\sum_{k=\ell+1}^{\ell+1+L_n}\Delta^{(-\ell)}_{k,0}(\theta) - \frac{1}{\sqrt{n}}\sum_{k=\ell}^{\ell+1+L_n}\Delta^{}_{k,0}(\theta) = O_{\mathbb{P}_*}(L_n/\sqrt{n}). $$
It remains to control the difference between the remaining contributions for suitably chosen $L_n$. Similarly to what is done in \cite{Douc2004}, we define $\Delta_{k,m}$ by conditioning instead on $Y_{m:k}$. Write $\Delta_{k,m}$ as a telescoping sum as
\begin{align*}
    \Delta_{k,m} &= \mathbb{E}_{\theta^*}[ \nabla_{\theta^*} \log (Q_{X_{k-1},X_k}\omega_{X_k,Y_k} ) | Y_{m:k} ] \\ &+ \sum_{i=m+1}^{k-1}\mathbb{E}_{\theta^*}[ \nabla_{\theta^*} \log (Q_{X_{i-1},X_i}\omega_{X_i,Y_i} ) | Y_{m:k} ]  - \mathbb{E}_{\theta^*}[ \nabla_{\theta^*} \log (Q_{X_{i-1},X_i}\omega_{X_i,Y_i} ) | Y_{m:k-1} ].
\end{align*}
Then their equation (19) states that, for $k>m>m^\prime$\footnote{Their result is only stated for $m<0$, but by stationarity extends to all $m\in\mathbb{Z}$}
$$ \lVert \mathbb{E}_{\theta^*}[ \nabla_{\theta^*} \log (Q_{X_{i-1},X_i}\omega_{X_i,Y_i} ) | Y_{m:k} ] - \mathbb{E}_{\theta^*}[ \nabla_{\theta^*} \log (Q_{X_{i-1},X_i}\omega_{X_i,Y_i} ) | Y_{m^\prime:k} ]\rVert \leq 2C \rho_*^{i-m-1},$$
where $\rho^*\in(0,1)$ is a constant which depends on $\theta^*$ alone. Furthermore, their equation (20) states that
$$ \lVert \mathbb{E}_{\theta^*}[ \nabla_{\theta^*} \log (Q_{X_{i-1},X_i}\omega_{X_i,Y_i} ) | Y_{m:k} ] - \mathbb{E}_{\theta^*}[ \nabla_{\theta^*} \log (Q_{X_{i-1},X_i}\omega_{X_i,Y_i} ) | Y_{m:k-1} ]\rVert \leq 2C \rho_*^{k-i-1}.$$
Combining the above two displays in the same way as Lemma 10 of \cite{Douc2004}, we have for a suitable constant $C(Q^*,\omega^*)$ that
$$ \lVert \Delta_{k,0} - \Delta_{k,\ell+1}\rVert_{L^2(\mathbb{P}_*)} \leq C(Q^*,\omega^*)\dfrac{\rho_*^{(k-\ell-1)/2-1}}{1-\rho_*}.$$
Arguing similarly for the case of $\Delta_{k,0}^{-\ell}$, we see that 

$$ \lVert \Delta_{k,0}^{-\ell} - \Delta^{-\ell}_{k,\ell+1}\rVert_{L^2(\mathbb{P}_*)} \leq C(Q^*,\omega^*)\dfrac{\rho_*^{(k-\ell-1)/2-1}}{1-\rho_*}.$$
The $\Delta^{-\ell}_{k,\ell+1}=\Delta_{k,\ell+1}$ as the conditioning excludes $Y_\ell$ in either case, and so the above displays combine to give
$$\lVert \Delta_{k,0}^{-\ell} - \Delta_{k,0}\rVert_{L^2(\mathbb{P}_*)} \leq 2 C(Q^*,\omega^*)\dfrac{\rho_*^{(k-\ell-1)/2-1}}{1-\rho_*}.$$
Applying Minkowski's inequality then gives that
$$\left\lVert\sum_{k>\ell+1+L_n} (\Delta_{k,0}^{-\ell} - \Delta_{k,0})\right\rVert_{L^2(\mathbb{P}_*)} = O(\rho_*^{L_n/2}).$$
Choosing $L_n = 2M \log 
\lvert \mathcal{L}_n \rvert$ for sufficiently large $M>0$ we have $\rho_*^{L_n/2}=e^{(\log \lvert\mathcal{L}_n\rvert)(M\log\rho_*)}= \lvert\mathcal{L}_n\rvert ^{M\log\rho_*} = o\left(\frac{1}{\lvert\mathcal{L}_n\rvert}\right)$. Under the assumption of the lemma, we have $\frac{\log\lvert\mathcal{L}_n\rvert}{\sqrt{n}}=o(1/\lvert\mathcal{L}_n\rvert)$ and so the claim is proved.
\end{proof}

\section{Assumptions required for the application of Theorem \ref{L1contractemission} to Dirichlet process mixtures}\label{sup:sec:dirmix}

In this section, we detail the assumptions required for Proposition \ref{DPMNormal}.
\vspace{1ex}
We will need the following assumptions on the behaviour of the emissions, which are assumptions (T1)-(T3) of \cite{vernet2015conc}.

\begin{assume}\label{emissionconditionDPM}
        \textbf{A.} There exist positive constants $M_0,\tau_0,\gamma_0$ such that, for all $1\leq i\leq R$ and all $y\in\mathbb{R}$,
        $$f_i^*(y) \leq M_0\exp(-\tau_0\lvert y\rvert^{\gamma_0}). $$
        
        \vspace{1ex}
        \textbf{B.} For all $1\leq i,j\leq R$ there exist constants $T_{i,j},M_{i,j},\tau_{i,j},\gamma_{i,j}<\gamma_0$ such that
        $$f_i^*(y) \leq f^*_j(y)M_{i,j}\exp(\tau_{i,j}\lvert y\rvert^{\gamma_{i,j}}),\hspace{3ex}\lvert y\rvert \geq T_{i,j} .$$
        
        \vspace{1ex}
        \textbf{C.} For all $1\leq i\leq R$, $f_i^*$ is positive and there exists $c_i>0$ and $y_i^{low} < y_i^{high}$ such that $f_I^*$ is non-decreasing on $(-\infty,y_i^{low})$, bounded below by $c>0$ on $(y_i^{low},y_i^{high})$, and non-increasing on $(y^{high},\infty)$.

\end{assume}
We will further require the following assumptions on the choice of prior, which are Assumptions (G1) and (S1)-(S3) of \cite{vernet2015conc}. These assumptions are verified by standard choices of Gaussian base measure and inverse gamma prior on the standard deviation.
\begin{assume}\label{priorcondition}
    \textbf{A.} $\alpha(\mathbb{R}\setminus [-y,y])\lesssim \exp(-C_1y^{a_1})$ for $y$ large enough and some $C_1,a_1>0$
    
    \vspace{1ex}
    \textbf{B.} $\sigma \sim \Pi_{\sigma}$ where $\Pi_\sigma $ is an Inverse Gamma distribution.
\end{assume}
Condition \textbf{B} on $\Pi_\sigma$ can be relaxed and we essentially need the density $\pi_\sigma$ to behave like an inverse - Gamma near 0 and have tails bounded by some power  of $1/\sigma$ near infinity, see for instance \cite{kruijer2010adaptive}.

\section{Key results from literature}\label{sup:sec:keyresults}
In this section, we document some of the key results from other works which are used in our contributions. Section \ref{sup:sec:forgetfulness} features the forgetfulness properties of \cite{Douc2004}, while Section \ref{sec:suppconv} expands upon the presentation of efficiency given in Section \ref{subsec:tangentspace}, recalling the relevant framework of \cite{mcneney2000application}.

\subsection{Forgetting of the hidden Markov chain}\label{sup:sec:forgetfulness}
The following results are from \cite{Douc2004}. Recall that $\bar{q}=\min_{ij}Q^*_{ij}>0$ under Assumption \ref{ass:Qbound}. The first result quantifies the exponential forgetting of the hidden chain.

\begin{lemma}[Douc et al. 2004, \cite{Douc2004}]\label{lemma:forgetting}
For all $s\in[R]$, $k,k^\prime\in\mathbb{Z}$, and $j\geq k$, we have for $\rho=\frac{1-2\bar{q}}{1-\bar{q}}$
\begin{equation*}
    \sup_{Y_{k:k^\prime}}\left\lVert \mathbb{P}_*(X_j = \cdot| X_k = s, Y_{k:k^\prime}) - \mathbb{P}_*(X_j = \cdot|Y_{k:k^\prime}) \right\rVert_{TV} \leq \rho^{j - k}. 
\end{equation*}
\end{lemma}
The second of these results follows from applying Equation (20) of \cite{Douc2004} to both the known-emission model and the submodels on the emissions. It in particular ensures that the score functions we consider are well-defined.
\begin{lemma}[Douc et al. 2004, \cite{Douc2004}]\label{lemma:scoretails}
For $h\in L^2(\mathbb{P}_*)$, $j\leq 0$, $r,s\in[R]$, and for $\mathbb{P}_*-$almost all $Y_{-\infty:0}$, and $\rho=\frac{1-2\bar{q}}{1-\bar{q}}$,
$$\lvert\mathbb{P}_*(X_{j-1}=s,X_j=r|Y_{-\infty:0}) - \mathbb{P}_*(X_{j-1}=s,X_j=r|Y_{-\infty:-1})  \rvert \leq \rho^{-j-1} $$
and
$$\lvert h(Y_j)\{\mathbb{P}_*(X_j=r|Y_{-\infty:0}) - \mathbb{P}_*(X_j=r|Y_{-\infty:-1})  \} \rvert \leq \lvert h(Y_j)\rvert\rho^{-j-1}.$$
In particular, by integrating the square of both sides, we have
$$\lVert h(Y_j)\{\mathbb{P}_*(X_j=r|Y_{-\infty:0}) - \mathbb{P}_*(X_j=r|Y_{-\infty:-1})\}  \rVert_{L^2(\mathbb{P}_*)} \leq \lVert h\rVert_{L^2(\mathbb{P}_*)}\rho^{-j-1}.$$
\end{lemma}

Lemmas \ref{lemma:forgetting} and \ref{lemma:scoretails} remain true when replacing $Y$ by $Y^{(M)}$, the coarsened data. The choice of $\rho$ is made independently of $M$ as it only depends on the transition matrix $Q^*$, which is the same in all histogram models as it is in the semiparametric model.

\subsection{Convolution Theorem and Efficiency}\label{sec:suppconv}

The following section pertains to Section \ref{subsec:tangentspace}, in which our model is considered in the framework of \cite{mcneney2000application} to understand the relevant notions of LAN expansion and efficiency. The following sequence of definitions, and the convolution theorem which follows, are as in Section 2 of \cite{mcneney2000application}. In this section, we also give further background to the exposition in Section \ref{subsec:tangentspace}. We do not present the material in its full generality, but rather to the extent required for our application, to which we will apply the general framework.

\vspace{1ex}

Let $\Theta$ denote the parameter space and $\mathcal{P}=\{P_\theta:\theta\in\Theta\}$ be a family of probability measures governing the law of $(Y_1,Y_2,\cdots)$ and write $P_{\theta,n}$ for the marginal of $Y_{1:n}=(Y_1,\dots,Y_n)$. For $\theta_1,\theta_2\in\Theta$, write $\Lambda_n(\theta_1,\theta_2)=\log \dfrac{\dd P_{n,\theta_1}}{\dd P_{n,\theta_2}} = \ell_n(\theta_1) - \ell_n(\theta_2)$ for the log-likelihood ratio between $\theta_1$ and $\theta_2$ based on $Y_{1:n}$.

\vspace{1ex}
The first step is to exhibit a LAN expansion, formally defined below, for the semiparametric model we consider. The parameter space is $\Theta =\mathcal{Q}\times\mathcal{F}^R$ and perturbations from the true parameter $\theta=(Q^*,\mathbf{f}^*)$ can be written in the form $(Q,\mathbf{f})=(Q,f_1,\dots,f_R)=(Q^*+a^TQ, f_1^*(1+h_1), \dots f_R^*(1+h_R) )$ for appropriate $a\in\mathbb{R}^{R\times (R-1)}$ and $h_r=\frac{f_r}{f_r^*} - 1$, which has finite $L^\infty(f_r^*)$ norm by positivity of $f_r^*$ and satisfies $\int h_r(y)f_r^*(y)\dd(y) = 0 $. This motivates the choice of tangent space at the parameter $(Q^*,\mathbf{f}^*)$ as $$\mathcal{H}=\{(a,\mathbf{h}), a\in\mathbb{R}^{R\times (R-1)}, \quad \mathbf{h}=(h_r)_{r=1}^R, \quad h_r\in\mathcal{H}_r\}$$ where $\mathcal{H}_r=\mathcal{H}_r(f_r^*) = \{h_r\in L^2(f_r^*) : \int h_rf_r^* = 1, \lVert h_r\rVert_{\infty} <\infty \}$.

\vspace{1ex}
The following definition is a specialisation to our context of the corresponding definition found in \cite{mcneney2000application}. We write $$\theta_n((a,\mathbf{h}))=\left(Q^*+\frac{a}{\sqrt{n}}, f_1^*(1+\frac{h_1}{\sqrt{n}}),\dots, f_R^*(1+\frac{h_R}{\sqrt{n}})\right ).$$

\begin{defn}[Local Asymptotic Normality]
Let $\mathcal{H}$ be as above, with inner product $\langle\cdot,\cdot\rangle_{\mathcal{H}}$ and induced norm $\lVert\cdot\rVert_{\mathcal{H}}$. We say the model is locally asymptotically normal (LAN) at $\theta_0\in\Theta$, indexed by the tangent space $\mathcal{H}$, if for each $(a,\mathbf{h})\in\mathcal{H}$, there exists a sequence $\{P_{n,\theta_n(a,\mathbf{h})}\}$ of probability measures with
$$\Lambda_n(\theta_n(a,\mathbf{h}),\theta_0) = \Delta_{n,a,\mathbf{h}} - \frac{1}{2}\lVert (a,\mathbf{h})\rVert_{\mathcal{H}}^2 + o_{\mathbb{P}_{*}}(1), $$
with $\Delta_{n,a,\mathbf{h}}$ real-valued random variables, which are linear functions of $(a,\mathbf{h})$, such that $\Delta_{n,a,\mathbf{h}}\rightarrow N(0, \lVert (a,\mathbf{h})\rVert_{\mathcal{H}}^2)$ for all $(a,\mathbf{h})\in\mathcal{H}$.

\end{defn}

In essence, the $\Delta_{n,a,\mathbf{h}}$ are score functions in the `direction' $h$, since $\Lambda_n$ is the local change in the log-likelihood when we perturb along the path characterised by $h$ and $\Delta_{n,a,\mathbf{h}}$ is the first order term.



\vspace{1ex}
To show the model is LAN with tangent space $\mathcal{H}$, we seek to identify the score functions $\Delta_{n,a,\mathbf{h}}$ note that the log-likelihood ratio $\Lambda(\theta_n(a,\mathbf{h}),\theta_0)$ can be approximated by considering the score in the submodel $$\Theta_{a, \mathbf{h}} = \{ \theta_t = \left(Q^* + t a , ( f_r^*( 1 + t h_r) : r \in [R])\right), \, \lvert t \rvert \leq  \epsilon \}, $$
which, for fixed $a,\mathbf{h}$ and fixed $\epsilon>0$ sufficiently small, is a one-dimensional submodel $\Theta_{a, \mathbf{h}}\subset \Theta$. From \cite{Douc2004}, in which the authors consider score functions in parametric HMMs, we can expand the gradient of the log-likelihood $\ell_n^{(a,\mathbf{h})}(t)$ in the (parametric) model $\Theta_{a,\mathbf{h}}$, at $t_0=0$ as
$$\frac{1}{\sqrt{n}}\nabla_t\ell_n^{(a,\mathbf{h})}(t_0) = \frac{1}{\sqrt{n}}\sum_{k=1}^n\Delta_{k,\infty}^{(a,\mathbf{h})} + o_{L^2(\mathbb{P}_*)}(1)$$
where,
\begin{align}\label{eq:submodscore}
    \Delta_{k,\infty}^{(a,\mathbf{h})} = \Delta_{k,\infty}^{(a,\mathbf{h})}(Y_{-\infty:k} ) &= \mathbb{E}_*[\phi(X_{k-1},X_k,Y_k)|Y_{-\infty}^k] \nonumber \\ &+ \sum_{i=-\infty}^{k-1}\left( \mathbb{E}_*[\phi(X_{i-1},X_i,Y_i)|Y_{-\infty}^k] - \mathbb{E}_*[\phi(X_{i-1},X_i,Y_i)|Y_{-\infty}^{k-1}] \right),
\end{align}
and where, writing $a_{rR}=-\sum_sa_{rs}$ and $Q_{r,R}=1-\sum_sQ_{rs}$, \begin{equation}\label{eq:phiscore}
    \phi(X_{k-1},X_k,Y_k) = \sum_{r=1}^R\sum_{s=1}^R 1\{X_{k-1}=r,X_k=s\}\frac{a_{r,s}}{Q^*_{r,s}} + \sum_{s=1}^R 1_{X_i=s} h_s(Y_i).
\end{equation} These formulae arise as an application of the results of Section 6.1 of \cite{Douc2004} to $\Theta_{a,\mathbf{h}}$.

\vspace{1ex}
We note that the contribution of the first term of the right hand side of \eqref{eq:phiscore}, which may be rewritten as $$ \sum_{r=1}^R\sum_{s=1}^{R-1} \left(\frac{1\{X_{k-1}=r,X_k=s\}}{Q^*_{rs}}-\frac{1\{X_{k-1}=r,X_k=R\}}{Q^*_{rR}}\right)a_{rs},$$to the expression defined in \eqref{eq:submodscore}, is precisely the score function for estimation in the model $\Theta_{a,\mathbf{0}}$ in which the emission densities are fixed and known. We see that this is equal to $a^TS_{Q^*}(Y_{-\infty:k})$, for $S_{Q^*}$ the score at $Q^*$ in the $R\times(R-1)$-dimensional parametric model with known emissions and unknown $Q$, which takes the form

\begin{equation}\label{ScoreQ}
\begin{split} 
S_{Q^*}(r, s) &(Y^{}_{-\infty : k}) = 
 \dfrac{\mathbb{P}_*({X_{k-1}=r},{X_k=s}|Y^{}_{-\infty : k}) }{Q^*_{rs}}-\dfrac{\mathbb{P}_*({X_{k-1}=r},{X_k=R}|Y^{}_{-\infty : k})}{Q^*_{rR}} \\
    &+\sum_{i=-\infty}^{k-1}\Bigg\{ \dfrac{\mathbb{P}_*({X_{i-1}=r},{X_i=s}|Y^{}_{-\infty : k})
    -\mathbb{P}_*({X_{i-1}=r},{X_i=s}|Y^{}_{-\infty : k-1})}{Q^*_{rs}}
    \\
    &\hspace{6ex}-\dfrac{\mathbb{P}_*({X_{i-1}=r},{X_i=R}|Y^{}_{-\infty : k})-\mathbb{P}_*({X_{i-1}=r},{X_i=R}|Y^{}_{-\infty : k-1})}{Q^*_{rR}}\Bigg\},
    \end{split}
\end{equation}
where we index with $(r,s)\in[R]\times[R-1]$ for convenience sake, but consider $S_{Q^*}$ as a vector of length $R(R-1)$. We then rewrite \eqref{eq:submodscore}, substituting also the expression in \eqref{eq:phiscore} as
\begin{align}\label{eq:npscores}
    \Delta_{k,\infty}^{(a,\mathbf{h})} = a^TS_{Q^*} +  \sum_{r=1}^R \Bigg\{& \mathbb{P}_*(X_k=r|Y_{-\infty}^k)h_r(Y_i) \nonumber \\ &+ \sum_{i=-\infty}^{k-1} (\mathbb{P}_*(X_i=r|Y_{-\infty}^k) -  \mathbb{P}_*(X_i=r|Y_{-\infty}^{k-1}))h_r(Y_i) \Bigg\}.
\end{align}
Henceforth we will write $H_r(h_r)(Y_{-\infty:k})$ for the function of $Y$ defined inside the curly brackets. In our notation, we suppress the dependence on $k$ because the law of $H_r(h_r)(Y_{-\infty:k})$ coincides with that of $H_r(h_r)(Y_{-\infty:0})$ under a suitable re-indexing, by stationarity.  A Taylor expansion of the log-likelihood in $\Theta_{a,\mathbf{h}}$ then gives us that, for some $t_1\in[0,n^{-1/2}]$,
\begin{align*}
    \ell_n(\theta_n((a,\mathbf{h}))) - \ell_n(\theta_*) &= \frac{1}{\sqrt{n}}\sum_{k=1}^\infty \Delta_{k,\infty}^{(a,\mathbf{h})} - \frac{1}{2}\frac{1}{n}\left(-\dfrac{\dd^2}{\dd t^2}\rvert_{t=t_1}\ell_n^{(a,\mathbf{h})}(t) \right) \quad  \\
    & = \Delta_{n,(a,\mathbf{h})} - \frac{1}{2}\lVert (a,\mathbf{h}) \rVert_{\mathcal{H}}^2 + o_{\mathbb{P}_*}(1).
\end{align*}
where $\lVert (a,\mathbf{h}) \rVert_{\mathcal{H}}^2$ the Fisher information at $t=0$ in the model $\Theta_{a,\mathbf{h}}$, which is defined in \cite{Douc2004} as the $L^2=L^2(\mathbb{P}_*)$ norm\footnote{The $L^2$ norms of all $\Delta^{(a,\mathbf{h})}_{k,\infty}$ coincide by stationarity.} of $\Delta^{(a,\mathbf{h})}_{0,\infty}$, and where $\Delta_{n,(a,\mathbf{h})} = \frac{1}{\sqrt{n}}\sum_{k=1}^\infty \Delta_{k,\infty}^{(a,\mathbf{h})} $, which is linear in $(a,\mathbf{h})$ and satisfies $\Delta_{n,(a,\mathbf{h})}\rightarrow N(0, \lVert (a,\mathbf{h}) \rVert_{\mathcal{H}}^2)$, by the discussion directly preceding Theorem 2 in \cite{Douc2004}.  We also used above the local uniform convergence of the second derivative of the score to the Fisher information matrix at $t=0$ in $\Theta_{a,\mathbf{h}}$, which is guaranteed by Theorem 3 of \cite{Douc2004}.

\vspace{2ex}
The preceding discussion shows that our model is LAN, which is the first step in understanding efficient estimators. The following definition is adapted from \cite{mcneney2000application} for our specific case.

\begin{defn}[Differentiability of parameter]
Let $v_n(P_{n,\theta})$ be $\mathbb{R}^p-$valued parameters for some $p>0$. We say the sequence $(v_n)$ is \textit{differentiable} if
$$ \sqrt{n}( v_n(P_{n,\theta_n\uh )}) - v_n(P_{n,0}) ) \rightarrow \Dot{v}\uh  \quad \forall h \in \mathcal{H} $$
for some continuous linear map $\Dot{v}: \mathcal{H} \longrightarrow \mathbb{R}^p$.
\end{defn}

In our case, the parameter of interest will be $v_n(P_{n,\theta_n\uh})=Q$, so that $p=R(R-1)$. We then obtain $\dot{v}\uh= a \in \mathbb{R}^p$.

\vspace{1ex}
We finally recall the definition of a regular estimator.

\begin{defn}[Regular estimators]
A sequence of maps $T_n = T_n(Y_1,\dots,Y_n) \in \R^p$ is said to be \textit{locally regular} for $v_n$ if, under $P_{n,\theta_n\uh}$, we have
$$ \sqrt{n}(T_n - v_n(P_{n,\theta_n\uh}) \Rightarrow \mathbb{Z} $$
as $n\rightarrow\infty$ for every $\uh\in\mathcal{H}$, where $\mathbb{Z}$ is a Borel measurable tight random element in $\mathbb{R}^p$ which does not depend on $\uh\in\mathcal{H}$.

\end{defn}

\vspace{1ex}
Following what precedes, we are ready to state the convolution theorem also described in \cite{mcneney2000application}, and originally proven in \cite{van1991differentiable}. As with what precedes, we slightly simplify the statement to make clear the application to our context.

\begin{theorem}[Convolution theorem]
Suppose $\mathcal{P}=\{P_\theta : \theta\in\Theta\}$ is LAN at a point $\theta_*$ indexed by a linear subspace $(\mathcal{H}, \langle\cdot,\cdot\rangle )$ of a Hilbert space. Further suppose $(v_n)$ is a differentiable sequence of parameters. Then if $T_n$ is locally regular for $v_n$, there exist tight Borel measurable elements $\mathbb{Z}_0$ and $\mathbb{W}$ in $\mathbb{R}^p$ with
\begin{align*}
    (A)\hspace{3ex} & P(\mathbb{Z}_0 \in \overline{ \dot{v}(\mathcal{H}) } )= 1 \\
    (B)\hspace{3ex}& \mathcal{L}(\mathbb{Z}) = \mathcal{L} ( \mathbb{Z}_0 + \mathbb{W} ) \\
    (C)\hspace{3ex} & \mathbb{Z}_0 \text{ and } \mathbb{W} \text{ are independent} \\
    (D)\hspace{3ex} & \mathcal{L}(b^T \mathbb{Z}_0 ) = N(0, \lVert \Dot{v}^T_{b}\rVert^2 ) \text{ for every } b\in \R^p
\end{align*}
Here $\Dot{v}^T_{b}$ is the unique element of $\mathcal{H}$ such that
$$ (\Dot{v}^T (b) ) \uh = \langle \Dot{v}^T_{b} , \uh \rangle_{\mathcal{H}}  $$
where $\Dot{v}^T : (\mathbb{R}^p)^\dag \rightarrow\mathcal{H}^\dag $ is the adjoint of $\Dot{v}$, and where the superscript ${\cdot}^{\dag}$ indicates the dual space.
\end{theorem}
The convolution theorem essentially states that the limiting law of a regular estimator is the convolution of the laws of a Gaussian (in this case $\mathbb{Z}_0$) and some independent random variable (in this case $\mathbb{W}$), and so in particular the variance is lower bounded by that of $\mathbb{Z}_0$. It also describes the variance of $\mathbb{Z}_0$ in terms of the tangent space, its dual space, and the parameter `derivative' $\dot{v}$. In our case, we recall that $\dot{v}(a,h)=a$, and $p=R(R-1)$. 

\vspace{1ex}
We note that we identify the dual of $\mathbb{R}^p$ with itself, through the bijection $b\leftrightarrow \langle b, \cdot \rangle_{\mathbb{R}^p}$, so $\dot{v}^T(b)$ is a shorthand for $\dot{v}^T(\langle b, \cdot \rangle_{\mathbb{R}^p})$. The adjoint $\dot{v}^T$ of $\dot{v}$ then satisfies, by definition, for $\uh\in \mathcal{H}$ and $b\in \R^p$,
\begin{equation}\label{eq:vadjoint}
    \dot{v}^T(b)\uh = b^T\dot{v}\uh = \langle b, a\rangle_{\mathbb{R}^{p}}.
\end{equation}
To find $\dot{v}^T_{b}$ , we need a matching between the usual inner product $\langle b, a\rangle_{\mathbb{R}^{p}}$ with the inner product $\langle \dot{v}^T_b,\uh\rangle_{\mathcal{H}}$, for all $\uh\in \mathcal{H}$. By considering the elements for which $a=0$, we see that $\dot{v}^T_{b}$ is in the orthogonal complement of the linear span of the scores in the models $\Theta_{0,\mathbf{h}}$, which is the span of the $H_r(h_r)$ as we vary $r$ and $h_r\in\mathcal{H}_r$. Write $\mathcal{A}$ for the projection onto this space, and write
$$\tilde{S}_Q = S_Q - \mathcal{A}S_Q  $$
for the projection of the score function $S_Q$ onto the orthogonal complement of this space. We recall from the form of \eqref{eq:npscores} that the inner product on $\mathcal{H}$ takes the form, for $\bar{h}_1=(a_1,\mathbf{h}_1)$ and $\bar{h}_2=(a_2,\mathbf{h}_2)$
\begin{align}\label{eq:orthoginnerprod}
    \langle\bar{h}_1,\bar{h}_2\rangle_{\mathcal{H}} & = \langle a_1^TS_{Q^*} + \sum_{r=1}^RH_r(h_{1r}) , a_2^TS_{Q^*} + \sum_{r=1}^RH_r(h_{2r}) \rangle_{L^2}\nonumber \\
    & = \langle a_1^T\tilde{S}_{Q^*} + a_1^T\mathcal{A}S_{Q^*}+ \sum_{r=1}^RH_r(h_{1r}) , a_2^T\tilde{S}_{Q^*} + a_2^T\mathcal{A}S_{Q^*} + \sum_{r=1}^RH_r(h_{2r}) \rangle_{L^2} \nonumber \\
    & = \langle a_1^T\tilde{S}_{Q^*}, a_2^T\tilde{S}_{Q_*} \rangle_{L^2} + \langle a_1^T\mathcal{A}S_{Q_*}+ \sum_{r=1}^RH_r(h_{1r}) ,  a_2^T\mathcal{A}S_{Q_*} + \sum_{r=1}^RH_r(h_{2r}) \rangle_{L^2}.
\end{align}
Write $\dot{v}^T_{b} = (a_{b}, \mathbf{h}_{b})$. Then \eqref{eq:orthoginnerprod} gives that
\begin{align*}
    \langle \dot{v}^T_{b} , \uh \rangle_{\mathcal{H}} &= \langle a_{b}^T\tilde{S}_{Q^*}, a^T\tilde{S}_{Q^*} \rangle_{L^2}  \\ &+ \langle a_{b}^T\mathcal{A}S_{Q^*}+ \sum_{r=1}^RH_r(h_{b,r}) ,  a^T\mathcal{A}S_{Q^*} + \sum_{r=1}^RH_r(h_{r}) \rangle_{L^2}  .
\end{align*}
We must then choose $a_{b}$ and $\mathbf{h}_{b}$ to match with \eqref{eq:vadjoint} as required in the statement of the theorem, at least insofar as to determine the efficient variance $\langle \dot{v}^T_{b}, \dot{v}^T_{b} \rangle_{\mathcal{H}}$. By considering inner products with elements in $\mathcal{H}$ of the form $(0,\mathbf{h})$, we see that $\dot{v}^T_{b}$ must be in the orthogonal complement of the space spanned by the $H_r(h_r)$ considered earlier. By choosing $a=b$, we see that 
$$ \langle a^T_{b} \tilde{S}_{Q^*} , b^T \tilde{S}_{Q^*} \rangle_{L^2} = a^T_{b}\tilde{J} b, $$
where $\tilde{J} = \mathbb{E}_*[S_{Q^*}S_{Q^*}^T]$ is the efficient Fisher information matrix. To get a matching between the above display and the usual $L^2$ inner product appearing on the right hand side of \eqref{eq:vadjoint}, we must choose $a_{b} = \tilde{J}^{-1}b$. This implies (using also the orthogonality established prior, and symmetry of $\tilde{J}$) that the efficient variance $$\langle \dot{v}^T_{b}, \dot{v}^T_{b} \rangle_{\mathcal{H}} = b^T\tilde{J}^{-1}\tilde{J}(\tilde{J}^{-1}b) = b^T\tilde{J}^{-1}b.$$
By varying $b$, this characterises the covariance of $\mathbb{Z}_0$ as $\mathrm{Cov}(\mathbb{Z}_0)=\tilde{J}^{-1}$, the inverse efficient information matrix.

\section{Details on MCMC algorithms used in Section \ref{simulation}}\label{extrasims}

\subsection{MCMC Algorithm for $\Pi_1$}
In this section we describe the algorithm (Algorithm \hyperref[AlgoS1]{SA1}) used for the simulation study in Section \ref{simulation}. $S_R$ denotes the symmetric group of order $R$ and  where we  use subscripts to refer to algorithm iterations. We use the RHmm package \cite{RHmmpackage} for the forward-backward algorithm in the simulation of the hidden states and computation of the log likelihood. We also use the gtools package \cite{gtoolspackage} for generation of permutations and Dirichlet draws.

\begin{algorithm}[h!]\label{AlgoS1}
    \SetAlgoRefName{SA1}
    \SetKwInOut{Input}{Input}
    \SetKwInOut{Output}{Output}
    \Input{Binned data $Y^{(M)}\in[\kappa_M]^n$, number of hidden states $R$, prior parameters $\gamma$,$\beta$, iterations $I$, initial hidden states value $\mathbf{X}_{init}$}
    \Output{List of $I$ draws of $Q$}
    Initialise $\mathbf{X}_1=\mathbf{X}_{init}$ \\
    \For{$i=1,\dots,I+b$}{Draw transition matrix $Q_i\sim P(Q|\mathbf{X}_i)$, histogram weights $\underline{\omega}_i\sim P(\underline{\omega}|\mathbf{X}_i,Y)$, hidden states $\mathbf{X}_{i+1}\sim P(\mathbf{X}|Y,Q_i,\underline{\omega}_i)$}
    Compute $i_{MAP}=\argmax_{i}\log\Pi_1(Q_i,\underline{\omega}_i|Y)$
    
    \For{$i=1,\dots,I$}{$\tau_i=\argmin_{\tau\in S_R} d\left(({}^{(\tau)}Q_i,{}^{(\tau)}\underline{\omega}_i),(Q_{i_{MAP}},\underline{\omega}_{i_{MAP}})\right) $
    }
    
    \Return{$\left\{{}^{(\tau_i)}Q_i\right\}_{i=1}^I$} \\
    \caption{Algorithm for MCMC draws targeting posterior of $Q$.}
\end{algorithm}

\subsection{MCMC Algorithm for $\Pi_2$}
In Algorithm \hyperref[AlgoS2]{SA2}, we detail the MCMC procedure we implement in R. Once again, we make use of the forward-backward algorithm to simulate latent states, allowing us to exploit the simple structure of the full likelihood. One key difference is that our algorithm requires allocation to a bivariate latent space (with total number of states $RS_{\max}$) involving both the HMM hidden state and the Dirichlet mixture component. Since the forward backward algorithm is $O(NK^2)$ with $K$ the number of states and $N$ the number of samples, and since we take $S_{\max}=\lfloor\sqrt{N}\rfloor$, the overall implementation is $O(N^2)$.

\subsection{MCMC Algorithm for $\Pi^\prime$}
The MCMC algorithm of $\Pi^\prime$ is obtained by modifying Algorithm \hyperref[AlgoS2]{SA2} to remove the interior loop over $c=1,\dots,C$ and sampling a transition matrix $Q$ from the conditional distribution given the latent states, as in Algorithm \hyperref[AlgoS1]{SA1}, at the beginning of each outer loop over $i$. As with Algorithm \hyperref[AlgoS2]{SA2}, the complexity is $O(N^2)$ and so using it to target the marginal posterior on $Q$ is much slower than using $\Pi_1$, as discussed in Section \ref{simulation}.

\begin{algorithm}[h!]\label{AlgoS2}
    \SetAlgoRefName{SA2}
    \SetKwInOut{Input}{Input}
    \SetKwInOut{Output}{Output}
    \Input{Data $Y\in\mathbb{R}^N$, Thinned transition matrix list $(Q_i)_{i=1,\dots,\tilde{I}}$, number of hidden states $R$, prior parameters $M_0,\alpha_\sigma,\beta_\sigma,\mu_c,\sigma_c^2$, interior iterations $C$, initial hidden state value $\mathbf{X}_{init}$, truncation level $S_{max}$.}
    \Output{List of $\tilde{I}$ draws of $\{f_r\}_{r=1}^R$}
    Initialise $\mathbf{X}_0=\mathbf{X}_{init}$; variances $v^{(r)}\stackrel{iid}{\sim}\text{InvGamma}(\alpha_\sigma,\beta_\sigma)$; weights $W^{(r)}=(W^{(r)}_{1},\dots,W^{(r)}_{S_{max}})\sim\textrm{Dir}(\alpha)$ with $\alpha=\alpha(M_0,S_{max})$ as in (\ref{Dirdraws}); allocation variables $S_0=(S_{01},\dots,S_{0n})$ as in (\ref{DPMstick}) \\
    \For{$i=1,\dots,\tilde{I}$}{\For{$c=1,\dots,C$}{ Sample locations $(\mu_{i1}^{c,(r)},\dots,\mu_{iS_{max}}^{c,(r)})_{r=1,\dots,R}$ from $P(\mu_i^c|S_i^{c-1},\mathbf{X}_i^{c-1},Y)$ \\ Sample variances $(v_i^{c,(r)})_{r=1,\dots, R}$ from $P(v_i^c|S_i^{c-1},\mathbf{X}_i^{c-1},\mu_i^{c},Y)$ \\
    Sample weights $(W_{i,1}^{c,(r)},\dots,W_{i,S_{max}}^{c,(r)})_{r=1,\dots,R}$ from $P(W_{i}^c|S_i^{c-1},\mathbf{X}_i^{c-1})$ \\ Sample latents $(S_i^c,X_i^c)$ from $P(S_{i}^c,\mathbf{X}_i^c|\mu_i^c,v_i^c,W_i^c,Y,Q_i)$ \\ }Store $(\mu_i,v_i,W_i,S_i)=(\mu_i^C,v_i^C,W_i^C,S_i^C)$}
    \For{$i=1,\dots,\tilde{I}$}{Compute $f_i^{(r)}$ from $(\mu_i,v_i,W_i,S_i)$ by evaluating (an approximation of) (\ref{DPM})
    }
    
    \Return{$\left\{f_i^{(r)}\right\}_{i=1}^{\tilde{I}}$} \\
    \caption{Algorithm for MCMC draws targeting cut posterior of emissions. Within the loop over $c$, we adopt the convention $W_{i}^0=W_{i-1}$ and $W_{i-1}^0=W_0$ etc.}
\end{algorithm}

\FloatBarrier

\end{document}